%% file: main.tex
\documentclass[11pt,leqno]{article}
\input{sections/packages}

\input{sections/commands}

\title{Hamilton-Jacobi equations on graphs with applications to semi-supervised learning and data depth\thanks{{\bf Source Code:} \url{https://github.com/jwcalder/peikonal}}}

\author{Jeff Calder\thanks{School of Mathematics, University of Minnesota. \href{mailto:jcalder@umn.edu}{jcalder@umn.edu}} 
\and 
Mahmood Ettehad\thanks{Institute for Mathematics and its Applications (IMA), University of Minnesota. \href{mailto:mettehad@umn.edu}{etteh001@umn.edu}}}

\begin{document}
\maketitle

\input{sections/abstract}

%\begin{keywords}
{\bf Keywords:}  Data depth, Graph learning, Hamilton-Jacobi equation, Robust statistics, Semi-supervised learning, viscosity solutions, discrete to continuum limits, partial differential equations
%\end{keywords}

\input{sections/intro}

\input{sections/discrete}

\input{sections/continuum}

\input{sections/convergence}

\input{sections/analysis}

\input{sections/numerics}

\input{sections/conclusion}

\section{Acknowledgments}

The authors thank the Institute for Mathematics and its Applications (IMA), where part of this work was conducted.
JC acknowledges funding from NSF grant DMS:1944925, the Alfred P. Sloan foundation, and a McKnight Presidential Fellowship.

\bibliographystyle{abbrv}
\bibliography{ref}

\appendix
\input{sections/appendix}

\end{document}

%% file: sections/packages.tex
\usepackage[margin=1in]{geometry}
\usepackage[T1]{fontenc}
\usepackage{enumerate}
\usepackage[leqno]{amsmath}%
\usepackage{amsthm}
\usepackage{amsxtra}%
\usepackage{amsfonts}%
\usepackage{amssymb}%
\usepackage{color,hyperref}
\usepackage{graphicx}
\usepackage{tikz}
\usepackage{subfig}
\usepackage{cite}
\hypersetup{colorlinks,breaklinks,
	linkcolor=blue,urlcolor=blue,
	anchorcolor=blue,citecolor=blue}
\usepackage{color}
\usepackage{array}
\usepackage{booktabs}
\usepackage{dsfont}
\usepackage[toc,page]{appendix}
\usepackage{booktabs}
\usepackage[font=small]{caption}
\usepackage{microtype}
\usepackage{nicefrac}
\usepackage[T1]{fontenc}
\usepackage{lmodern}
\usepackage[export]{adjustbox}

%% file: sections/commands.tex
\newcommand{\one}{\mathds{1}}

\newcommand{\A}{{\mathcal A}}

\renewcommand{\bar}[1]{{\overline{#1}}}
\newcommand{\F}{{\mathcal F}}
\renewcommand{\P}{{\mathbb P}}
\newcommand{\X}{X}

\newcommand{\R}{\mathbb{R}}
\newcommand{\E}{\mathbb{E}}

\newcommand{\eps}{\varepsilon}
\renewcommand{\epsilon}{\varepsilon}
\renewcommand{\tilde}[1]{\widetilde{#1}}
\renewcommand{\phi}{\varphi}

\DeclareMathOperator*{\argmin}{argmin}

\DeclareMathOperator{\diam}{diam}
\DeclareMathOperator{\Lip}{Lip}
\DeclareMathOperator{\dist}{dist}

\def\XXint#1#2#3{{\setbox0=\hbox{$#1{#2#3}{\int}$ }
		\vcenter{\hbox{$#2#3$ }}\kern-.6\wd0}}

\newcommand{\cA}{\mathcal{A}}

\newcommand{\cH}{\mathcal{H}}

\newcommand{\cO}{\mathcal{O}}

\newcommand{\cS}{\mathcal{S}}

\newcommand{\cX}{X}

\newcommand{\bP}{\mathbb{P}}

\newcommand{\gx}{\nabla_\X}

\newcommand{\lx}{F(\X)}

\newtheorem{theorem}{Theorem}
\newtheorem{lemma}[theorem]{Lemma}

\newtheorem{proposition}[theorem]{Proposition}
\theoremstyle{definition}

\newtheorem{remark}[theorem]{Remark}
\newtheorem{definition}[theorem]{Definition}

\numberwithin{equation}{section}
\numberwithin{theorem}{section}

\graphicspath{{./figures/}}

%% file: sections/abstract.tex
\begin{abstract}
Shortest path graph distances are widely used in data science and machine learning, since they can approximate the underlying geodesic distance on the data manifold. However, the shortest path distance is highly sensitive to the addition of corrupted edges in the graph, either through noise or an adversarial perturbation. In this paper we study a family of Hamilton-Jacobi equations on graphs that we call the $p$-eikonal equation. We show that the $p$-eikonal equation with $p=1$ is a provably robust distance-type function on a graph, and the $p\to \infty$ limit recovers shortest path distances. While the $p$-eikonal equation does not correspond to a shortest-path graph distance, we nonetheless show that the continuum limit of the $p$-eikonal equation on a random geometric graph recovers a geodesic density weighted distance in the continuum. We consider applications of the $p$-eikonal equation to data depth and semi-supervised learning, and use the continuum limit to prove asymptotic consistency results for both applications. Finally, we show the results of experiments with data depth and semi-supervised learning on real image datasets, including MNIST, FashionMNIST and CIFAR-10, which show that the $p$-eikonal equation offers significantly better results compared to shortest path distances.
\end{abstract}

%% file: sections/intro.tex
\section{Introduction}

Shortest path distances on graphs have found applications in many areas of data science and machine learning, including dimensionality reduction (e.g., the ISOMAP algorithm \cite{tenenbaum2000global}), semi-supervised learning on graphs \cite{moscovichfast,chapelle2005semi,bijral2012semi,rozza2014novel,yang2021spagan}, graph classification \cite{borgwardt2005shortest}, and data depth \cite{molina2021tukey,molina2022eikonal,calder2021boundary}.  In many applications, the shortest paths are density weighted, to make path lengths shorter in high density regions of the graph, and longer in sparse regions \cite{bijral2012semi}. Shortest path algorithms offer different information compared to second order methods based on graph Laplacians, like spectral clustering \cite{ng2002spectral},  Laplacian eigenmaps \cite{belkin2003laplacian}, diffusion maps \cite{coifman2006diffusion}, or Laplacian based semi-supervised learning \cite{zhu2003semi,calder2020poisson}, which offer information about average or typical paths through graphs.

However, a main drawback of shortest path distances is their lack of robustness to perturbations in graph structure. The addition of a single edge can have a strong effect on the shortest path, while simultaneously having little or no effect on the average or typical path, which gives an intuitive reason for the apparent superiority of graph Laplacian based methods for semi-supervised learning and dimension reduction, among other problems.

In this paper, we approach the problem of robustly computing distance functions on graphs from the viewpoint of Hamilton-Jacobi equations. We study a family of Hamilton-Jacobi equations on graphs, which we call the $p$-eikonal equations, that are provably robust to graph perturbations, especially for $p=1$. The equations have the form
\begin{equation}\label{eq:peikonal_intro}
\sum_{j=1}^n w_{ji}(u(x_i) - u(x_j))_+^p = f(x_i),
\end{equation}
where $a_+=\max\{a,0\}$, and $w_{ij}$ is the weight between nodes $i$ and $j$ in the graph.  We prove that as $p\to \infty$, these $p$-eikonal equations recover shortest path graph distances, while for $p=1$ the solutions provide information that is different from shortest paths and far more robust to graph perturbations.  The solution of the $p$-eikonal equation can be computed in similar time to shortest path distances, using a slight variation on the fast marching method \cite{sethian1996fast}. 

While the $p$-eikonal equations do not describe shortest path graph distances, we prove rigorously that the continuum limit of the $p$-eikonal equation, as the number of data points tends to infinity while $p$ is fixed, is exactly a density weighted geodesic distance function on the underlying space (either a Euclidean domain or data manifold). Hence, the $p$-eikonal equation offers a robust estimation of geodesic distances in the continuum for any finite value of $p$. Our techniques for proving discrete to continuum convergence are quite different from existing spectral convergence results for graph Laplacians (see, e.g., \cite{calder2020Lip,calder2019improved,garcia2020error}). We use the viscosity solution machinery and the maximum principle, as in \cite{yuan2020continuum,flores2022algorithms}. Our theory is also quite different from previous work on continuum limits for shortest path distances (see, e.g., \cite{alamgir2012shortest,bungert2021uniform,hwang2016shortest,calder2021boundary}) which crucially use the shortest path interpretation on the graph. 

To illustrate the robustness of the $p$-eikonal equation, we consider applications of density weighted graph distances to data depth and semi-supervised learning. For data depth we use an approach similar to geometric medians on Riemannian manifolds \cite{fletcher2009geometric}.  For semi-supervised learning we use a nearest neighbor classifier via the $p$-eikonal distance. In both applications we consider \emph{density-weighted} distances, for which path lengths are shorter in high density regions of the graph and longer in sparse regions. This improves accuracy in semi-supervised learning and encourages the median to be placed in a high density region of the graph in data depth problems, making the methods more robust to outliers. We test the methods on both toy and real datasets, including semi-supervised learning on MNIST, FashionMNIST and CIFAR-10. The classification results for the $p$-eikonal equation are uniformly better than shortest path graph distances, which we attribute to the robustness properties of the $p$-eikonal equation to spurious corrupted edges in real world graphs.

Using our continuum limit results, we go on to prove that $p$-eikonal based data depth and semi-supervised learning are asymptotically consistent. In particular, for semi-supervised learning, we take a clusterability assumption for the data and show that $p$-eikonal semi-supervised learning with arbitrarily few labels can recover the true labels for each cluster. The proofs of asymptotic consistency are particularly simple for graph distances, compared to the analogous results for graph Laplacian based techniques (see, e.g., \cite{hoffmann2022spectral}).  We also examine the role of class priors in semi-supervised learning, and show how utilizing information about the relative sizes of each class improves the asymptotic consistency results by allowing a weaker clusterability assumption. We enforce class priors by using a weighted minimum in the label decision, as was done in the volume label projection in \cite{calder2020poisson}.

There is a considerably amount of related work in both data depth and semi-supervsied learning. The problem of data depth, and in general, the ordering of multivariate data, is a common problem in statistics \cite{barnett1976ordering,liu1999multivariate}.  The Tukey halfspace depth \cite{tukey1975mathematics} is one of the oldest and most well-studied notions of depths, and it has been extended to graphs \cite{small1997multidimensional} and metric spaces \cite{carrizosa1996characterization}. The Tukey depth has been connected, at the continuum population level, to the solution of a non-standard Hamilton-Jacobi equation \cite{molina2021tukey}. Other interesting notions of data depth include the Monge-Kantorovich depth \cite{chernozhukov2017monge}, and notions of depth for curves \cite{de2020depth}. Another way to define data depth is by repeatedly peeling away extremal points. Several related algorithms, including convex hull peeling, nondominated sorting, and Pareto envelope peeling, have been recently connected to viscosity solutions of partial differential equations (PDEs) in the continuum limit \cite{calder2020convex,calder2014,calder2015pde,calder2015a,calder2015numerical,bou2021hamilton,cook2020nondom}. 

We were recently made aware of another paper \cite{molina2022eikonal} that was developed in parallel with ours, and proposes to use the eikonal equation for data depth. The method in \cite{molina2022eikonal} requires identifying boundary points first, and then the depth is defined as the length of a shortest density-weighted path back to the boundary. This approach, without density weighting, was also used in \cite{calder2021boundary}, in combination with a method for detecting boundary points. Our approach to data depth based on the geometric median framework is much different than these works, and in particular, it does not require the identification of boundary points to compute depth. We also mention that the $p$-eikonal equation \eqref{eq:peikonal_intro} has been used previously for image segmentation and data classification on graphs \cite{desquesnes2013eikonal,desquesnes2017nonmonotonic}. Our work provides a rigorous foundation for these applications.

Finally, let us mention that the problem of semi-supervised learning at low label rates has received a significant amount of attention recently, since it was pointed out in \cite{nadler2009semi} that Laplace learning (or label propagation)  \cite{zhu2003semi} is ill-posed with very few labels. Many graph-based semi-supervised learning algorithms have been proposed recently at low label rates, including higher-order Laplacians \cite{zhou2011semi}, $p$-Laplacian methods \cite{el2016asymptotic,kyng2015algorithms,slepcev2019analysis,calder2018game,calder2019lip,flores2022algorithms}, reweighted Laplacians \cite{shi2017weighted,calder2019properly}, the centered-kernel method \cite{mai2018random,mai2018randomJMLR}, volume constrained MBO \cite{jacobs2018auction}, and Poisson learning \cite{calder2020poisson}, and the low label rate issue has been studied theoretically in \cite{calder2020rates}. The only methods that are provably well-posed at arbitrarily low label rates are the $p$-Laplacian methods \cite{calder2018game,slepcev2019analysis} for $p>d$\footnote{Here, $d$ is the intrinsic dimension of the data.} and the Properly Weighted Laplacian \cite{calder2019properly}, but neither has been shown to be asymptotically consistent at low label rates. In contrast, our results show that $p$-eikonal based semi-supervised learning gives well-posed, stable and informative classification results, and is asymptotically consistent, at arbitrarily low label rates and for any $p\geq 1$.

\subsection{Outline}

This paper is organized as follows. In Section \ref{sec:discrete} we study Hamilton-Jacobi equations on graphs, and introduce the $p$-eikonal equation. We establish our main robustness result, and then consider applications to data depth and semi-supervised learning. In Section \ref{sec:continuum_eikonal} we introduce the continuum geodesic distances, and review the connection to state constrained eikonal equations. In Section \ref{sec:convergence} we prove our main discrete to continuum convergence result, showing that the $p$-eikonal equations recover geodesic density weighted distances in the continuum limit, for any value of $p\geq 1$. In Section \ref{sec:analysis} we use the continuum limit theory to study the asymptotic consistency of data depth and semi-supervised learning with the $p$-eikonal equation. Finally, in Section \ref{sec:numerics} we show the results of experiments with real data.

%% file: sections/discrete.tex
\section{First order equations on graphs}
\label{sec:discrete}

In this section we study the general theory of first order equations on graphs (Section \ref{sec:general}), and then review the graph distance function (Section \ref{sec:graph_distance}) and then introduce our $p$-eikonal equation in Section \ref{sec:peikonal}, where we discuss robustness properties and computational complexity. Then in Section \ref{sec:discrete_app} we consider applications of the graph $p$-eikonal equation to data depth and semi-supervised learning. We give some toy examples in Section \ref{sec:discrete_app}, and postpone experiments with real data to Section \ref{sec:numerics}. 

Let us first introduce some notation. Let $G=(\X,W)$ be a weighted graph with vertices $\X=\{x_1,\dots,x_n\}\subset \R^d$ and nonnegative edge weights $W=(w_{ij})_{i,j=1}^n$. The edge weights encode similarity between data points, with $w_{ij}\gg 0$ indicating $x_i$ and $x_j$ are similar, and $w_{ij}\approx 0$ indicating dissimilarity. We do \emph{not} assume the weight matrix is symmetric, so in general we have $w_{ij}\neq w_{ji}$. This includes graphs such as $k$-nearest neighbor graphs. For first order equations, symmetry is not a main concern, since we do not require any operators to be self-adjoint, as in the case of graph Laplacians. Any zero edge weight $w_{ij}=0$ indicates the absence of an edge from $i$ to $j$. 
%Since our graphs may not be symmetric, there are two choices for the degree of each node; the incoming and outgoing degrees. We work with the incoming degree
%\[\d(x_i) = \sum_{j=1}^n w_{ji}.\]
  We also let $\lx$ denote the vector space of functions $u:\X\to \R$, and let $I_n = \{1,\dots,n\}$ denote the indices of the graph vertices.    For a function $u\in \lx$ and a vertex $x_i \in \X$, we define the gradient $\gx u(x_i)\in \R^n$ by
\begin{equation}\label{eq:gradient}
	\gx u(x_i) = (u(x_i)-u(x_1),u(x_i)-u(x_2),\dots,u(x_i)-u(x_n)).
\end{equation}
For convenience, we will write $\gx^j u(x_i)=u(x_i)-u(x_j)$, so that
\[\gx u(x_i) = (\gx^1 u(x_i),\gx^2 u(x_i),\dots,\gx^n u(x_i)).\]
Finally, throughout this section, we let $K$ denote the unweighted maximum incoming degree of the graph, that is 
\begin{equation}\label{eq:unweighted_degree}
K = \max_{1\leq i\leq n}\sum_{j=1}^n \one_{w_{ji}>0}.
\end{equation}

%We also define the $\infty$-norm of functions on the graph, which is given by
%\[\|u\|_\infty = \max_{x_i\in \X}|u(x_i)|,\]
%for $u\in \lx$. 

\subsection{General theory}
\label{sec:general}

We begin by developing a general theory for first order equations on graphs, and give general  existence and uniqueness results. Letting $\Gamma\subset \X$ denote a set of boundary or terminal nodes, a general graph PDE has the form
\begin{equation}\label{eq:genPDE}
	\left\{\begin{aligned}
		H(\gx u(x_i),u(x_i),x_i) &= 0,&&\text{if }x_i\in \X\setminus \Gamma\\ 
		u(x_i) &=g(x_i),&&\text{if }x_i\in \Gamma,
	\end{aligned}\right.
\end{equation}
where $g:\Gamma\to \R$ are some prescribed boundary values. The Hamiltonian $H$ is a function
\begin{equation}\label{eq:hamiltonian}
	H:\R^n\times \R \times \X \to \R,
\end{equation}
that also implicitly depends on the weight matrix $W$, which encodes the graph structure. It is also possible to pose a graph PDE on all of $\X$ with no boundary conditions, in the form
\begin{equation}\label{eq:genPDEnoBC}
	H(\gx u(x_i),u(x_i),x_i) = 0 \ \ \text{ for all }x_i\in \X.
\end{equation}
We will write $H=H(q,z,x_i)$ in general, for $q\in \R^n, z\in \R,x_i\in \X$. While we will focus on first order equations (in the sense that their continuum limits are first order PDEs), we note that this formulation of graph PDEs is very general, and contains as a subset the graph Laplacian by setting
\begin{equation}\label{eq:graphLap}
H(q,z,x_i) = \sum_{j=1}^nw_{ij}q_j.
\end{equation}

In this section, we establish existence and uniqueness of solutions to the graph PDE \eqref{eq:genPDE}. Some of this analysis is similar to previous work studying PDEs on graphs, see for instance \cite{manfredi2015nonlinear,calder2018game,calder2019lip}. Our arguments are slightly different, and cover more general cases. 

Existence and uniqueness of solutions to \eqref{eq:genPDE} is based on a comparison principle, which allows us to compare the values of a subsolution $u$ to a supersolution $v$, based on comparing their values on the boundary $\Gamma$. A subsolution $u\in \lx$ of \eqref{eq:genPDE} satisfies
\begin{equation}\label{eq:subsol}
	H(\nabla_\X u(x_i),u(x_i),x_i) \leq 0 \ \ \ \text{ for all }x_i\in \X\setminus \Gamma,
\end{equation}
while a supersolution $v\in \lx$ of \eqref{eq:genPDE} satisfies
\begin{equation}\label{eq:supersol}
	H(\nabla_\X v(x_i),v(x_i),x_i) \geq 0 \ \ \ \text{ for all }x_i\in \X\setminus \Gamma.
\end{equation}
Throughout this section, $\Gamma\subset \X$ is fixed, and may be empty.
\begin{definition}\label{def:comp}
	We say that $H$ \emph{admits comparison} if for all $u\in \lx$ satisfying \eqref{eq:subsol} and $v\in \lx$ satisfying \eqref{eq:supersol}, if $u\leq v$ on $\Gamma$ then $u\leq v$ on $\X$.
\end{definition}

In this section, we establish conditions under which $H$ admits comparison. An important class of PDEs are those which are \emph{monotone}. For vectors $p,q\in \R^n$, we write $p\leq q$ if $p_i\leq q_i$ for all $i$.
\begin{definition}\label{def:monotone}
	We say $H$ is \emph{monotone} if 
	\begin{equation}\label{eq:monotone}
		p\leq q \text{ and }s \leq t \implies H(p,s,x) \leq H(q,t,x)
	\end{equation}
	for all $x\in \X$.
\end{definition}
This definition of monotonicity is related to upwind discretizations of Hamilton-Jacobi equations, and monotone discetizations of second order equations \cite{sethian1996fast,oberman2006convergent}. As an example, the graph Laplacian \eqref{eq:graphLap} is clearly monotone, since $w_{ij}\geq 0$.

Monotonicity allows us to apply maximum principle arguments to prove a comparison principle, which is based on the following observation.
\begin{proposition}\label{prop:monotone}
	Assume $H$ is monotone and let $u,v\in \lx$. If $u-v$ attains its maximum over $\X$ at $x_i\in \X$ and $u(x_i)\geq v(x_i)$, then 
	\[H(\nabla_\X u(x_i),u(x_i),x_i) \geq  H(\nabla_\X v(x_i),v(x_i),x_i).\]
\end{proposition}
\begin{proof}
	We simply note that $u(x_j) - v(x_j) \leq u(x_i) - v(x_i)$ for all $j$, which implies that
	\[u(x_i) - u(x_j) \geq v(x_i) - v(x_j) \ \ \text{ for all }j,\]
	and so $\nabla_\X u(x_i) \geq \nabla_\X v(x_i)$.  The result now follows from monotonicity of $H$
\end{proof}

We can immediately prove a comparison principle when $H$ is monotone, and one of the sub or supersolutions is strict.
\begin{theorem}\label{thm:strict_comp}
	Assume $H$ is monotone. Let $u,v\in \lx$ such that
	\begin{equation}\label{eq:strict_sub}
		H(\nabla_\X u(x_i),u(x_i),x_i)  < H(\nabla_\X v(x_i),v(x_i),x_i) \ \ \ \text{ for all }x_i\in \X\setminus \Gamma,
	\end{equation}
	and $u\leq v$ on $\Gamma$. Then $u\leq v$ on $\X$.
\end{theorem}
\begin{proof}
	Let $x_i\in \X$ be a point at which $u-v$ attains its maximum over $\X$. If $x_i\in \X\setminus \Gamma$, then by Proposition \ref{prop:monotone} and the assumption \eqref{eq:strict_sub}, we find that $u(x_i) < v(x_i)$. If $x_i\in \Gamma$, then $u(x_i)\leq v(x_i)$ by assumption, which completes the proof.
\end{proof}
The comparison principle in Theorem \ref{thm:strict_comp} requires that $u$ be a strict subsolution relative to $v$. The strategy to prove a true comparison principle (i.e., without the strictness, as in Definition \ref{def:comp}) will be to make small perturbations of subsolutions (or supersolutions) to obtain the strictness required in Theorem \ref{thm:strict_comp}. This requires that we place further assumptions on $H$.

\begin{definition}\label{def:strictmonotone}
	We say $H$ is \emph{proper} if there exists a strictly increasing function $\gamma:[0,\infty) \to [0,\infty)$ with $\gamma(0)=0$ such that when $t\geq s$ we have
	\begin{equation}\label{eq:proper}
		H(q,t,x) \geq H(q,s,x) + \gamma(t-s)
	\end{equation}
	for all $x\in \X$ and $q\in \R^n$.
\end{definition}
An example of an equation that is proper is one with a positive zeroth order term, of the form
\[H(q,z,x) = \lambda z + G(q,x),\]
where $\lambda>0$ and $G:\R^n\times \X\to \R$. In this case, $\gamma(t) = \lambda t$.

We now establish several situations where comparison holds.

\begin{lemma}\label{lem:comparison}
Assume $H$ is monotone. Then $H$ admits comparison if any of the following hold.
\begin{enumerate}[{\rm (i)}]
\item $H$ is proper.
\item $H=H(q,x)$, $q\mapsto H(q,x)$ is convex, and there exists $\phi\in \lx$ and $\lambda>0$ such that
\begin{equation}\label{eq:strictphi}
H(\nabla_\X \varphi(x_i), x_i) + \lambda \leq 0 \ \  \text{ for all } \ \ x_i\in \X\setminus \Gamma.
\end{equation}
\item $H(q,z,x)=G(q) - f(x)$, where $f>0$ on $\X$, and $G$ is positively $p$-homogeneous for $p>0$. 
\end{enumerate}
\end{lemma}
\begin{proof}
Let $u$ satisfy \eqref{eq:subsol} and $v$ satisfy \eqref{eq:supersol}, and assume that $u\leq v$ on $\Gamma$. In each case we will produce a perturbation $u_\epsilon$ of $u$ satisfying $u_\epsilon\leq v$ on $\Gamma$, $H(\nabla_\X u_\epsilon,u,x)< 0$, and $u_\epsilon\to u$ as $\epsilon\to 0$. Then by Theorem \ref{thm:strict_comp} we have $u_\epsilon \leq v$ and sending $\epsilon\to 0$ completes the proof.

\hspace{1.2mm}(i) We set $u_\epsilon = u-\epsilon$ and use the fact that $H$ is proper to get the strict subsolution condition. 

(ii) We set $u_\eps = (1-\eps) u + \eps \phi$. We can shift $\phi$ by a constant, if necessary, so that $\phi-u \leq 0$, and so $u_\epsilon\leq u$. Since $q\mapsto H(q,z,x)$ is convex, we have
\begin{align*}
H(\nabla_\X u_\eps(x_i),x_i)&=H((1-\eps)\nabla_\X u + \eps\nabla_\X\varphi,x_i)\\
&\leq (1-\eps)H(\nabla_\X u(x_i),x_i) + \eps H(\nabla_\X \varphi(x_i),x_i) \leq -\lambda \epsilon,
\end{align*}
for all $x_i\in \X\setminus \Gamma$. 

(iii) Define $u_\eps = (1-\eps) u + \epsilon\min_{\X} u$. Then $u_\epsilon\leq u$.  Since $G$ is positively $p$-homogeneous we have $G(aq)=|a|^pG(q)$ for all $a\in \R$ and $q\in \R^n$, and so 
\[G(\nabla_\X u_\epsilon(x_i)) = G((1-\epsilon)\nabla_\X u(x_i)) = (1-\epsilon)G(\nabla_\X u(x_i)) \leq (1-\epsilon)f(x_i).\]
Hence, we have
\[G(\nabla_\X u_\epsilon(x_i)) - f(x_i) \leq -\epsilon f(x_i) < 0.\qedhere\]
\end{proof}

If $H$ admits comparison, then we can prove existence of a solution to \eqref{eq:genPDE} using the Perron method. We summarize this in the following result.
\begin{theorem}\label{thm:existence}
Assume $H$ is monotone, continuous in $p$ and $z$, and admits comparison. Assume there exists $\phi,\psi\in \lx$ such that $\psi\geq \phi= g$ on $\Gamma$ and for $x_i\in \X\setminus \Gamma$
\[H(\nabla_\X \phi(x_i),\phi(x_i),x_i) \leq 0 \ \ \text{ and }\ \ H(\nabla_\X \psi(x_i),\psi(x_i),x_i) \geq 0.\]
Then there exists a unique solution $u\in \lx$ of \eqref{eq:genPDE} and $\phi \leq u \leq \psi$.
\end{theorem}
The proof of Theorem \ref{thm:existence} is very similar to existing results (e.g., Theorem 4 of \cite{calder2018game}). We include the proof in Appendix \ref{sec:proofs} for reference.
\begin{remark}
Notice that none of the results in this section have required graph connectivity, which is a common assumption in the analysis of PDEs on graphs. Normally, graph connectivity is used in a \emph{path to the boundary} argument to establish a comparison principle (see, e.g., \cite{manfredi2015nonlinear,calder2018game,calder2019lip}). Our arguments do not require graph connectivity to establish comparison. The one place connectivity requirements may appear is in the construction of the super and subsolutions $\phi$ and $\psi$ in the Perron method  in Theorem \ref{thm:existence}.
\label{rem:connected}
\end{remark}

\subsection{Graph distance functions}
\label{sec:graph_distance}

The graph distance $d_G:\X\times \X\to R$ is defined by 
\begin{equation}\label{eq:graph_distance}
d_G(x_i,x_j) = \min_{m\geq 1}\min_{\tau\in I_n^m} \left\{w_{i,\tau_1}^{-1} + \sum_{i=1}^{m-1} w_{\tau_i,\tau_{i+1}}^{-1}+  w_{\tau_m,j}^{-1}\right\}.
\end{equation}
We use the interpretation that $w_{ij}^{-1}=\infty$ whenever $w_{ij}=0$, which implicitly restricts the feasible paths to follow edges in the graph and to connect $x_i$ to $x_j$. 
\begin{definition}\label{def:connected}
We say that the graph $G$ is connected if $d_G(x_i,x_j) < \infty$ for all $x_i,x_j \in \X$.
\end{definition}
We also define the graph distance to a set $\Gamma\subset \X$ as follows
\[d_G(x_i,\Gamma) = \min_{x_j \in \Gamma}d_G(x_i,x_j).\]
We recall that the graph distance function satisfies a certain graph eikonal equation. The result is well-known (see, e.g., Lemma 3 of \cite{bungert2021uniform}), but usually stated for symmetric graphs, so we will sketch a proof for completeness.
\begin{lemma}[{\cite{bungert2021uniform}}]\label{lem:graph_eikonal}
Assume $G$ is connected and let $\Gamma\subset \X$. Then the graph distance function $u(x):=d_G(x,\Gamma)$ is the unique solution of the graph eikonal equation. 
\begin{equation}\label{eq:graph_eikonal}
\max_{x_j\in \X} w_{ji}(u(x_i) - u(x_j)) = 1 \ \ \text{for all } x_i\in \X\setminus \Gamma,
\end{equation}
satisfying $u(x_i)=0$ for $x_i\in \Gamma$.
\end{lemma}
\begin{remark}\label{rem:eikonal_monotone}
We call \eqref{eq:graph_eikonal} the \emph{graph eikonal equation}, since its solution is a distance function, in the same way that the continuum eikonal equation (see Section \ref{sec:continuum_eikonal}) represents continuum path distances. In the notation of Section \ref{sec:general}, the graph eikonal equation corresponds to the \emph{monotone} Hamiltonian $H(q) = \max_{1\leq j\leq n}w_{ji}q_j - 1$. 

In terms of computational complexity, the solution of \eqref{eq:graph_eikonal} can be computed with Dijkstra's algorithm in $\cO(nK\log(n))$ time, where we recall $K$ is the maximum (unweighted) degree of any node in the graph, defined in \eqref{eq:unweighted_degree}.
\end{remark}
\begin{proof}[Proof of Lemma \ref{lem:graph_eikonal}]
The main idea of the proof is to use the fact that $u$ satisfies the dynamic programming principle
\begin{equation}\label{eq:dpp}
u(x_i) = \min_{x_j\in \X}(u(x_j) + w_{ji}^{-1}).
\end{equation}
Since the graph is connected, there exists some $j$ with $w_{ji}>0$, and both $u(x_j)$ and $u(x_i)$ are finite. We can rearrange this to obtain
\[\max_{x_j \in \X} (u(x_i) - u(x_j) - w_{ji}^{-1}) = 0.\]
Since the max is zero, we can multiply by $w_{ji}$ inside the brackets above and rearrange to obtain the result. To prove uniqueness, we can run the proof in the opposite direction, showing that any solution of \eqref{eq:graph_eikonal} satisfies the dynamic programming principle \eqref{eq:dpp}, and is thus the graph distance function $d_G(\cdot,\Gamma)$.
\end{proof}

It is common to consider density weighted distances in data science and machine learning applications. This allows us to make it more expensive for paths to travel through sparse regions in space, and less expensive to travel within dense regions. This makes points within clusters closer together, while driving points in different clusters further apart, which is useful for cluster and semi-supervised learning. 

In the context of the graph eikonal equation \eqref{eq:graph_eikonal}, density weighting can be introduced by solving the equation with a right hand side, of the form
\begin{equation}\label{eq:density_eikonal}
\max_{x_j\in \X} w_{ji}(u(x_i) - u(x_j)) = f(x_i) \ \ \text{for all } x_i\in \X\setminus \Gamma.
\end{equation}
One can choose, for example, $f(x_i)=\hat{\rho}(x_i)^{-\alpha}$ for $\alpha\geq 0$, where $\hat{\rho}:\X\to \R$ is any density estimator (say, a kernel density estimator or a $k$-nearest neighbor estimator), and $\alpha$ is a tunable parameter. Since we did not assume the graph was connected in Lemma \ref{lem:graph_eikonal}, we can apply the lemma to \eqref{eq:density_eikonal} with the graph weights $\bar{w}_{ij} = f(x_j)^{-1}w_{ij}$ to obtain that the solution of \eqref{eq:density_eikonal} subject to $u(x_i)=0$ for $x_i\in \Gamma$ corresponds to the density weighted graph distance
\begin{equation}\label{eq:density_graph_distance}
d_{G,f}(x_i,x_j) := \min_{m\geq 1}\min_{\tau\in I_n^m} \left\{w_{i,\tau_1}^{-1}f(x_{\tau_1}) + \sum_{i=1}^{m-1} w_{\tau_i,\tau_{i+1}}^{-1}f(x_{\tau_{i+1}})+  w_{\tau_m,j}^{-1}f(x_{\tau_j})\right\}.
\end{equation}
When $f=\hat{\rho}^{-\alpha}$ with $\alpha\geq 0$, the reweighted equation \eqref{eq:density_eikonal} makes it more expensive for paths to travel through regions where the density, $\hat{\rho}$, is low, and less expensive where the density is high. Of course, choosing $\alpha\leq 0$ has the opposite effect.

\subsubsection{Sensitivity to noise}

We mention that the graph eikonal equation \eqref{eq:density_eikonal} is highly sensitive to corruption in the weight matrix $W$ used to construct the graph. Indeed, we can see this quite easily from the distance function interpretation, since adding a single spurious edge between two distant nodes in a graph creates a \emph{short-cut} that drastically changes the distance function. Thus, while the graph eikonal equation \eqref{eq:density_eikonal} does indeed approximate geodesic distances on the underlying data manifold well (see, e.g., \cite{hwang2016shortest}), the graph distance lacks robustness to noise and other corruptions. To illustrate this, we refer to Figure \ref{fig:robust_eikonal}, which shows how drastically the graph distance function can change with the addition of a few spurious edges in the graph. The graph is a simple unweighted proximity graph on $n=20000$ uniformly distributed random variables on the unit ball. Points within distance $\epsilon=0.05$ are connected by an edge with edge weight of $1$, and the boundary set $\Gamma$ is chosen to be all points within distance $\epsilon$ of the boundary of the ball. From left to right in Figure \ref{fig:robust_eikonal}, we add 0, 10, 20, and 50 corrupted edges at random, and show the resulting distance functions to the boundary.  

\subsection{The p-eikonal equation}
\label{sec:peikonal}

\begin{figure}[!t]
\centering
\subfloat[Graph distance function with corrupted edges]{
\includegraphics[width=0.24\textwidth]{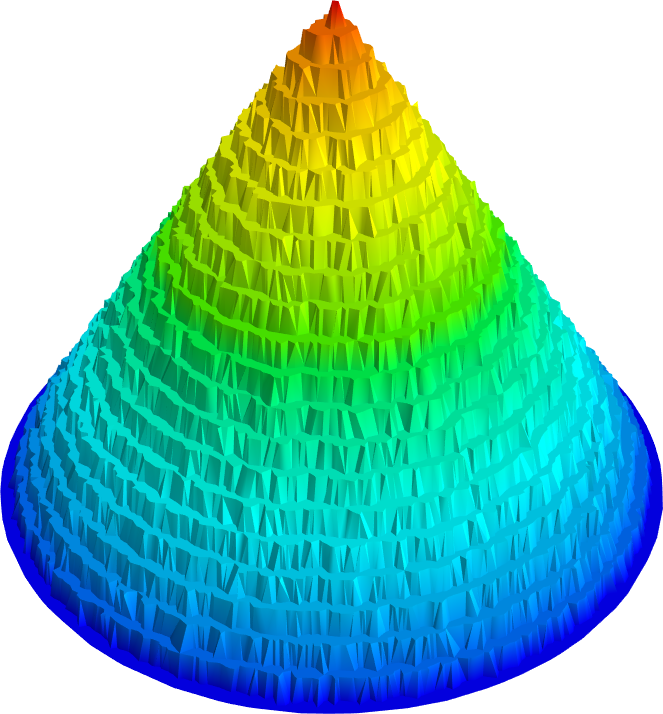}
\includegraphics[width=0.24\textwidth]{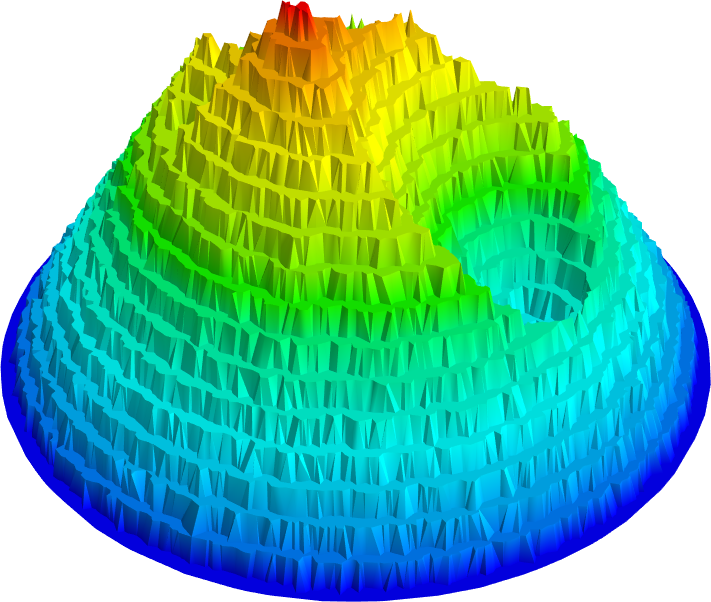}
\includegraphics[width=0.24\textwidth]{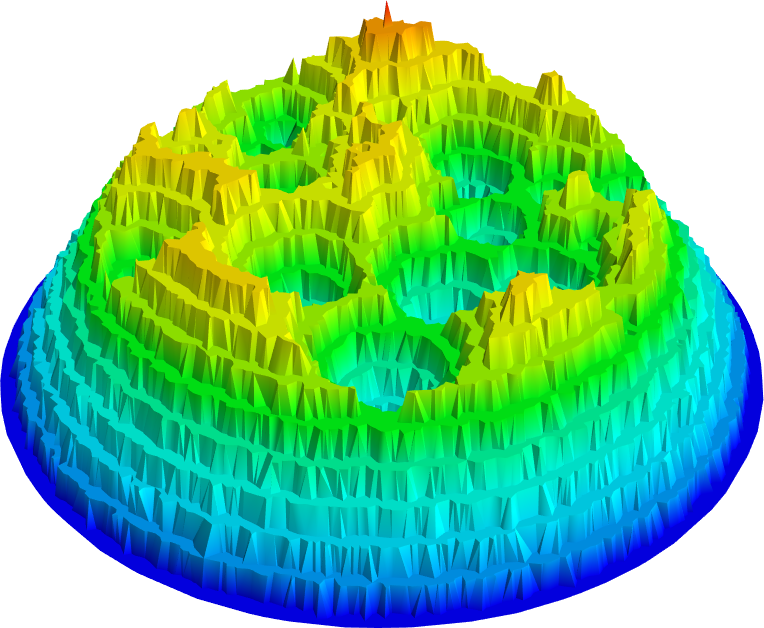}
\includegraphics[width=0.24\textwidth]{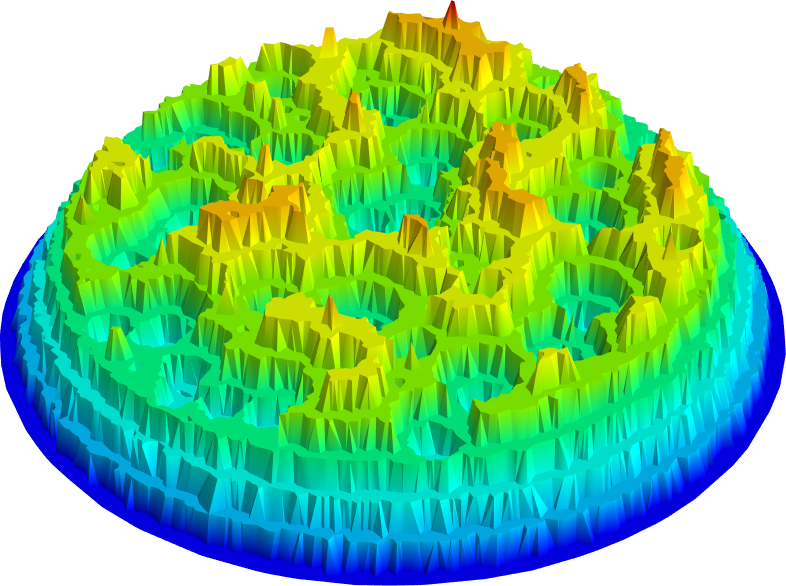}
\label{fig:robust_eikonal}}\\
\subfloat[$p$-eikonal equation with $p=1$ with corrupted edges]{
\includegraphics[width=0.24\textwidth]{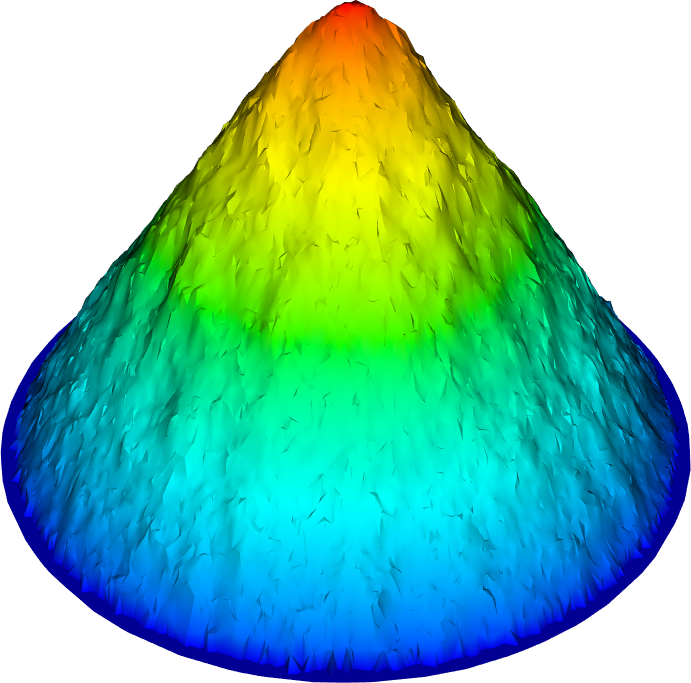}
\includegraphics[width=0.24\textwidth]{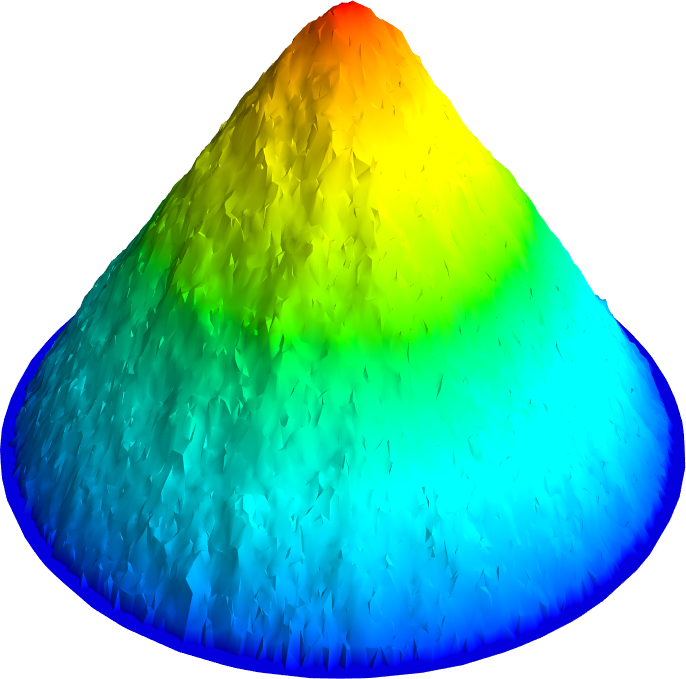}
\includegraphics[width=0.24\textwidth]{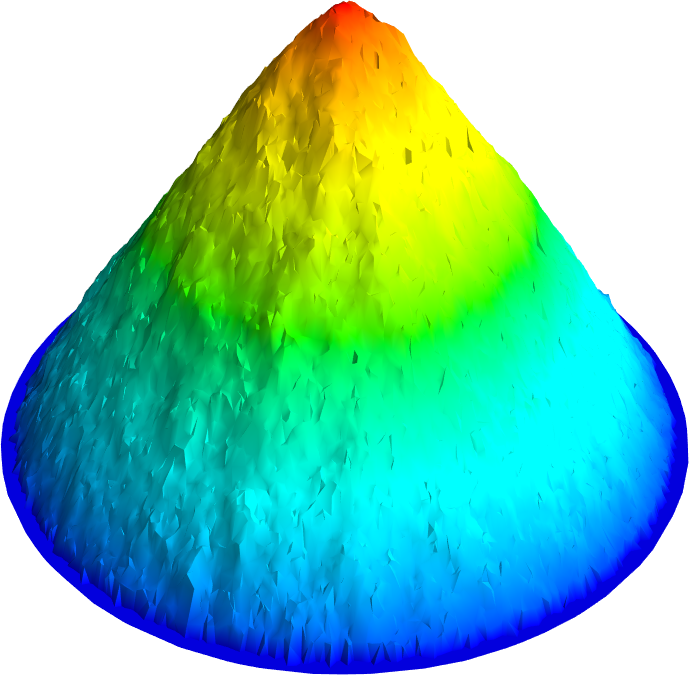}
\includegraphics[width=0.24\textwidth]{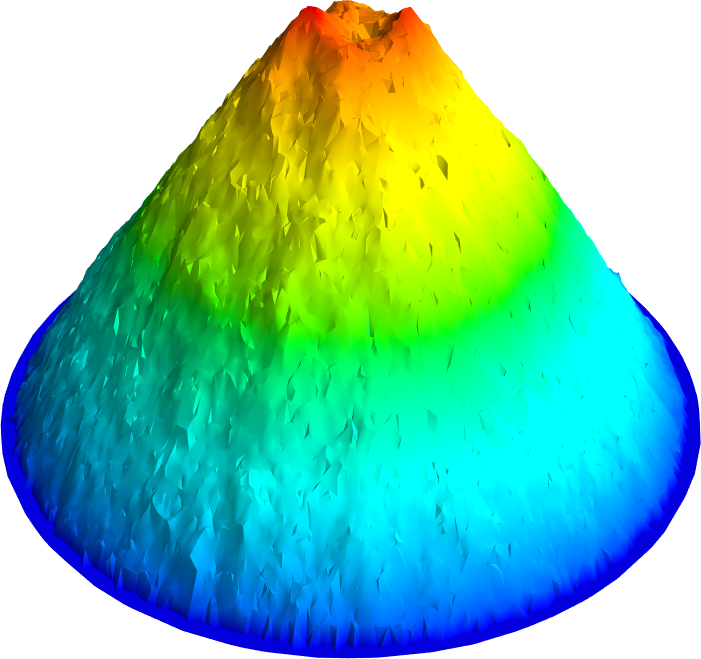}
%\subfloat[Graph distance function]{
%\includegraphics[clip=true,trim=110 45 100 50,width=0.24\textwidth]{robustness_dijkstra_0.png}
%\includegraphics[clip=true,trim=110 45 100 50,width=0.24\textwidth]{robustness_dijkstra_10.png}
%\includegraphics[clip=true,trim=110 45 100 50,width=0.24\textwidth]{robustness_dijkstra_20.png}
%\includegraphics[clip=true,trim=110 45 100 50,width=0.24\textwidth]{robustness_dijkstra_50.png}
%\label{fig:robust_eikonal}}\\
%\subfloat[$p$-eikonal equation with $p=1$]{
%\includegraphics[clip=true,trim=110 45 100 50,width=0.24\textwidth]{robustness_peikonal_0.png}
%\includegraphics[clip=true,trim=110 45 100 50,width=0.24\textwidth]{robustness_peikonal_10.png}
%\includegraphics[clip=true,trim=110 45 100 50,width=0.24\textwidth]{robustness_peikonal_20.png}
%\includegraphics[clip=true,trim=110 45 100 50,width=0.24\textwidth]{robustness_peikonal_50.png}
\label{fig:robust_peikonal}}
\caption{Robustness of graph-distance functions compared to the $p$-eikonal equation under random corruptions of edges in the graph. We computed each distance function on an unweighted proximity graph over $n=20000$ uniformly distributed random variables on the unit ball with graph connectivity length scale $\epsilon=0.05$. The boundary points $\Gamma$ were chosen to be all points within $\epsilon$ of the boundary of the unit ball, so the distance function gives a notion of data depth. From left to right we added an increasing number of corrupted edges (0, 10, 50, and 200) with edge weight $w_{ij}=1$.  We see the solution of the $p$-eikonal equation is far more robust under the addition of corrupted edges.}
\label{fig:robustness_graphdistance}
\end{figure}

The issue with lack of robustness of the graph eikonal equation \eqref{eq:density_eikonal} stems from the form of the \emph{max} in the operator, which means its value is highly sensitive to a single outlying edge weight. We introduce here the \emph{$p$-eikonal equation} on a graph, which uses information from all neighbors, and as we will show below, gives a more robust distance function on a graph. For $p>0$, we define the $p$-eikonal operator $\cA_{G,p}:\lx \to \lx$ by
\begin{equation}\label{eq:peikonal_op}
\cA_{G,p} u(x_i) = \sum_{j=1}^n w_{ji}(u(x_i) - u(x_j))_+^p,
\end{equation}
where $a_+:=\max\{a,0\}$ is the positive part. For $\Gamma \subset \X$ and $f\in \lx$, we consider the $p$-eikonal equation
\begin{equation}\label{eq:graph_peikonal}
\left\{
\begin{aligned}
\cA_{G,p}u  &=f ,&& \text{in } \X \setminus \Gamma \\
u &=0,&& \text{on } \Gamma.
\end{aligned}
\right.
\end{equation}
We show in Figure \ref{fig:robust_peikonal} the robustness experiment described in the last section with the $p$-eikonal equation with $p=1$. The $p$-eikonal equation is clearly more robust to the additional corrupted edges in the graph. After some preliminary results, we prove in Theorem \ref{thm:robust}, below, a robustness estimate for the $p$-eikonal equation that explains the experimental results in Figure \ref{fig:robust_peikonal}. 

We first use the theory from Section \ref{sec:general} to establish that \eqref{eq:graph_peikonal} is well-posed.  For $p\geq 0$ we denote by $G^p$ the graph $G^p=(\X,W^p)$ with weights $W^p=(w_{ij}^p)_{i,j=1}^n$. We interpret $0^0=0$ so that $G^0$ is the unweighted graph with the same edges as $G$.
\begin{theorem}[Well-posedness]\label{thm:peikonal_wellposed}
Let $p>0$ and $f>0$. If $G$ is connected, then \eqref{eq:graph_peikonal} has a unique solution $u\in \lx$, and
\begin{equation}\label{eq:peikonal_bound}
K^{-\frac{1}{p}}\left(\min_{\X}f^{\frac{1}{p}}\right)d_{G^{\frac{1}{p}}}(x_i,\Gamma)\leq u(x_i) \leq \left(\max_{\X}f^{\frac{1}{p}}\right)d_{G^{\frac{1}{p}}}(x_i,\Gamma).
\end{equation}
\end{theorem}
\begin{proof}
In the notation of Section \ref{sec:general}, the $p$-eikonal equation \eqref{eq:graph_peikonal} corresponds to the Hamiltonian
\[H(q,x_i) =\sum_{j=1}^n w_{ji}(q_j)_+^p - f(x_i).\]
This Hamiltonian is monotone, and positively $p$-homogeneous (and also convex in $q$ when $p\geq 1$). Thus, \eqref{eq:graph_peikonal} admits comparison by Lemma \ref{lem:comparison}. Hence, existence follows from the Perron method (Theorem \ref{thm:existence}), provided we can exhibit a subsolution $\phi$ and supersolution $\psi$ with $\phi=0\leq \psi$ on $\Gamma$. We can take $\phi=0$, but we will construct a larger subsolution to prove the bound \eqref{eq:peikonal_bound}. 

For $c>0$ to be determined, let us define 
\[\phi(x_i) = cd_{G^{\frac{1}{p}}}(x_i,\Gamma).\]
By Lemma \ref{lem:graph_eikonal}, $\phi$ solves the graph eikonal equation
\[\max_{x_j\in \X} w_{ji}^{\frac{1}{p}}(\phi(x_i) - \phi(x_j))_+ = c.\]
Since the right hand side $c$ is positive, we can trivially add the positive part above.  Then we have
\begin{align*}
\A_{G,p} \phi(x_i) &= \sum_{j=1}^n w_{ji}(\phi(x_i) - \phi(x_j))_+^p\\
&\leq K\max_{x_j\in \X}w_{ji}(\phi(x_i) - \phi(x_j))_+^p\\
&\leq K\left(\max_{x_j\in \X}w_{ji}^{\frac{1}{p}}(\phi(x_i) - \phi(x_j))_+\right)^p= Kc^p.
\end{align*}
Setting $c = K^{-\frac{1}{p}}\min_\X f^{\frac{1}{p}}$, we have $\A_{G,p}\phi \leq f$ on $\X\setminus \Gamma$, which proves the subsolution condition. We likewise define 
\[\psi(x_i) = Cd_{G^{\frac{1}{p}}}(x_i,\Gamma),\]
and use a similar argument to find that a choice of $C=\max_{\X}f^{\frac{1}{p}}$ yields the supersolution condition.
\end{proof}
\begin{remark}\label{rem:pinfty}
Let $u_p$ for $p>0$ denote the solution of \eqref{eq:graph_peikonal}, which exists and is unique due to Theorem \ref{thm:peikonal_wellposed}. By \eqref{eq:peikonal_bound} we see that $u_p \to d_{G^0,p}(\cdot,\Gamma)$ as $p\to \infty$. Thus, the $p\to \infty$ limit of the $p$-eikonal equation recovers the unweighted graph distance. By a similar argument, the solution of $\A_{G^p,p}u_p = f^p$ will satisfy $u_p\to u$ as $p\to \infty$, where $u$ is the solution of the graph eikonal equation \eqref{eq:density_eikonal}. 
\end{remark}

\subsubsection{Robustness to noise}

We now turn to the question of robustness of the $p$-eikonal equation to graph perturbations.  We consider a perturbation $\tilde{W} = W +\delta W$, where the perturbation matrix $\delta W\in \R^{n\times n}$ has nonnegative entries $\delta W_{ij}\geq 0$. This models corruption in the weight matrix by either adding new edges that did not exist in the original graph, or increasing the weights at existing edges. 
\begin{theorem}[Robustness]\label{thm:robust}
Let $\delta W$ have nonnegative entries, and set $\tilde{G}=(\X,W+\delta W)$ and $\delta G=(\X,\delta W)$. Let $\Gamma\subset \X$, $f\in \lx$ with $0< f_{min} \leq f \leq f_{max}$, and let $u,\tilde{u}\in \lx$ satisfy
\begin{equation}\label{eq:noisy_pde}
\left\{
\begin{aligned}
 \A_{\tilde{G},p}\tilde{u}(x_i) = \A_{G,p}u(x_i)  &=f(x_i) ,&& \text{if } x_i \in \X\setminus \Gamma \\
\tilde{u}(x_i) = u(x_i) &=0 ,&& \text{if } x_i\in \Gamma.
\end{aligned}
\right.
\end{equation}
Then for all $x_i\in \X$ we have
\begin{equation}
\label{eq:perturbBound}
0 \leq \frac{u(x_i) - \tilde{u}(x_i)}{u(x_i)} \leq \left(\max_{\X \setminus \Gamma} \frac{\A_{\delta G,p}u}{f}\right)^{\frac{1}{p}}.
\end{equation} 
\end{theorem}
\begin{proof}
We denote the entries of $\delta W$ by $\delta w_{ij}$. Since $\delta w_{ij}\geq 0$ we have $\A_{\tilde{G},p}u \geq \A_{G,p}u = f = \A_{\tilde G,p}\tilde u$ on $\X\setminus \Gamma$, and so by the comparison principle we have $0 \leq \tilde{u} \leq u$ on $\X$. By the linearity of $\A_{G,p}$ in the weight matrix we have 
\begin{align*}
\frac{\A_{\tilde{G},p}u(x_i)}{\A_{\tilde{G},p}\tilde{u}(x_i)} &= \frac{f(x_i)+ \A_{\delta G,p}u(x_i)}{f(x_i)} \leq 1 + \max_{\X \setminus \Gamma} \frac{\A_{\delta G,p}u}{f} =:C,
\end{align*}
for all $x_i\in \X\setminus \Gamma$. Therefore $\A_{\tilde{G},p}u \leq \A_{\tilde{G},p} (C^{\frac{1}{p}}\tilde{u})$ on $\X\setminus \Gamma$. By the comparison principle we have $u\leq C^{\frac{1}{p}}\tilde{u}$, and so
\[u(x_i) \leq  \left(1 + \max_{\X \setminus \Gamma} \frac{\A_{\delta G,p}u}{f}\right)^{\frac{1}{p}}\tilde{u}(x_i) \leq \tilde{u}(x_i) + \left(\max_{\X \setminus \Gamma} \frac{\A_{\delta G,p}u}{f}\right)^{\frac{1}{p}}u(x_i),\]
which completes the proof.
\end{proof}
\begin{remark}\label{rem:}
Theorem \ref{thm:robust} controls the relative error between $u$ and $\tilde{u}$. We note, in particular, that the dependence on $p$ shows that $p=1$ offers the greatest robustness, and as $p\to \infty$ we lose the robustness completely. We note that there are several ways we can reformulate Theorem \ref{thm:robust}. First, let us define the \emph{upwind} $1$-norm of a matrix, relative to the function $u\in \lx$, by
\[\|A\|_{u,1} = \max_{1\leq j\leq n} \sum_{i=1}^n |A_{ij}|\one_{u(x_j)>u(x_i)}.\]
We note that $\|A\|_{u,1}\leq \|A\|_1$, where $\|A\|_1 = \max_{1\leq j\leq n} \sum_{i=1}^n |A_{ij}|$ is the usual $1$-norm. The norm $\|\delta W\|_{u,1}$ measures the maximum amount of corruption among the incoming edges of any node from directions where $u$ is smaller (the \emph{upwind} direction). Then we compute
\[\A_{\delta G,p}u(x_i) = \sum_{j=1}^n \delta w_{ji}(u(x_i) - u(x_j))_+^p \leq u(x_i)^p \sum_{j=1}^n \delta w_{ji}\one_{u(x_i) > u(x_j)} \leq u(x_i)^p \|\delta W\|_{u,1}.\]
Thus, Theorem \ref{thm:robust} implies that
\[0 \leq \frac{u(x_i) - \tilde{u}(x_i)}{u(x_i)} \leq \left(\max_{\X \setminus \Gamma} \frac{u}{f^{\frac{1}{p}}}\right)\|\delta  W\|_{u,1}.\]
Finally, using the upper bound in Theorem \ref{thm:peikonal_wellposed} we obtain
\[0 \leq \frac{u(x_i) - \tilde{u}(x_i)}{u(x_i)} \leq C\left( \frac{f_{max}}{f_{min}}\right)^{\frac{1}{p}}\|\delta W\|^{\frac{1}{p}}_{u,1},\]
where $C=\max_{x_i\in \X}d_{G^{\frac{1}{p}}}(x_i,\Gamma)$.
\end{remark}

\subsubsection{Computational complexity}
\label{sec:comp}

The $p$-eikonal equation \eqref{eq:graph_peikonal} can be solved in a similar computational time as the graph eikonal equation \eqref{eq:density_eikonal} using the fast marching method \cite{sethian1996fast} on a graph. The solution of \eqref{eq:graph_peikonal} via fast marching requires repeatedly solving the equation
\begin{equation}\label{eq:scheme}
\sum_{j=1}^n w_{ji}(t - s_j)_+^p = a,
\end{equation}
for the unknown $t$, given $s_j$, $j=1,\dots,n$, and $a$. Of course, only the $s_j$ with $w_{ji}>0$ need to be considered. Since the left hand side is increasing in $t$, the equation can be solved with a bisection search for any $p>0$. Using a tolerance of $\delta$, the complexity of solving the scheme \eqref{eq:scheme} with a bisection search is $\cO(K\log(\delta^{-1}))$, where $K$ is the maximum unweighted degree of the graph defined in \eqref{eq:unweighted_degree}.  

When $p=1$, we can in fact solve the scheme \eqref{eq:scheme} explicitly without a bisection search. We first sort the $s_j$ in ascending order (and relabel the $w_{ji}$ in the same order), and then note that the solution $t$ will have the form
\[ t=t_m:=\frac{a + \sum_{j=1}^m w_{ji}s_j}{\sum_{j=1}^m w_{ji}},\]
for some $m\leq n$. We can compute all the $t_m$ recursively in $\cO(K\log(K))$ time, and simply check which is correct, yielding $\cO(K\log(K))$ complexity for solving \eqref{eq:scheme} when $p=1$.  A similar observation can be made for $p=2$, except that  $t=t_m$ will be the solution of a quadratic equation.

The fast marching method visits each node in the graph exactly once, in order of increasing values of the solution $u(x_i)$.  When each node is visited, the scheme \eqref{eq:scheme} is solved at all neighbors of the node. Each time the scheme is solved, a heap\footnote{The heap stores the current best guesses of $u(x_i)$ for nodes $x_i$ that have not been finalized/visited yet. At each iteration of fast marching, we need to retrieve the node with smallest best guess, which can be done in $\cO(\log(n))$ time with a heap data structure. When updating the scheme at all neighbors, the heap needs to be updated, also taking $\cO(\log(n))$ time.} of size at most $n$ is updated, which takes $\log(n)$ time. Thus the fast marching method takes $\cO(nK^2\log(K)\log(n))$ computational time for $p=1$ or $p=2$, and $\cO(nK^2\log(\delta^{-1})\log(n))$ time for other positive values of $p$, where $\delta$ is the bisection solver tolerance. In our implementation, of the method, we use the exact solution of the scheme for $p=1$, and the bisection search for all $p>1$ (i.e., we did not implement the quadratic method described above for $p=2$, since we found it did not improve over the bisection search).

\subsubsection{Shortest paths}
\label{sec:short}

While the solution of the $p$-eikonal equation \eqref{eq:graph_peikonal} does not represent a true distance function on the graph, as the eikonal equation \eqref{eq:density_eikonal} does, we can still construct a notion of a \emph{shortest path} from any $x_i\in \X\setminus \Gamma$ back to the set $\Gamma$, by descending on $u$ as quickly as possible. In particular, given the solution $u$ of \eqref{eq:graph_peikonal} and an initial point $x_{i_0}\in \X\setminus \Gamma$, we select the next point, for $k\geq 0$, to satisfy
\[x_{i_{k+1}} \in \argmin_{\substack{x_j\in \X \\ w_{j,i_k}>0}} u(x_j).\]
In other words, the next point is the neighbor of $x_{i_k}$ with the smallest value of $u$, which is the ``closest'' to $\Gamma$. Provided $f>0$ and $x_{i_k}\not\in \Gamma$, there must exist a neighbor with a strictly smaller value for $u$, otherwise we would have $\A_{G,p} u(x_{i_k}) = 0 < f(x_{i_k})$, which contradicts that $u$ solves the $p$-eikonal equation \eqref{eq:graph_peikonal}. Thus, the path chosen in this way is strictly decreasing in $u$, that is
\[u(x_{i_0}) > u(x_{i_1}) > u(x_{i_2}) > \cdots.\]
This guarantees that the path can never visit a node twice, and will eventually terminate at a point $x_{i_T}\in \Gamma$ after some number of steps, $T$. The shortest paths computed in this way are shown in red in the data depth experiments in Figures \ref{fig:depth} and \ref{fig:depth3D}. We also use this method to compute the shortest paths through real data in Section \ref{sec:numerics}. 

\subsection{Applications}
\label{sec:discrete_app}

The solution of the $p$-eikonal equation \eqref{eq:graph_peikonal}, while not a true graph distance function, gives us a notion of distance that is useful for data depth and semi-supervised learning. We discuss these applications initially in this section, and show the results of some experiments on toy datasets. We postpone experiments with real data to Section \ref{sec:numerics}.

Given a set $\Gamma\subset \X$ and a density estimation $\hat{\rho}:\X\to \R$, we consider solving the density reweighted $p$-eikonal equation
\begin{equation}\label{eq:graph_peikonal_weighted}
\left\{
\begin{aligned}
\cA_{G,p}u  &=\hat{\rho}^{-\alpha}, && \text{in } \X \setminus \Gamma \\
u &=0,&& \text{on } \Gamma,
\end{aligned}
\right.
\end{equation}
where the exponent $\alpha$ is a tunable parameter.  We denote the solution of \eqref{eq:graph_peikonal_weighted} by
$D^{p,\alpha}_{\Gamma}(x) = u(x)$. When $\Gamma=\{x\}$ is a single point we write $D^{p,\alpha}_x$.

\subsubsection{Data depth}
\label{sec:depth_discrete}

\begin{figure}[!t]
\centering
\subfloat[Moon]{\includegraphics[height=0.25\textheight,clip=true,trim=100 30 200 50]{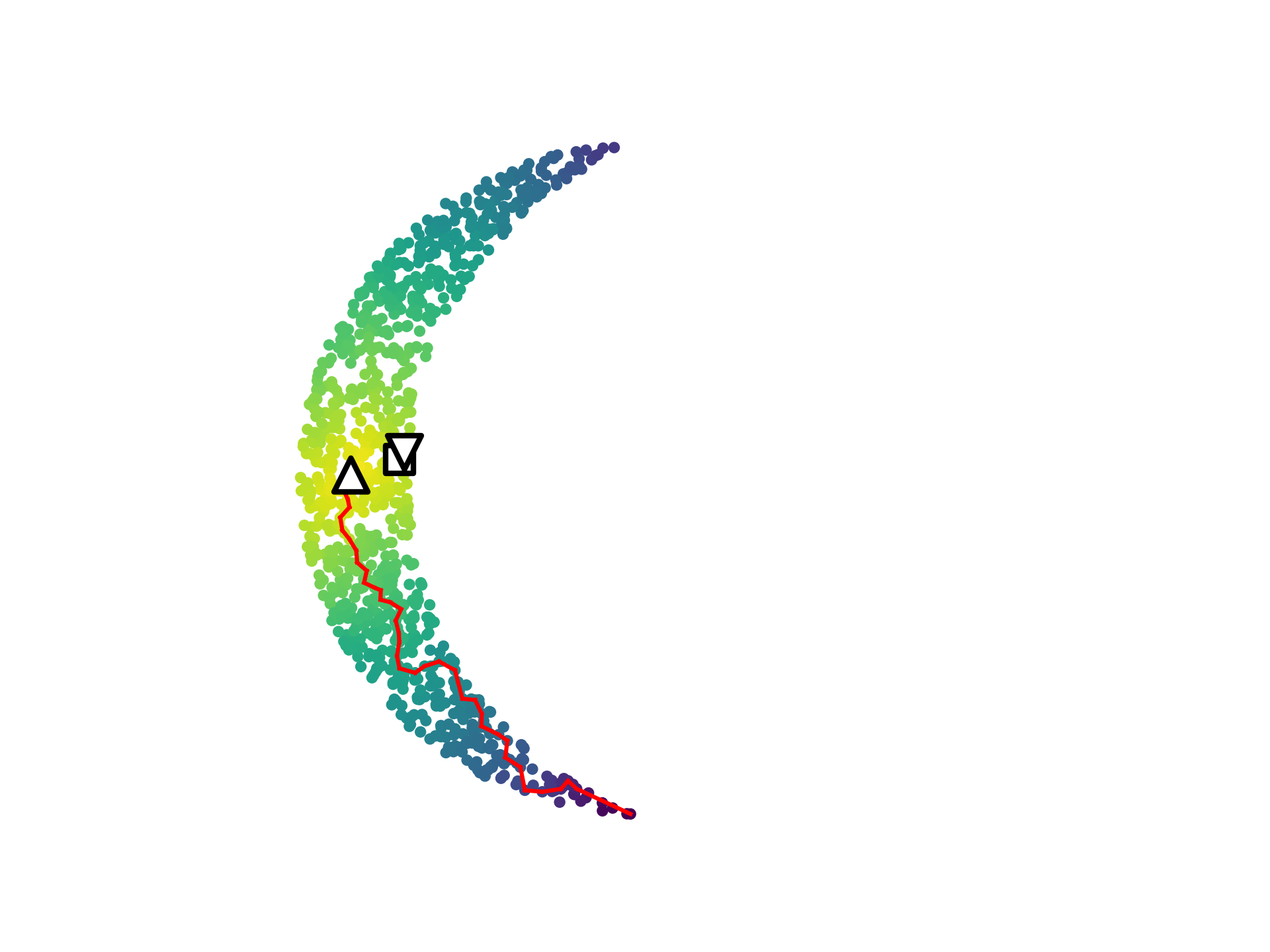}}
\subfloat[Gaussian]{\includegraphics[height=0.25\textheight,clip=true,trim=100 30 100 50]{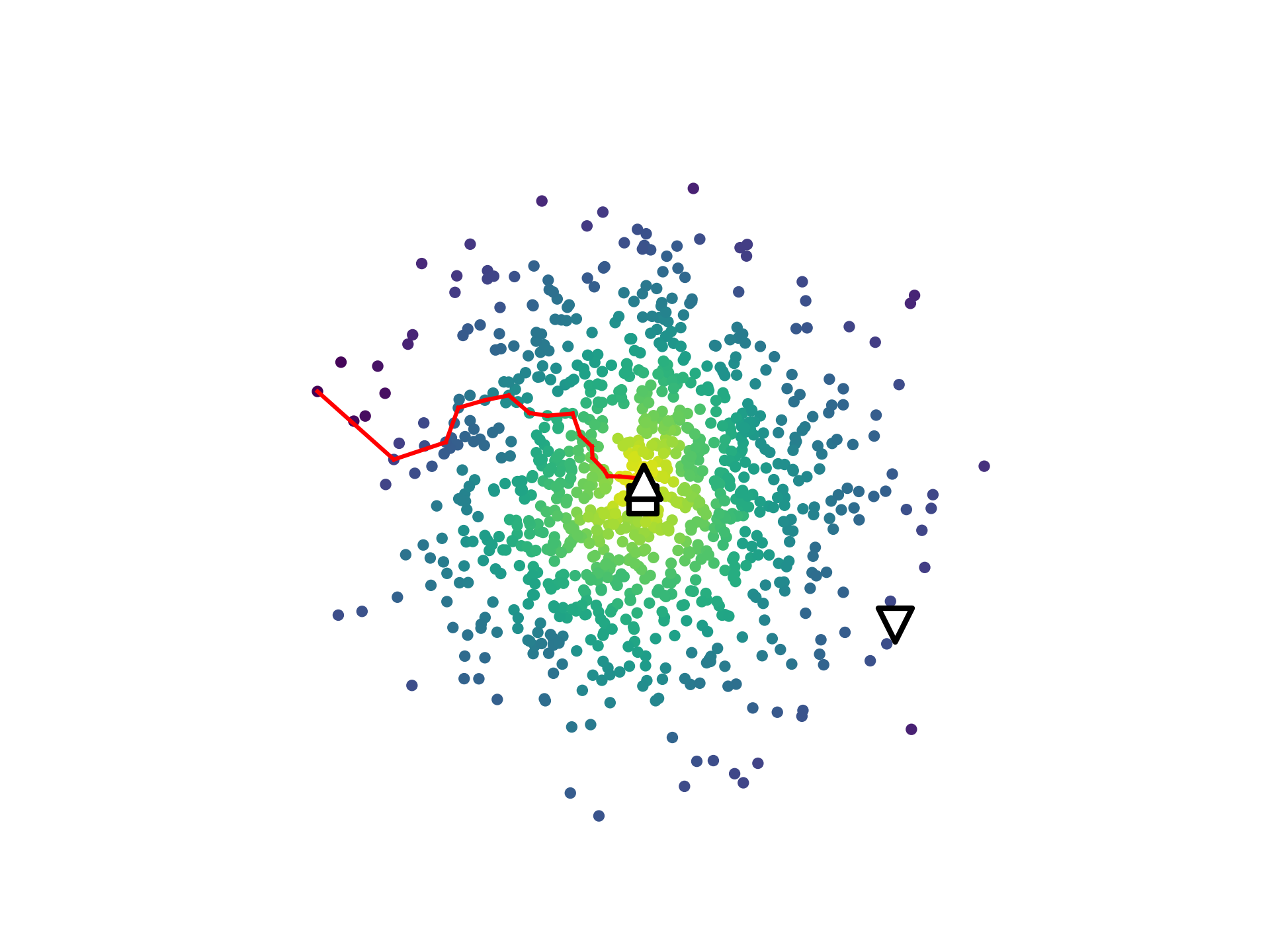}}
\hspace{3mm}
\subfloat[Gaussian mixture]{\includegraphics[height=0.25\textheight,clip=true,trim=100 30 100 50]{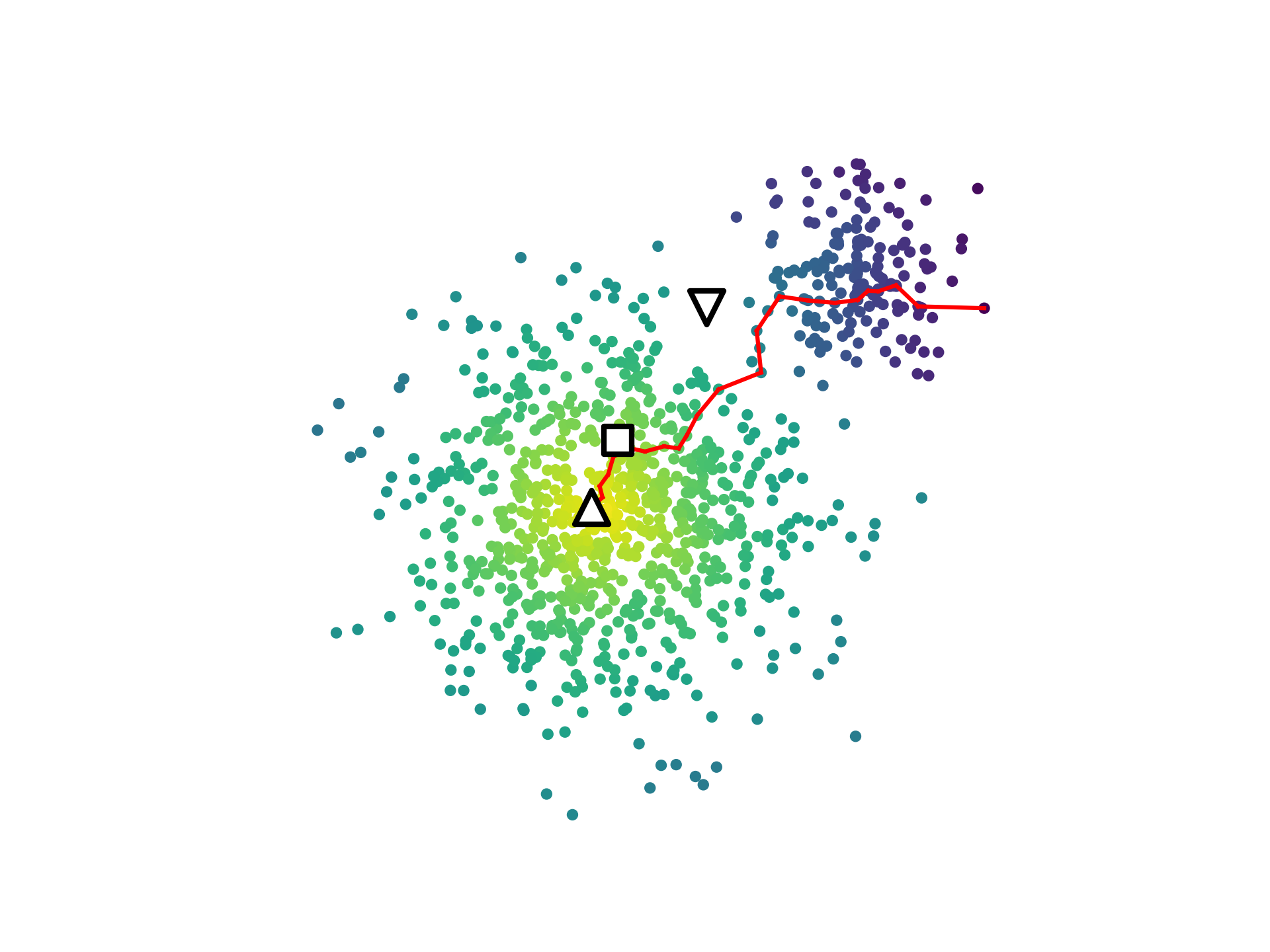}}
\caption{The $p$-eikonal medians and depth on 2D toy datasets with $p=1$. The medians are shown for $\alpha=-1$ ($\triangledown$), $\alpha=0$ ($\square$) and the $\alpha=1$ ($\triangle$), while the points are colored by the $\alpha=1$ data depth. We also show the shortest path from the shallowest point to the deepest point in red. We only recommend $\alpha\geq 0$ in all our applications; we have shown $\alpha=-1$ just to illustrate how reverse density weighting affects the median computation (in this case, it prefers placing the median in sparse regions of the graph).}
\label{fig:depth}
\end{figure}
\begin{figure}[!t]
\centering
\subfloat[Helix]{\includegraphics[height=0.18\textheight,clip=true,trim=160 110 140 80]{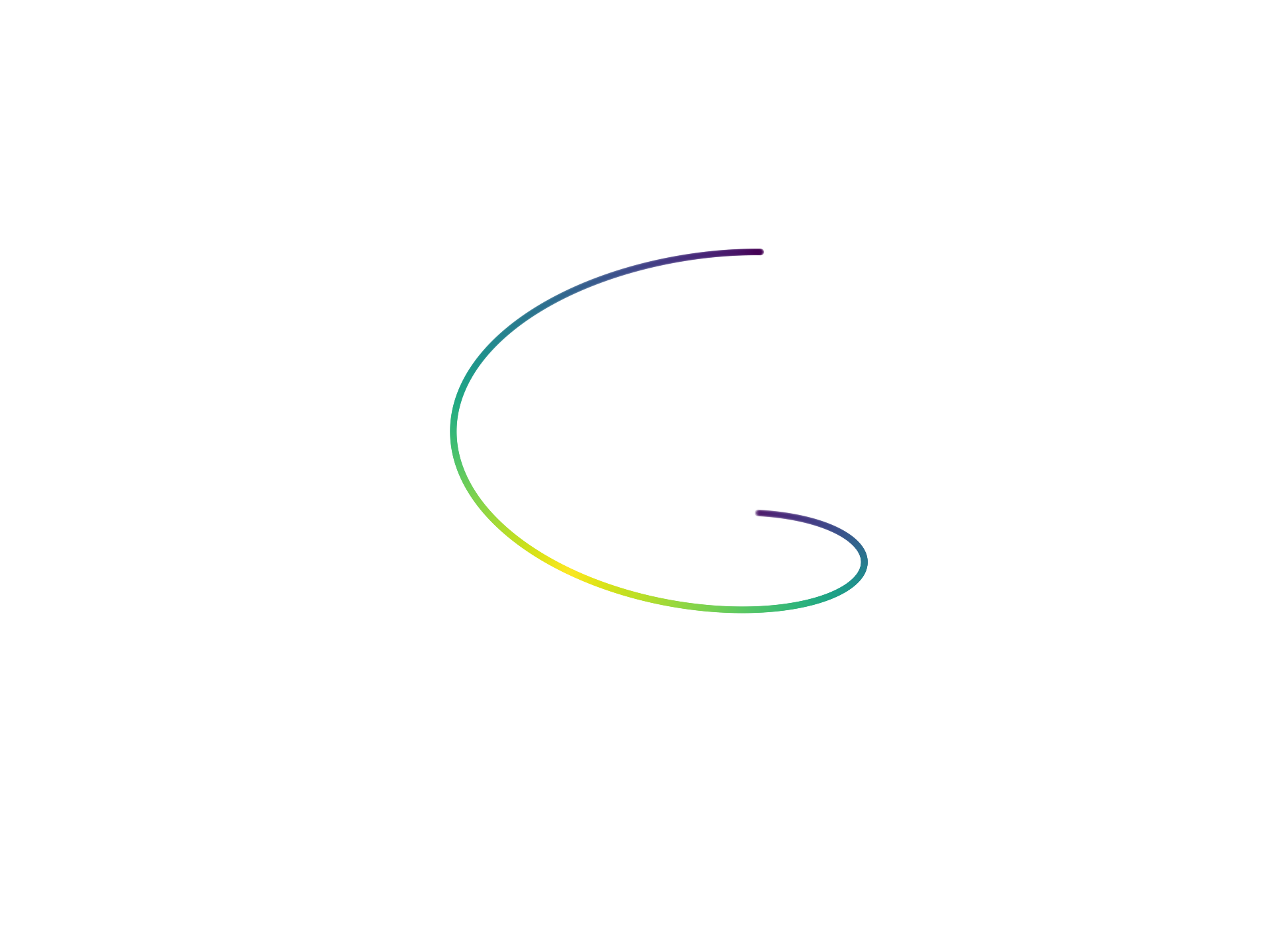}}
\subfloat[Half Sphere]{\includegraphics[height=0.18\textheight,clip=true,trim=130 80 130 110]{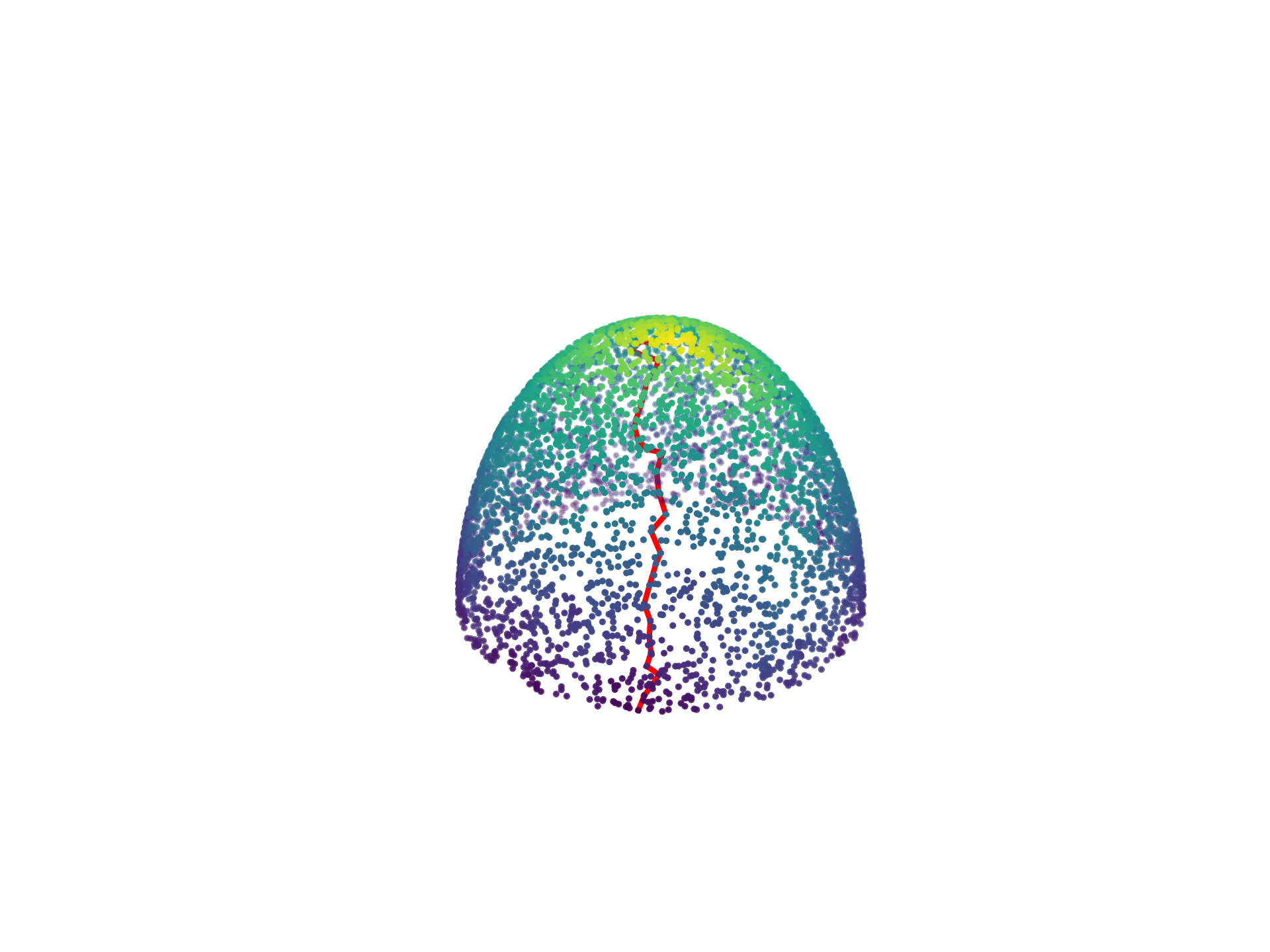}}
\subfloat[Swiss Roll]{\includegraphics[height=0.18\textheight,clip=true,trim=130 100 130 110]{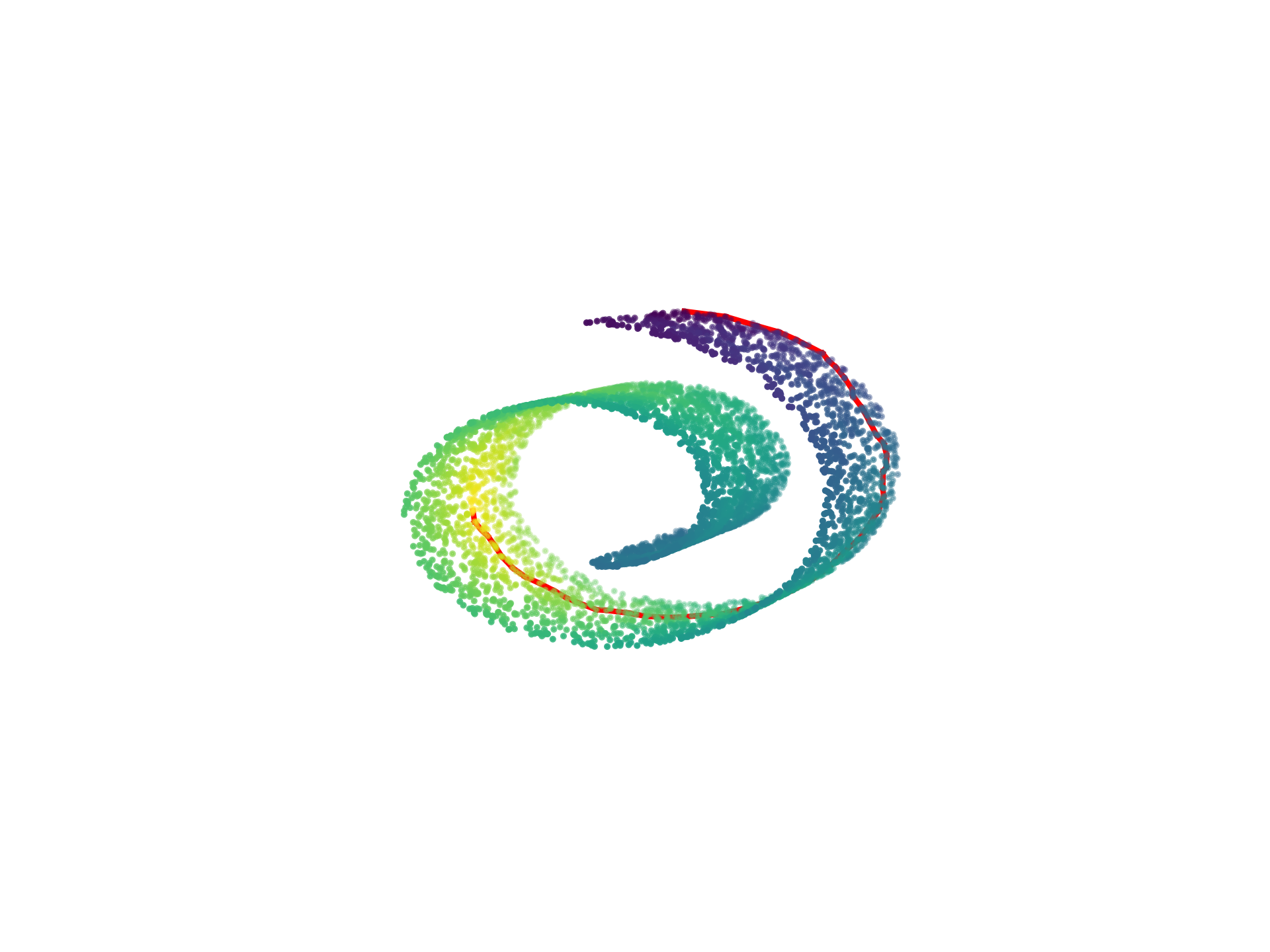}}
\caption{The $p$-eikonal data depth on 3D toy datasets sampled from manifolds embedded in $\R^3$. We use $p=1$ and $\alpha=1$. We note that the swiss roll is more dense on one end than the other, which explains why the depth is not symmetric along the length of the roll.}
\label{fig:depth3D}
\end{figure}

We can approach data depth through the framework of the geometric median. Let us recall that for a collection of points $x_1,\dots,x_n$ in $\R^d$, the \emph{geometric median} $x_*$ is defined by
\[x_* \in \argmin_{x\in \R^d}\sum_{i=1}^n |x_i-x|.\]
The geometric median generalizes the 1-dimensional median, and inherits many of its robustness properties (its breakdown point is also $0.5$, for example). Given the notion of depth $D^{p,\alpha}_x$, we define the $p$-eikonal median $x_{p,\alpha}$ by
\begin{equation}\label{eq:peikonal_median}
x_{p,\alpha}\in \argmin_{x_j\in \X} \sum_{x_i \in \X}D^{p,\alpha}_{x_j}(x_i).
\end{equation}
In practice, we approximate the median by restricting $x_j \in \hat{\X} \subset \X$, where $\hat{\X}$ is a much smaller subset of $\X$ chosen at random. In all our experiments we take $\hat{\X}$ to have 5\% of the points in $\X$. 

Once we have computed the median $x_{p,\alpha}$, we obtain a notion of data depth via the distance to the median
\[ \text{depth}_{p,\alpha}(x) = \max_\X D^{p,\alpha}_{x_{p,\alpha}} - D^{p,\alpha}_{x_{p,\alpha}}(x).\]
Figure \ref{fig:depth} gives an example of the medians and depths for different toy datasets in 2 dimensions, and for $\alpha\in \{-1,0,1\}$. We use $p=1$ in all experiments, and color the point cloud by the $\alpha=1$ depth. We can see that the $\alpha=1$ median outperforms the other weighting choices. In particular, in the Gaussian mixture example, the $\alpha=1$ median is completely insensitive to the addition of the outlying cluster, which has $\frac{1}{6}$ of the points in the main cluster. We show in Figure \ref{fig:depth3D} example of the $p$-eikonal median and depth on point clouds sampled from submanifolds of $\R^3$. In this case we just show the $\alpha=1$ depth. In all images we also show the shortest path, computed as described in Section \ref{sec:short}, from the shallowest to the deepest point.

Let us remark briefly that the density weighting with $\alpha>0$ encourages the median to be placed in regions of high density, since path lengths are shorter here. In contrast, taking $\alpha<0$ encourages the median to be in regions of low density. We do not recommend using $\alpha<0$ in data depth (or in semi-supervised learning).   We also remark that in Figure \ref{fig:depth3D}, the depth on the swiss roll is not symmetric along the length of the roll. This is to be expected with density weighting, since the swiss roll is more dense near one end of the roll (near the origin) and less dense on the other end. We postpone examples of the $p$-eikonal depth on real data to Section \ref{sec:numerics}.

\subsubsection{Semi-supervised learning}
\label{sec:ssl_discrete}

Given the pseudo distances $D^{p,\alpha}_\Gamma$ we can perform semi-supervised learning with a nearest neighbor approach. Suppose we have $k$ classes, and for each class $j=1,\dots,k$, we are provided some labeled nodes $\Gamma_j\subset \X$. The label prediction $\ell_i$ for an unlabeled node $x_i\not \in \Gamma_j$ for any $j$, is the label of the closest labeled node, under the distance $D^{p,\alpha}_\Gamma$, that is 
\begin{equation}\label{eq:label_dec}
\ell_i = \argmin_{1 \leq j \leq k}D^{p,\alpha}_{\Gamma_j}(x_i).
\end{equation}
Semi-supervised learning with the $p$-eikonal equation thus requires solving $k$ separate $p$-eikonal equations, which is similar to the one-vs-rest approach in machine learning for producing a multi-class classifier out of a binary one. 

As we shall see in our analysis later in Section \ref{sec:analysis}, distance-based classifiers can be highly sensitive to the geometry of the clusters, even with appropriate density weighting. In such cases, we can improve the accuracy of the classifier by incorporating information about \emph{class priors}, so that the classifier predicts the correct proportion of nodes in each class. To do this, we follow \cite{calder2020poisson} and modify the label decision with the addition of positive weights $s_1,\dots,s_k$ so that the new label decision is 
\begin{equation}\label{eq:label_dec_priors}
\ell_i = \argmin_{1 \leq j \leq k}\left\{s_jD^{p,\alpha}_{\Gamma_j}(x_i)\right\}.
\end{equation}
By increasing or decreasing the weights $s_j$, we can increase or decrease the number of nodes predicted in each class. The weights $s_j$ can be adjusted incrementally until class balancing is achieved. We do this with the volume constrained label projection method from \cite{calder2020poisson}. 

\begin{figure}[!t]
\centering
\subfloat{\includegraphics[width=0.25\textwidth,clip=true,trim=50 50 50 50]{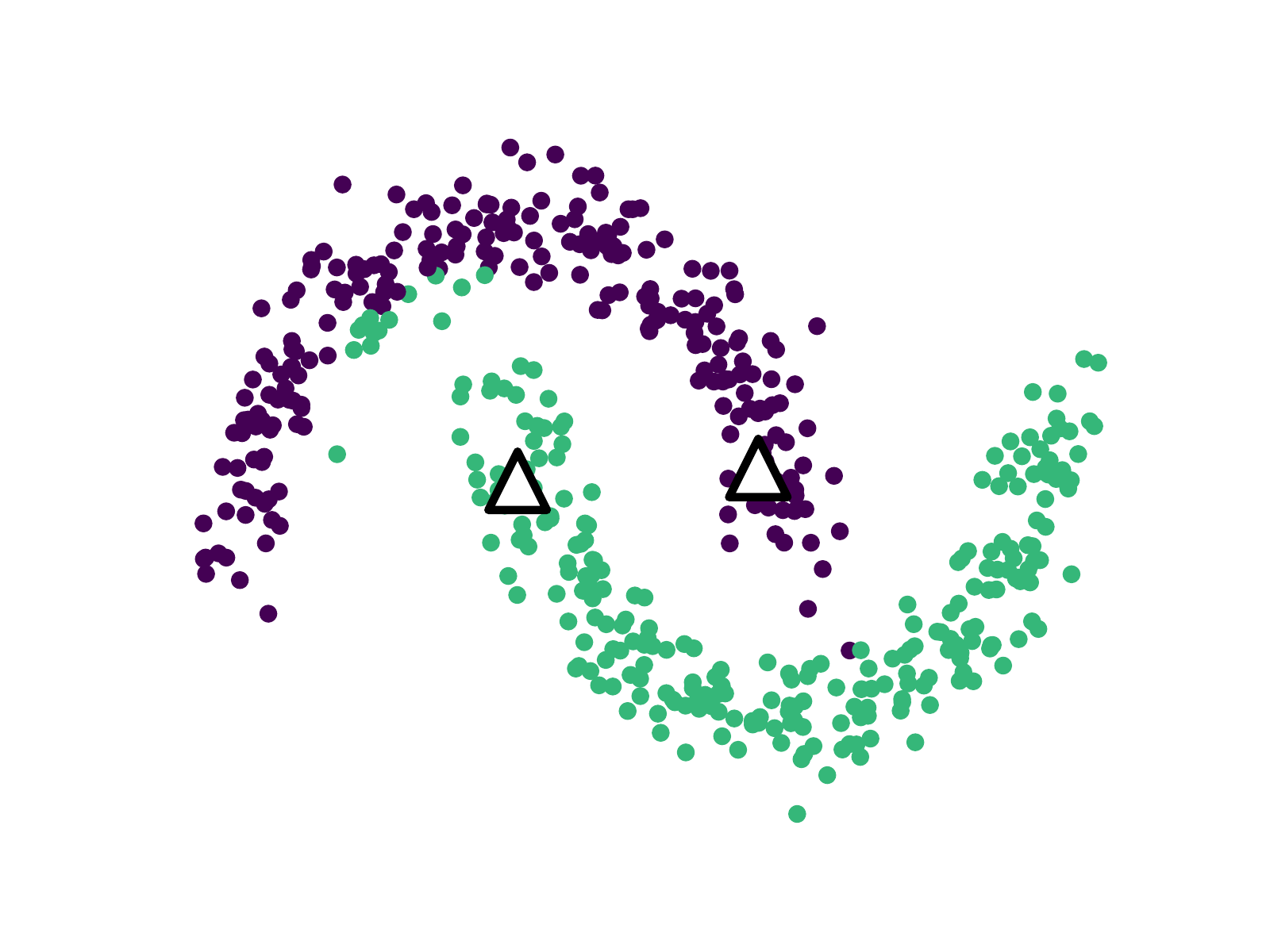}}
\subfloat{\includegraphics[width=0.25\textwidth,clip=true,trim=50 50 50 50]{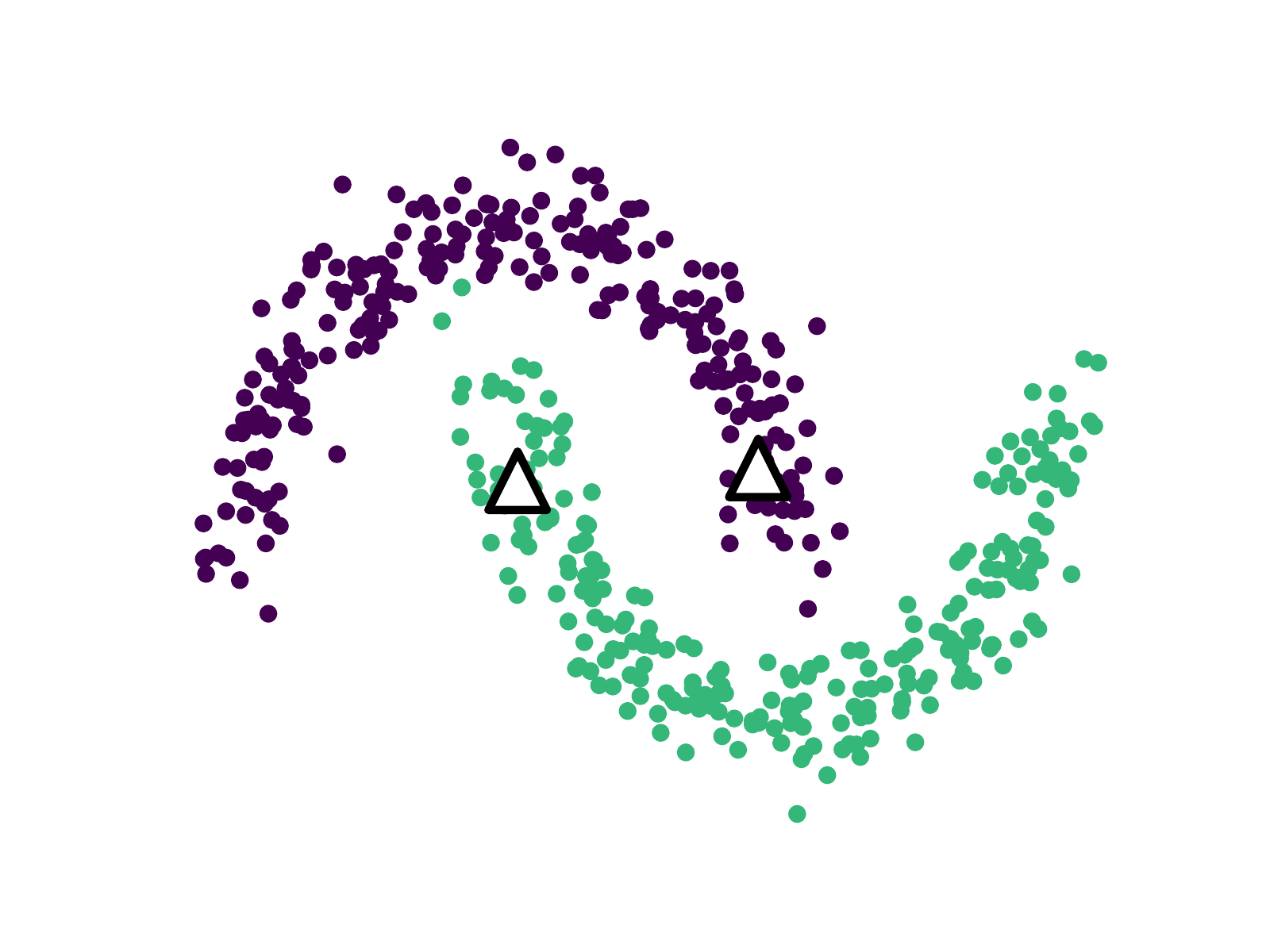}}
\subfloat{\includegraphics[width=0.25\textwidth,clip=true,trim=50 50 50 50]{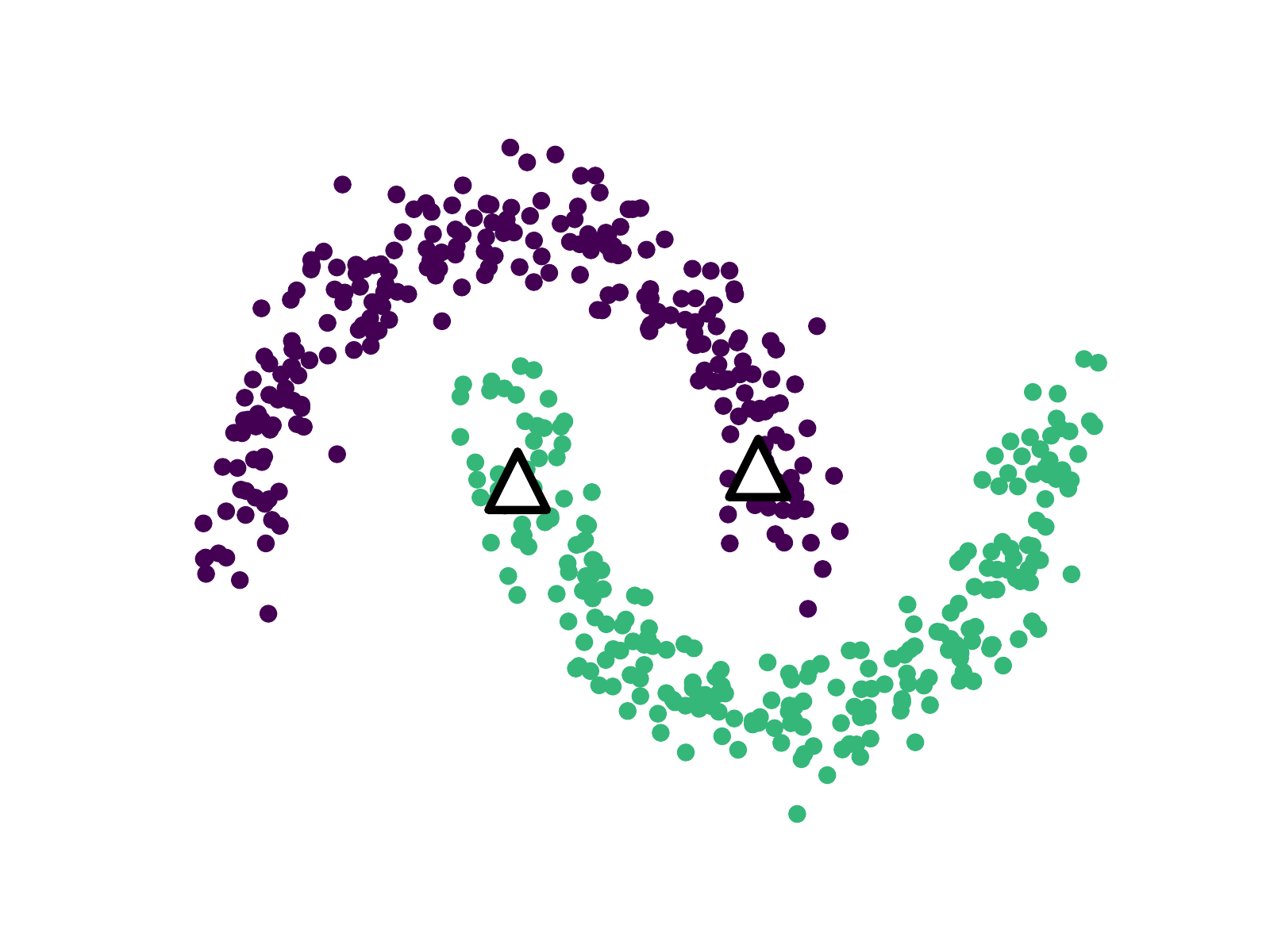}}
%\subfloat[$\alpha=1$ with class priors]{\includegraphics[width=0.33\textwidth,clip=true,trim=50 50 50 50]{twomoons_priors_100.00_3_1.pdf}}
%\subfloat[]{\includegraphics[width=0.33\textwidth,clip=true,trim=50 50 50 50]{twomoons_priors_99.20_3_-1.pdf}}
%\subfloat[]{\includegraphics[width=0.33\textwidth,clip=true,trim=50 50 50 50]{twomoons_priors_99.60_3_0.pdf}}
\subfloat{\includegraphics[width=0.25\textwidth,clip=true,trim=50 50 50 50]{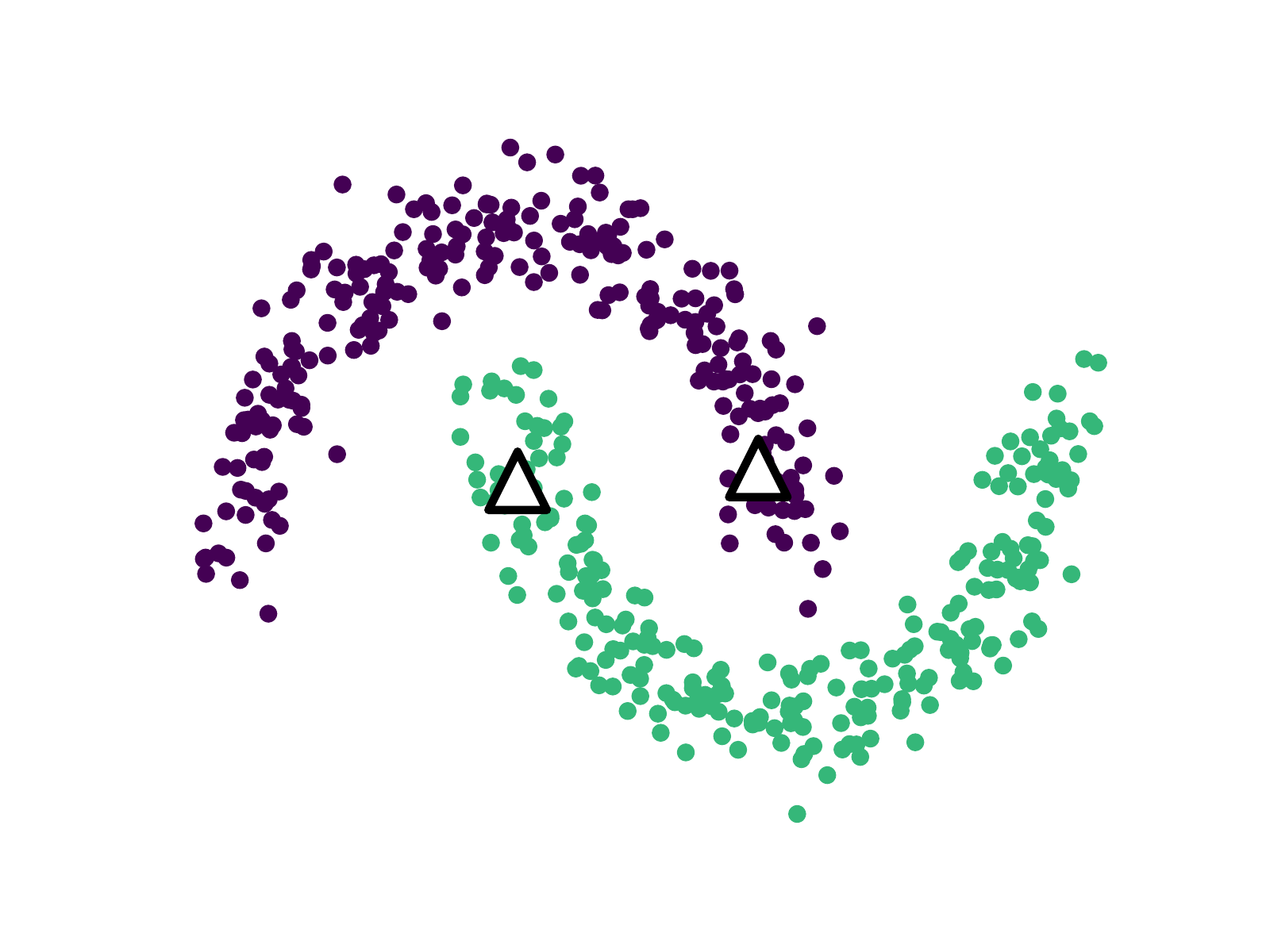}}\\
\subfloat[$\alpha=-1$]{\includegraphics[width=0.25\textwidth,clip=true,trim=50 50 50 50]{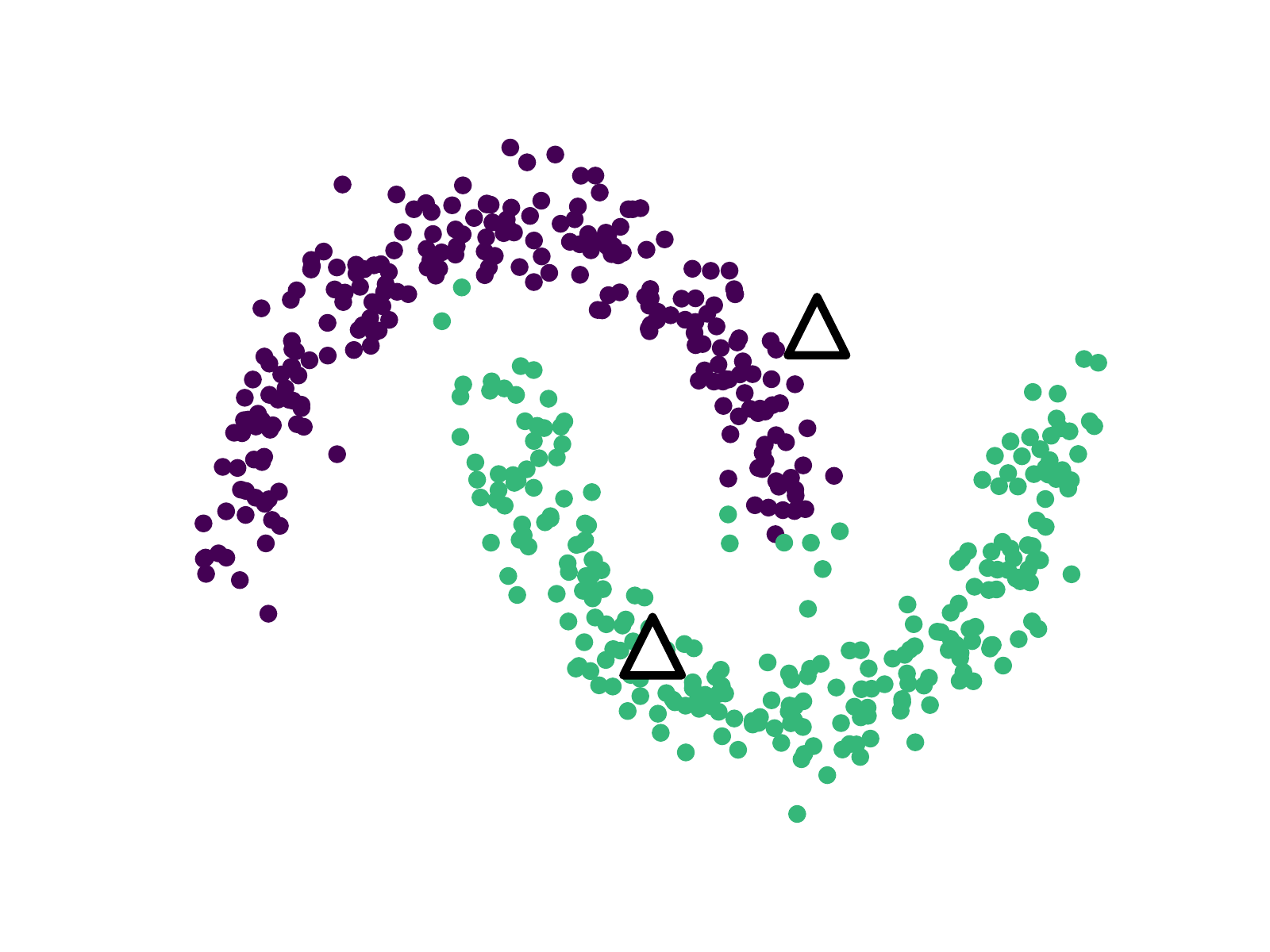}}
\subfloat[$\alpha=0$]{\includegraphics[width=0.25\textwidth,clip=true,trim=50 50 50 50]{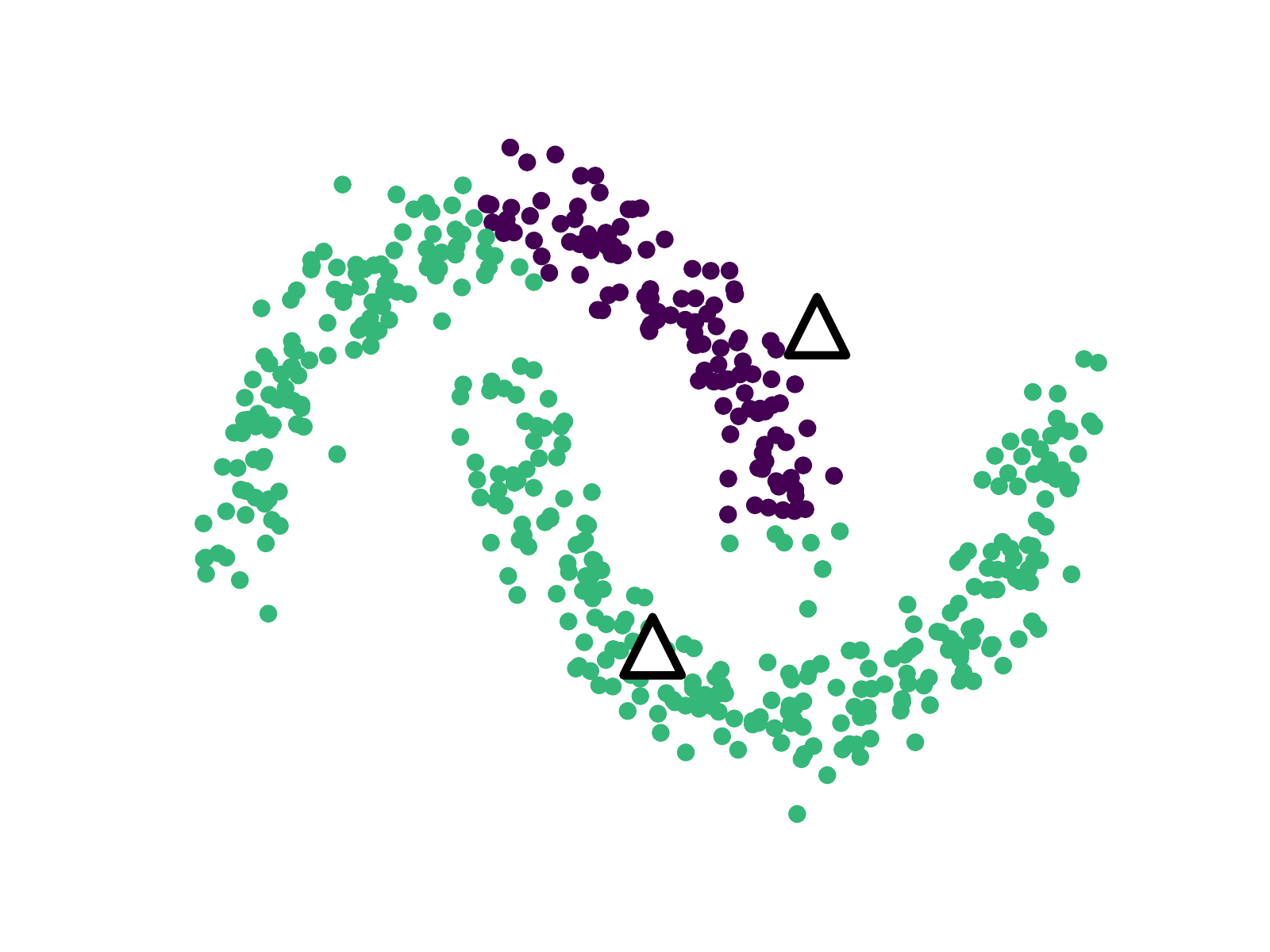}}
\subfloat[$\alpha=1$]{\includegraphics[width=0.25\textwidth,clip=true,trim=50 50 50 50]{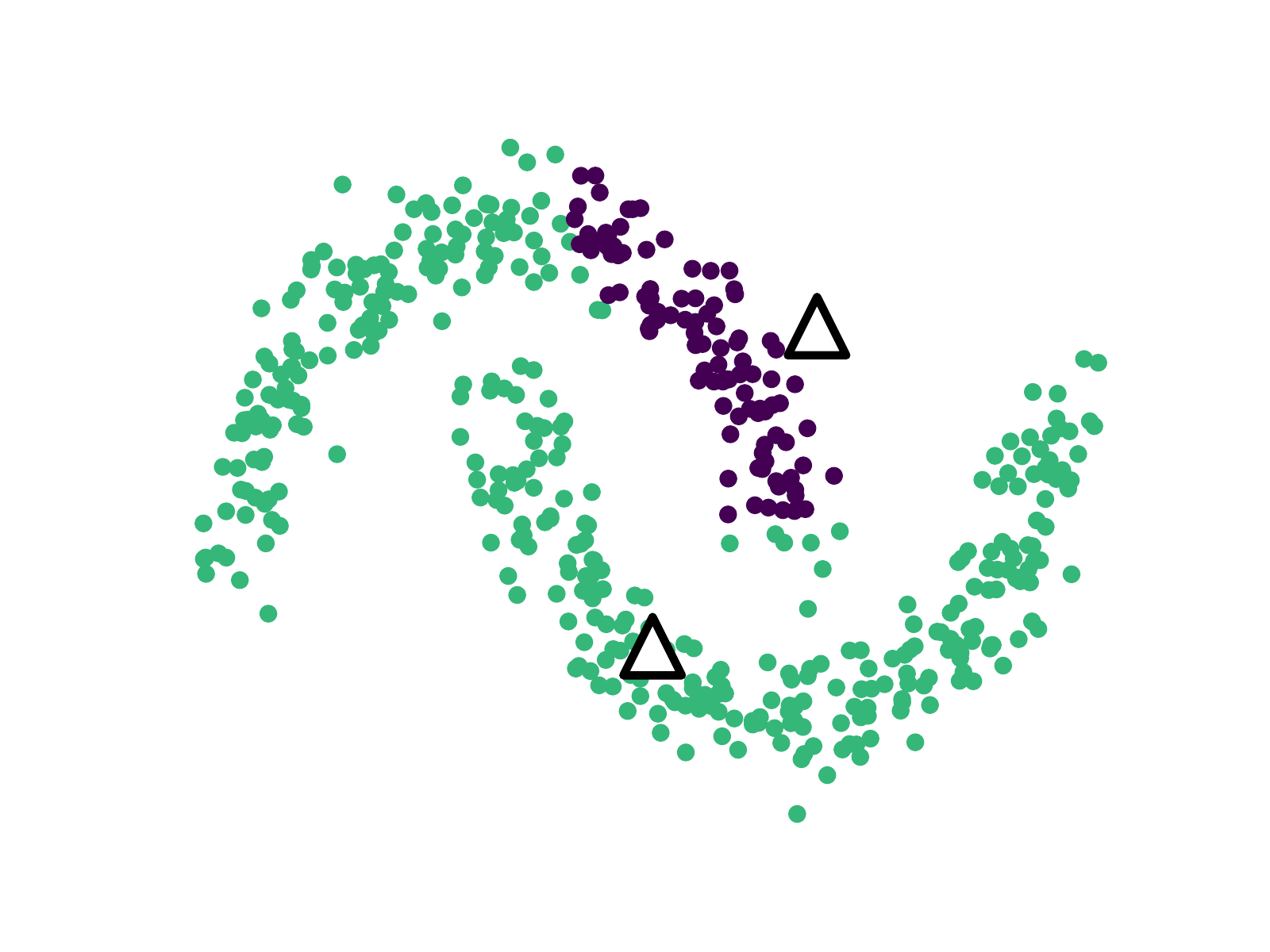}}
%\subfloat[$\alpha=-1$ with priors]{\includegraphics[width=0.33\textwidth,clip=true,trim=50 50 50 50]{twomoons_priors_100.00_15_-1.pdf}}
%\subfloat[$\alpha=0$ with priors]{\includegraphics[width=0.33\textwidth,clip=true,trim=50 50 50 50]{twomoons_priors_100.00_15_0.pdf}}
\subfloat[$\alpha=1$ with class priors]{\includegraphics[width=0.25\textwidth,clip=true,trim=50 50 50 50]{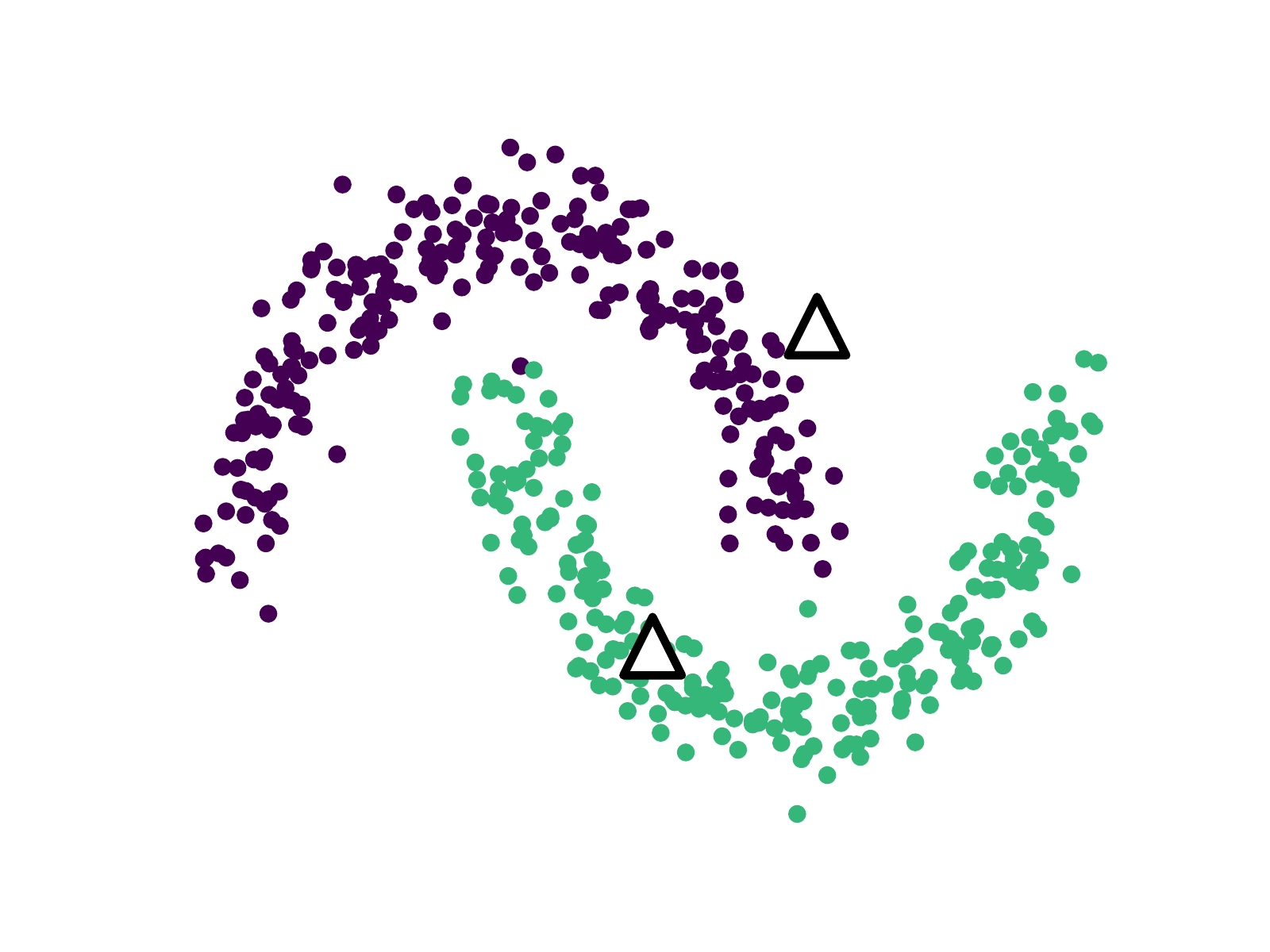}}\\
\caption{Example of semi-supervised learning with the density weighted $p$-eikonal equation on the two moons dataset. The $\triangle$ markers give the locations of the two labeled points in each example. We show different choices of density reweighting, and the addition of class priors on the right side.}
\label{fig:ssl_twomoons}
\end{figure}

As a preliminary toy example, we consider classification of the two-moons dataset in Figure \ref{fig:ssl_twomoons}. The two rows in the figure correspond to different choices of the labeled nodes. In each case we take one label per class and indicate its position with a $\triangle$. In the first row, the training nodes are both inliers in their respective clusters, and all choices of weighting exponents $\alpha$ give good classification, and the addition of class priors is not needed. In the second row, the training point for the upper half of the moon is an outlier for that cluster, and the lower cluster leaks over significantly for $\alpha=0,1$. We see on the right that this issue can be corrected with the addition of class prior information, to enforce the predicted classes to have the same size. It is also interesting to note in the second row that the reverse density weighting $\alpha=-1$ produces the correct classification without class priors. This is because the reverse density weighting brings the outlying training point closer to its cluster. In general, when we do not expect training points to be outliers, we do not recommend reverse density weighting in semi-supervised learning (and we do not observe good results with reverse density weighting with real data). In Section \ref{sec:numerics} we present more in depth results with semi-supervised learning on real data.

%% file: sections/continuum.tex
\section{State-constrained eikonal equations}
\label{sec:continuum_eikonal}

The continuum limit of the $p$-eikonal equation is a PDE called the state-constrained eikonal equation
\begin{equation}\label{eq:state_constrained_eikonal}
\left\{
\begin{aligned}
|\nabla u| &= f,&& \text{in } \Omega\setminus \Gamma \\
u &= 0,&& \text{on } \Gamma.
\end{aligned}
\right.
\end{equation}
The equation is \emph{state-constrained} because, as we shall see below, the solution represents a geodesic distance function to $\Gamma$, and the set $\Omega$ constrains the geodesic paths (e.g., the state).   Before proving the discrete to continuum convergence, which we do in Section \ref{sec:convergence}, we need to review some properties of the state-constrained eikonal equation \eqref{eq:state_constrained_eikonal}. 

Throughout this section we assume $f$ is positive and Lipschitz continuous, $\Omega\subset \R^d$ is an open, bounded and connected domain, with a $C^{1,1}$ boundary $\partial \Omega$, and $\Gamma \subset \Omega$ is a closed set where we specify the homogeneous Dirichlet boundary conditions. In particular, we are not explicitly specifying boundary conditions on $\partial\Omega$, and instead we consider the \emph{state constrained} problem \cite{capuzzo1990hamilton}.  For the reader unfamiliar with PDE theory, we note that the assumption that $\partial \Omega$ is $C^{1,1}$ is equivalent to assuming there is a radius $R$ such that at every boundary point $x\in \partial \Omega$, there exist balls of radius $R$ touching $x$ from inside and outside the domain \cite{lewicka2020domains}. This is also equivalent to assuming the \emph{reach}  of the boundary, as a submanifold of $\R^d$ is lower bounded by $R$, and that the unit normal vector to the boundary is Lipschitz with constant $\frac{1}{R}$.  Throughout this section we use the $C^{0,1}$ norm of a function, which is defined by
\[\|u\|_{C^{0,1}(\Omega)} = \|u\|_{L^\infty(\Omega)} + \Lip(u),\]
where $\|u\|_{L^\infty(\Omega)} = \max_{x\in \bar{\Omega}} |u(x)|$ and 
\[\Lip(u) = \sup_{\substack{x,y\in \bar{\Omega}\\x\neq y}} \frac{|u(x)-u(y)|}{|x-y|}.\]

We review the definition of viscosity solution for the state constrained problem here. 
\begin{definition}\label{def:viscosity_solution}
We say that $u\in C(\bar{\Omega})$ is a \emph{viscosity subsolution} of \eqref{eq:state_constrained_eikonal} if $u\leq 0$ on $\Gamma$ and if for each $x\in \Omega\setminus \Gamma$ and each $\phi \in C^\infty(\R^d)$ such that $u-\phi$ has a local maximum at $x$, we have
\begin{equation}\label{eq:viscosity_subsol}
|\nabla \phi(x)| \leq f(x).
\end{equation}

We say that $v\in C(\bar{\Omega})$ is a \emph{viscosity supersolution} of \eqref{eq:state_constrained_eikonal} if $v\geq 0$ on $\Gamma$ and if for each $x\in \bar{\Omega}\setminus \Gamma$ and each $\phi \in C^\infty(\R^d)$ such that $v-\phi$ has a local minimum at $x$, relative to $\bar{\Omega}$, we have
\begin{equation}\label{eq:viscosity_supersol}
|\nabla \phi(x)| \geq f(x).
\end{equation}

We say that $u$ is a \emph{viscosity solution} of \eqref{eq:state_constrained_eikonal} if $u$ is both a viscosity subsolution and a viscosity supersolution.
\end{definition}
Notice the key difference between the super and subsolution definitions is that in state constrained problems, we require the supersolution property to hold on the boundary $\partial \Omega$, but do not require the same in the  subsolution property. We give a simple justification for this fact in the connection with the variational interpretation below. For a reference on viscosity solutions of Hamilton-Jacobi equations and connections to optimal control, we refer the reader to \cite{bardi1997optimal}, while as a reference for state constrained Hamilton-Jacobi equations, we refer to \cite{capuzzo1990hamilton}.

We quote below a comparison principle for state constrained Hamilton-Jacobi equations.
\begin{theorem}[\cite{capuzzo1990hamilton}]\label{thm:comparison_prinicple}
If $u$ is a viscosity subsolution of \eqref{eq:state_constrained_eikonal} and $v$ is a viscosity supersolution, then $u\leq v$ on $\bar{\Omega}$.
\end{theorem}
It follows from Theorem \ref{thm:comparison_prinicple} that solutions of \eqref{eq:state_constrained_eikonal} are unique. Existence can be obtained with the Perron method, or through the variational interpretation, which we discuss next.

\subsection{Variational interpretation}

The variational interpretation of \eqref{eq:state_constrained_eikonal} states that the solution of \eqref{eq:state_constrained_eikonal} is essentially a distance function on $\Omega$ to the set $\Gamma$, where distance is weighted by the positive function $f$. In particular, we first define the pairwise distance  
\begin{equation}\label{eq:distance}
d_f(x,y) := \inf\left\{ \int_0^1 f(\gamma(t))|\gamma'(t)| \, dt \,: \, \gamma\in C^1([0,1];\bar{\Omega}), \gamma(0)=x, \text{ and }  \gamma(1)=y   \right\}.
\end{equation}
The function $d_f:\bar{\Omega}\times \bar{\Omega} \to \R$ is a metric, and in particular, it satisfies the triangle inequality 
\[d_f(x,z) \leq d_f(x,y) + d_f(y,z).\]
We denote the distance function $d_f$ with $f\equiv 1$ as
\[d_\Omega(x,y) = d_1(x,y).\]
The function $d_\Omega :\bar{\Omega}\times \bar{\Omega} \to \R$ is the \emph{geodesic distance function} on $\Omega$. Associated with the geodesic distance function, we define geodesic balls by
\[B_\Omega(x,r) = \{y\in \bar{\Omega}\, : \, d_\Omega(x,y) \leq r\}.\]
We will have to frequently utilize the geodesic distance $d_\Omega(x,y)$ in place of the Euclidean distance $|x-y|$, and we will need to compare the two distances. Since the boundary $\partial\Omega$ is $C^{1,1}$, there exists a constant $C>0$, depending only on $\partial\Omega$ such that
\begin{equation}\label{eq:geodesic_euclidean}
|x-y| \leq d_\Omega(x,y) \leq |x-y| + C |x-y|^2 \ \ \text{for all } x,y\in  \bar{\Omega}.
\end{equation}
In fact, if the domain is \emph{convex}  then we have $d_\Omega(x,y)=|x-y|$, but we will not place such strong assumptions on the domain here. Associated with the geodesic distance, we also define the geodesic diameter
\[\diam(\Omega) = \max_{x,y\in \bar{\Omega}}d_{\Omega}(x,y).\]
The geodesic diameter is finite, since $\Omega$ is connected (and hence path connected).

Given the definition of the path distance function $d_f(x,y)$, we recall that the solution $u$ of \eqref{eq:state_constrained_eikonal} is given by the variational representation formula
\begin{equation}\label{eq:variational_interpretation}
u(x) = \min_{y\in \Gamma} d_f(x,y).
\end{equation}
%While the connection between variational problems and Hamilton-Jacobi equations is well-studied \cite{bardi1997optimal}, the situation of the state-constrained problem is less well-known. We give below a short sketch of a proof that the function $u$ defined by \eqref{eq:variational_interpretation} is indeed the viscosity solution of \eqref{eq:state_constrained_eikonal}. We, in particular, point out the key place in the proof where the supersolution condition is obtained on the boundary, but the subsolution condition is not.
\begin{theorem}\label{thm:existence_state_const}
The function $u$ defined in \eqref{eq:variational_interpretation} is the unique viscosity solution of \eqref{eq:state_constrained_eikonal}. 
\end{theorem}
The proof of Theorem \ref{thm:existence_state_const} is standard in viscosity solution theory, and follows arguments in \cite{bardi1997optimal} closely. We include a proof in Appendix Section \ref{sec:proofs} for the interested reader.

\subsection{Lipschitz regularity}

The variational interpretation of the eikonal equation gives a simple proof of Lipschitzness of the solution $u$.\begin{lemma}\label{lem:lip}
Let $u\in C(\bar{\Omega})$ be the solution of \eqref{eq:state_constrained_eikonal}. Then $u$ is Lipschitz continuous and 
\[\Lip(u) \leq C\|f\|_{L^\infty(\Omega)},\]
where $C$ depends only on $\diam(\Omega)$ and $\partial\Omega$.
\end{lemma}
\begin{proof}
Since $\Omega$ is open and connected, we have $d_\Omega(x,y)<\infty$ for all $x,y\in \bar{\Omega}$. By \eqref{eq:geodesic_euclidean}, there exists $\tilde{C}>0$ such that
\[d_\Omega(x,y) \leq \tilde{C}|x-y| \ \ \text{for all } x,y\in \bar{\Omega} \text{ with } |x-y|\leq 1.\]
For $|x-y|\geq 1$ we have
\[d_\Omega(x,y) \leq \diam(\Omega) \leq \diam(\Omega)|x-y|.\]
Therefore
\[d_\Omega(x,y) \leq C|x-y| \ \ \text{for all } x,y\in \bar{\Omega},\]
where $C = \max\{\tilde{C},\diam(\Omega)\}$.

Using Theorem \ref{thm:existence_state_const} and the dynamic programming principle we have
\[u(y)  \leq u(x) + d_f(x,y) \leq u(x) + \|f\|_{L^{\infty}(\Omega)}d_\Omega(x,y).\]
Swapping the roles of $x$ and $y$ yields
\[|u(x) - u(y)| \leq \|f\|_{L^{\infty}(\Omega)}d_\Omega(x,y) \leq C\|f\|_{L^\infty(\Omega)}|x-y|,\]
which completes the proof.
\end{proof}

\subsection{Domain perturbations}

In our discrete to continuum convergence theory in Section \ref{sec:convergence} below, we will need some results on the stability of the solution $u$ of \eqref{eq:state_constrained_eikonal} under perturbations in the domain $\Omega$.  Let us define the signed distance function to the boundary $\partial \Omega$ by
\[d_{\partial\Omega}(x) = 
\begin{cases}
	\text{dist}(x,\partial\Omega),&\text{if }x\in \bar{\Omega}\\
	-\text{dist}(x,\partial\Omega),&\text{otherwise}.
\end{cases}\]
For $\delta \in \R$ we also define
\begin{equation}\label{eq:enlarged}
	\Omega_\delta = \{x\in \R^d \, : \, d_{\partial\Omega}(x) > \delta\} \ \ \text{ and } \ \ \partial_\delta\Omega = \{x\in \R^d \, : \, d_{\partial \Omega}(x) \leq \delta\} .
\end{equation}
%, and so there exists $C>0$, depending only on $\partial \Omega$, such that
%\begin{equation}\label{eq:dc11}
%|\nabla d_{\partial \Omega}(x) - \nabla d_{\partial \Omega}(y)| \leq C|x-y| \ \ \text{for all} \ \ x,y\in \partial_{\frac{R}{2}}\Omega. 
%\end{equation}

\begin{theorem}\label{thm:domain_perturbation}
For $\delta\in \R$ let $u_\delta\in C(\bar{\Omega_\delta})$ denote the viscosity solution of
\begin{equation}\label{eq:state_constrained_eikonal_delta}
\left\{
\begin{aligned}
|\nabla u_\delta| &= f,&& \text{in } \Omega_\delta\setminus \Gamma \\
u_\delta &= 0,&& \text{on } \Gamma,
\end{aligned}
\right.
\end{equation}
and let $u=u_0$ be the viscosity solution of \eqref{eq:state_constrained_eikonal}. There exists $C,c>0$, depending only on $\partial\Omega$ and $\dist(\Gamma,\partial\Omega)$, such that whenever $|\delta| \leq c$ the following hold.
\begin{enumerate}[{\rm (i)}]
\item $\Lip(u_\delta) \leq C\|f\|_{L^\infty(\Omega)}$, and
\item $\|u - u_\delta\|_{L^\infty(\Omega_{\delta_+})}\leq C f_{min}^{-1}\|f\|_{C^{0,1}(\Omega_{\delta_-})}\delta$, where $f_{min}=\min_{\Omega_{\delta_-}}f$.
\end{enumerate}
\end{theorem}
\begin{proof}
Since the boundary $\partial \Omega$ is $C^{1,1}$, the reach of $\partial\Omega$ is bounded below by a positive number $R>0$ (in fact, $\frac{1}{R}$ is the Lipschitz constant of the unit normal vector to the boundary). Hence, within the tube $\partial_{\frac{R}{2}}\Omega$, the signed distance function $d_{\partial \Omega}$ is uniformly $C^{1,1}$. Hence, the perturbed boundaries $\partial \Omega_\delta$ are uniformly $C^{1,1}$ for $|\delta| \leq \frac{R}{4}$. Invoking Lemma \ref{lem:lip} proves (i). We take $c\leq \frac{R}{2}$ smaller, if necessary, so that $\Gamma\subset \Omega_c$, and we assume $|\delta|\leq c$ for the rest of the proof. 

We will prove the case of $\delta>0$; the proof for $\delta<0$ is very similar. It is clear that $u \leq u_\delta$ on $\Omega_\delta$, since there are more restrictions on the feasible paths in the variational interpretation of $u_\delta$, compared to $u$. To prove the estimate in the other direction, that $u_\delta \leq u + C\delta$, we use the comparison principle Theorem \ref{thm:comparison_prinicple}, with a suitable extension of $u_\delta$ to  $\Omega$. 

We define the cutoff function
\begin{equation}\label{eq:zeta}
\zeta(x) =
\begin{cases}
1, & \text{if } 0 \leq d_{\partial \Omega}(x) \leq \frac{R}{4}\\
2-\frac{4}{R}d_{\partial \Omega}(x),& \text{if } \frac{R}{4}\leq d_{\partial \Omega}(x) \leq \frac{R}{2}\\
0,& \text{if } d_{\partial \Omega}(x) \geq \frac{R}{2}.
\end{cases}
\end{equation}
The function $\zeta$ is a Lipschitz cutoff functions near the boundary $\partial \Omega$. Since $|\nabla d_{\partial\Omega}|=1$ we have that $|\nabla \zeta| \leq \frac{4}{R}=C$, where $C$ depends only on $\partial\Omega$, at all points of differentiability of $\zeta$ in $\partial_{\frac{R}{2}}\Omega$. We now define the extended function $w\in C(\bar{\Omega})$ by 
\begin{equation}\label{eq:ueps}
	w(x) = u_\delta(x + \delta \,\zeta(x) \nabla d_{\partial \Omega}(x)). 
\end{equation}
To shed light on the definition of $w$, we note that $\nabla d_{\partial \Omega}$ gives a natural extension of the unit inward normal vector $\nu$ from the boundary $\partial\Omega$ to the tube $\partial_R \Omega$. Indeed, $\nabla d_{\partial \Omega}$ agrees with the unit inward normal vector on the boundary $\partial \Omega$, and in fact, $\nabla d_{\partial\Omega}(x)=\nu(x_*)$, where $x_*\in \partial \Omega$ is the closest point to $x$ from the boundary. Thus, we are simply stretching $u_\delta$ onto the larger domain $\Omega$.

We first check that $w$ is well-defined. If $x\in \Omega \setminus \Omega_{\frac{R}{4}}$,  then $\zeta(x)=1$ and so
\[d_{\partial \Omega}(x + \delta \,\zeta(x) \nabla d_{\partial \Omega}(x))= d_{\partial \Omega}(x + \delta \nabla d_{\partial \Omega}(x)) = d_{\partial \Omega}(x) + \delta > \delta.\]
Hence $x+ \delta \,\zeta(x) \nabla d_{\partial \Omega}(x)\in \Omega_\delta$ belongs to the domain of $u_\delta$. If $x\in \Omega_{\frac{R}{4}}$, then 
\[d_{\partial \Omega}(x + \delta \zeta(x) \nabla d_{\partial \Omega}(x)) = d_{\partial \Omega}(x) + \zeta(x) \delta \geq d_{\partial\Omega}(x)  > \frac{R}{4} \geq \delta,\]
and we reach the same conclusion. This establishes that $u_\delta$ is well-defined.

We will show that $w$ is a viscosity subsolution of a similar equation, and then apply the comparison principle. To do this, we will use the fact that for a Hamiltonian that is convex in the gradient (i.e., the eikonal Hamiltonian $|\nabla u|$), Lipschitz continuous almost everywhere subsolutions are also viscosity subsolutions (the same is not true for supersolutions). This is a standard fact in viscosity solution theory, whose proof can be found in standard references \cite{bardi1997optimal}. Thus, we can work directly with $\nabla w$ at points of differentiability, instead of using the test function definition of viscosity solutions.

Let $x\in \Omega$, and assume that $w$ and $\zeta$ are differentiable and $x$, and that $d_{\partial\Omega}$ is twice differentiable at $x$. Then we compute
\[\nabla w(x) = [I + \delta(\zeta \nabla^2 d_{\partial\Omega}(x) + \nabla \zeta(x) \otimes \nabla d_{\partial\Omega}(x)] \nabla u_\delta(x + \delta \zeta(x) \nabla d_{\partial\Omega}(x)).\]
Taking norms on both sides yields
\begin{align*}
|\nabla w(x)| &\leq \|I + \delta(\zeta \nabla^2 d_{\partial\Omega}(x) + \nabla \zeta(x) \otimes \nabla d_{\partial\Omega}(x)\| |\nabla u(x + \delta \zeta(x) \nabla d_{\partial\Omega}(x))|\\
&\leq \left(1 + \delta\left(\|\nabla^2 d_{\partial\Omega}(x)\| + |\nabla \zeta(x)| |\nabla d_{\partial\Omega}(x)| \right)\right) f(x + \delta \zeta(x) \nabla d_{\partial\Omega}(x)).
\end{align*}
Since $\nabla d_{\partial\Omega}$ is Lipschitz continuous in $\Omega_{\frac{R}{2}}$, we have a uniform bound on $\|\nabla^2 d_{\partial\Omega}\|$ at all points of differentiability. Thus, taking $C$ larger, if necessary, we have
\[|\nabla w(x)|  \leq (1 + C\delta)f(x + \delta \zeta(x) \nabla d_{\partial\Omega}(x)) \leq f(x) + C\|f\|_{C^{0,1}(\Omega)}\delta .\]
Set $v(x) = (1 + Cf_{min}^{-1}\|f\|_{C^{0,1}(\Omega)}\delta)u(x)$. Then $v$ is a viscosity solution of
\[|\nabla v(x)| \geq  (1 + Cf_{min}^{-1}\|f\|_{C^{0,1}(\Omega)}\delta)f(x) \geq f(x) + C\|f\|_{C^{0,1}(\Omega)}\delta.\]
By the comparison principle Theorem \ref{thm:comparison_prinicple} we have $w \leq v$, and hence
\[u_\delta(x + \delta \,\zeta(x) \nabla d_{\partial \Omega}(x)) \leq u(x) + Cf_{min}^{-1}\|f\|_{C^{0,1}(\Omega)}\delta \]
for all $x\in \Omega$. For $x\in \Omega_\delta$ we compute
\[u_\delta(x) \leq u_\delta(x + \delta \,\zeta(x) \nabla d_{\partial \Omega}(x)) + \Lip(u_\delta)\delta \leq u(x) + Cf_{min}^{-1}\|f\|_{C^{0,1}(\Omega)}\delta + C\|f\|_{L^\infty(\Omega)},\]
which completes the proof.
\end{proof}

%% file: sections/convergence.tex
\section{Discrete to continuum convergence}
\label{sec:convergence}

In this section we establish a continuum limit for the $p$-eikonal equation on a random geometric graph. In particular, we show that even though the $p$-eikonal equation does not correspond to a graph distance function, its continuum limit does in fact recover the geodesic distance. 

Let $x_1, x_2, \ldots, x_n$ be a sequence of $i.i.d$ random variables on $\Omega$ with density $\rho$. As in Section \ref{sec:continuum_eikonal} we assume that $\Omega\subset \R^d$ is open, bounded and connected with a $C^{1,1}$ boundary $\partial \Omega$. We assume the density $\rho$ is Lipschitz continuous and bounded above and below by positive constants
\begin{equation}
	\rho_{min}\leq \rho(x) \leq \rho_{max}
\end{equation}
for all $x \in \Omega$. The vertices of the graph are denoted by 
\begin{equation}
	\cX := \{x_1, x_2, \ldots, x_n\}.
\end{equation}
To define the edges in a random geometric graphs, we introduce a kernel $\eta : [0,\infty) \rightarrow [0,\infty)$, which is smooth and nonincreasing and satisfies $\eta(0) > 0$ and $\eta(t) = 0$ for $t > 1$. For notational convenience, we also assume $\eta$ has unit mass, so that
\[\int_{B(0,1)} \eta(|z|)\, dz =1.\]
For $\eps > 0$ define $\eta_\eps(t) := \frac{1}{\eps^d}\eta(\frac{t}{\eps})$ and set $\sigma_p := \int_{\R^d} \eta_\eps(|z|) |z_1|^p dz$. Note also that $\int_{B(0,\epsilon)} \eta_\epsilon(|z|)\, dz =1$. The normalized weight $w_{ij}$ between $x_i$ and $x_j$ is then given by
\begin{equation}\label{eq:edge_weights}
	w_{ij} = \frac{\eta_\eps(|x_i-x_j|)}{n \sigma_p \eps^{p}}.
\end{equation}  
Letting $G_{n,\epsilon}$ denote the graph with edge weights given in \eqref{eq:edge_weights}, the $p$-eikonal operator $\A_{G_{n,\epsilon},p}$ is defined in \eqref{eq:peikonal_op}, and is given by
\begin{equation}
	\cA_{G_{n,\eps},p} u(x) := \frac{1}{n \sigma_p \eps^{p}} \sum_{y \in \cX} \eta_\eps\big(|x-y|\big) \big(u(x)-u(y)\big)_+^p.
\end{equation}
For notational simplicity, we will write $\A_{n,\epsilon}=\A_{G_{n,\epsilon},p}$. 

For $p\geq 1$ we consider the $p$-eikonal equation with arbitrary right hand side $f$:
\begin{equation}\label{eq:eikonalGraphPDELP}
	\left\{\begin{aligned}
		\cA_{n,\eps} u(x) &= f(x)&&\text{if } x\in \cX \setminus \Gamma\\ 
		u(x) &=0&&\text{if } x\in \Gamma,
	\end{aligned}\right.
\end{equation}
where $\Gamma\subset \X$ is a subset of the graph nodes where the homogeneous Dirichlet condition is set. We need to assume $\Gamma$ is not too close to the topological boundary $\partial \Omega$. In particular, we assume
\begin{equation}\label{eq:Gamma}
\dist(\Gamma,\partial\Omega) \geq R
\end{equation}
where $R$ is the reach of $\partial \Omega$. Other than this, we place no assumptions on $\Gamma$. We compare this graph equation to its continuum counterpart, the state-constrained eikonal equation
\begin{equation}\label{eq:eikonalContinuumPDELP}
	\left\{\begin{aligned}
		\rho |\nabla u|^p &= f&&\text{in } \Omega \setminus \Gamma\\ 
		u &=0&&\text{on }\Gamma.
	\end{aligned}\right.
\end{equation}
We recall from Section \ref{sec:continuum_eikonal} that the solution of the state-constrained eikonal equation \eqref{eq:eikonalContinuumPDELP} is given by the geodesic distance function $u(x) = d_g(x,\Gamma)$, where $g=\rho^{-\frac{1}{p}}f^{\frac{1}{p}}$. 

Our main discrete to continuum convergence results are broken into two theorems, which are summarized below. In the theorem statements we write  $u_{n,\eps}$ for the solution \eqref{eq:eikonalGraphPDELP}.
\begin{theorem}\label{thm:main1}
There exists $C,c>0$ such that for $\epsilon$ sufficiently small and any $0 < \lambda \leq 1$ we have
\begin{equation}\label{eq:main1}
\P\left(\max_{x\in \X}(d_{g}(x,\Gamma) - u_{n,\epsilon}(x)) \leq C( \sqrt{\epsilon} + \lambda)\right) \geq 1-2n\exp(-cn\epsilon^d \lambda^2).
\end{equation}
\end{theorem}

\begin{theorem}\label{thm:main2}
There exists $C,c>0$ such that for $\epsilon$ sufficiently small and any $0 < \lambda \leq 1$ we have
\begin{equation}\label{eq:main2}
\P\left(\max_{x\in \X}(u_{n,\epsilon}(x) - d_g(x,\Gamma)) \leq C\left( \sqrt{\epsilon} + \left( n\epsilon^{p+d}\right)^{\frac{1}{p}}+\lambda\right)\right) \geq 1-3n^2\exp(-cn\epsilon^d \lambda^2).
\end{equation}
\end{theorem}
\begin{remark}\label{rem:constants}
The constants in both theorems depend on $\diam(\Omega)$, the $C^{1,1}$ bound on $\partial\Omega$ (or, equivalently, the reach of $\partial \Omega$), the kernel $\eta$ (in particular $\eta(0)$, and the constants $r\in (0,1]$ and $\mu>0$ defined in Lemma \ref{localSuperSol_1}), the dimension $d$, $\rho_{min},\rho_{max},\Lip(\rho), f_{min},f_{max}, \Lip(f)$, and $p$. The dependence on $p$ is uniform over compact sets, that is if $p\in [1,p_0]$, the constants depend only on $p_0$.
\end{remark}

\begin{remark}\label{rem:barrier}
In order for the result of Theorem \ref{thm:main2} to be non-vacuous, we require that 
\begin{equation}\label{eq:scaling}
n\epsilon^{d+p}\ll 1.
\end{equation}
For the probabilities in both Theorems \ref{thm:main1} and \ref{thm:main2} to be close to one, for arbitrarily small choices of $\lambda>0$, we require $n\epsilon^d \gg \log(n)$.  Combining these two restrictions leads to the following restrictions on $\epsilon$:
\begin{equation}\label{eq:eps_scaling}
\left( \frac{\log(n)}{n}\right)^{\frac{1}{d}}\ll \epsilon \ll \left(\frac{1}{n}\right)^{\frac{1}{p+d}}.
\end{equation}
Since $p\geq 1$, there is always room between the upper and lower bounds to select a feasible $\epsilon$. In general, we believe the upper bound is tight. Figure \ref{fig:epsrange} shows the solution of the $p$-eikonal equation with $\Gamma = \{0\}$, giving a cone-like function, for different choices of $p$ and $\eps$. When the upper bound is violated, we see a spike forming at $\Gamma$, and the solution will fail to attain the boundary condition $u=0$  on $\Gamma$ in the continuum limit (note that this spike is utilized in our Lipschitz estimate in Section \ref{sec:Alip}). We do  expect, however, that the upper bound in \eqref{eq:eps_scaling} can be relaxed if we place more assumptions on the boundary nodes $\Gamma$, so that isolated points need not be considered. In particular, if $\Gamma$ contains all points within distance $\epsilon$ of $\partial\Omega$, then the upper bound can be dropped using arguments from \cite{calder2020rates}. 

Finally, we note that there is some precedent for bandwidth restrictions like \eqref{eq:eps_scaling} in the analysis of $p$-Laplacian semi-supervised learning in the same setting of arbitrarily low label rates. In \cite{slepcev2019analysis} it was shown that $p$-Laplacian semi-supervised learning at low label rates requires the much more restrictive condition 
\[\left( \frac{\log(n)}{n}\right)^{\frac{1}{d}}\ll \epsilon \ll \left(\frac{1}{n}\right)^{\frac{1}{p}},\]
which is only true when $p>d$.
\end{remark}

\begin{figure}[!t]
\centering
%\subfloat[$\eps=0.03, p=1$]{\includegraphics[width=0.329\textwidth, clip=true, trim=85 70 90 75]{cone_eps_0.03_p_1.png}}
%\subfloat[$\eps=0.06, p=1$]{\includegraphics[width=0.329\textwidth, clip=true, trim=90 60 90 80]{cone_eps_0.06_p_1.png}}
%\subfloat[$\eps=0.09, p=1$]{\includegraphics[width=0.329\textwidth, clip=true, trim=90 60 90 80]{cone_eps_0.09_p_1.png}}\\
%\subfloat[$\eps=0.03, p=2$]{\includegraphics[width=0.329\textwidth, clip=true, trim=90 60 90 80]{cone_eps_0.03_p_2.png}}
%\subfloat[$\eps=0.06, p=2$]{\includegraphics[width=0.329\textwidth, clip=true, trim=90 60 90 80]{cone_eps_0.06_p_2.png}}
%\subfloat[$\eps=0.09, p=2$]{\includegraphics[width=0.329\textwidth, clip=true, trim=90 60 90 80]{cone_eps_0.09_p_2.png}}\\
%\subfloat[$\eps=0.03, p=4$]{\includegraphics[width=0.329\textwidth, clip=true, trim=90 60 90 80]{cone_eps_0.03_p_4.png}}
%\subfloat[$\eps=0.06, p=4$]{\includegraphics[width=0.329\textwidth, clip=true, trim=90 60 90 80]{cone_eps_0.06_p_4.png}}
%\subfloat[$\eps=0.09, p=4$]{\includegraphics[width=0.329\textwidth, clip=true, trim=90 60 90 80]{cone_eps_0.09_p_4.png}}\\
																						      \hspace{2mm}
\subfloat[$\eps=0.03, p=1$]{\includegraphics[width=0.3\textwidth]{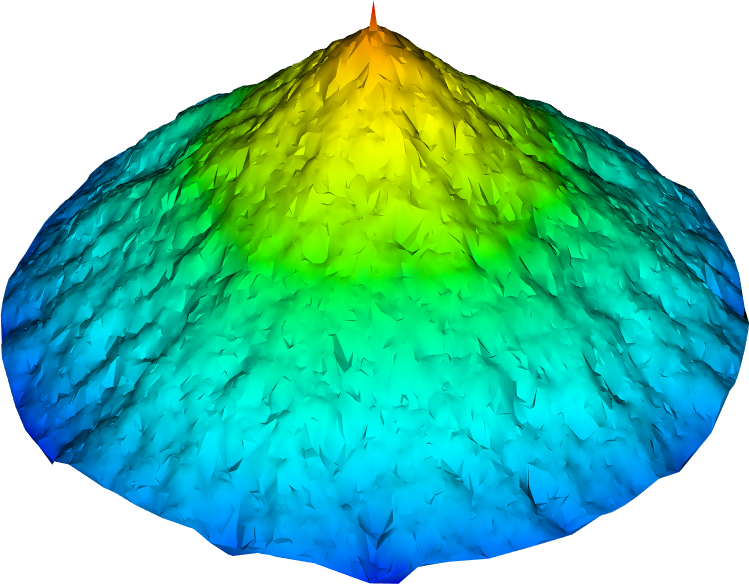}}\hspace{3mm}
\subfloat[$\eps=0.06, p=1$]{\includegraphics[width=0.3\textwidth]{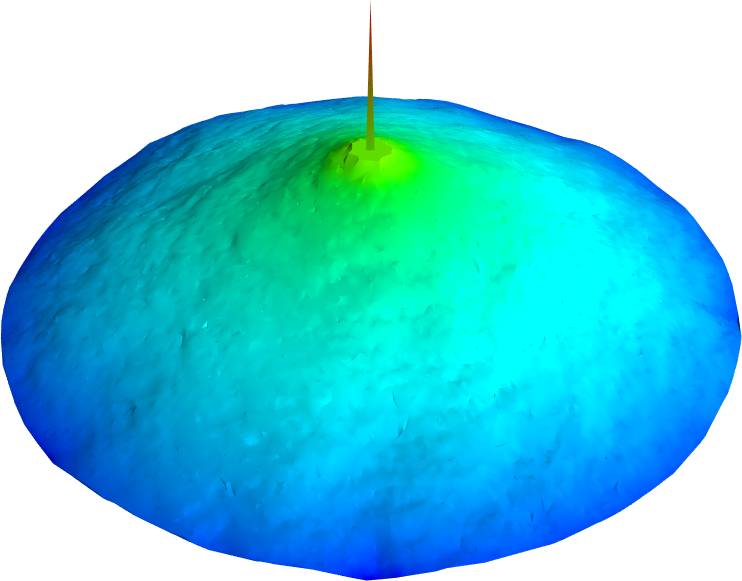}}\hspace{3mm}
\subfloat[$\eps=0.09, p=1$]{\includegraphics[width=0.3\textwidth]{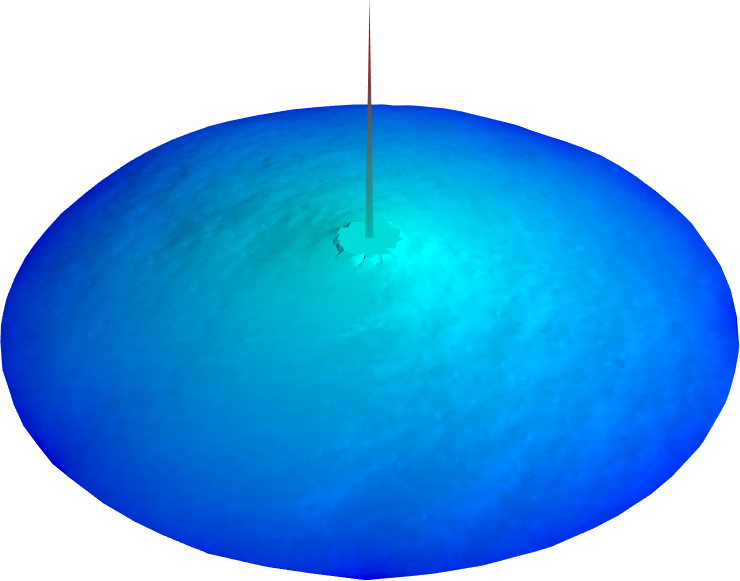}}\hspace{3mm}\\
																						      \hspace{2mm}
\subfloat[$\eps=0.03, p=2$]{\includegraphics[width=0.3\textwidth]{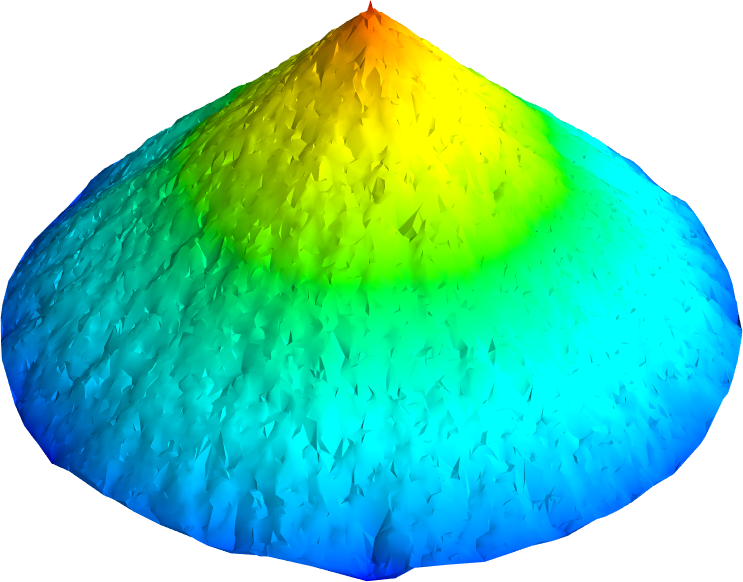}}\hspace{3mm}
\subfloat[$\eps=0.06, p=2$]{\includegraphics[width=0.3\textwidth]{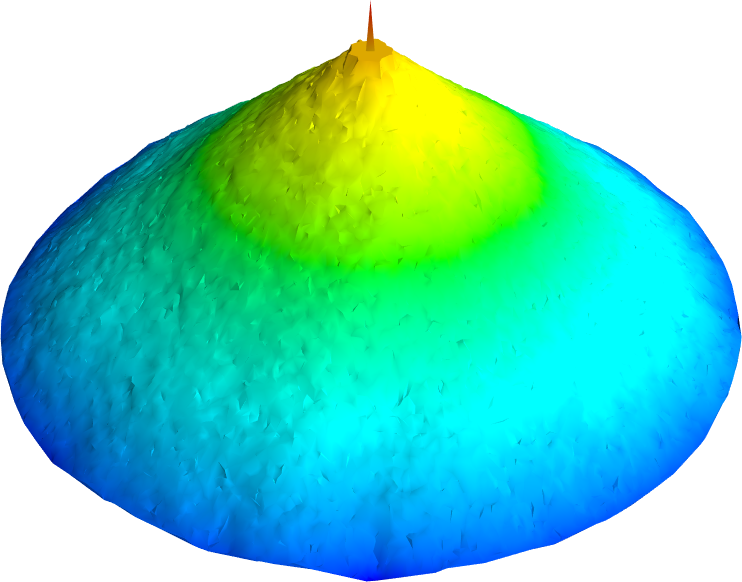}}\hspace{3mm}
\subfloat[$\eps=0.09, p=2$]{\includegraphics[width=0.3\textwidth]{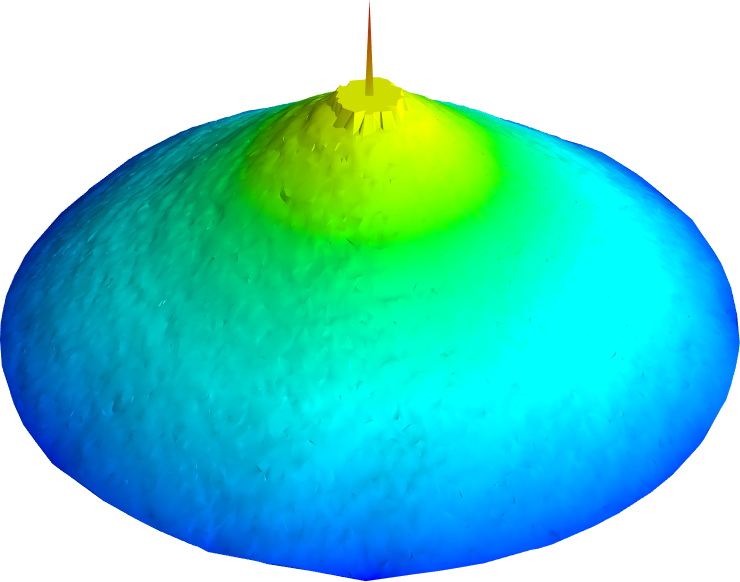}}\hspace{3mm}\\
																						      \hspace{2mm}
\subfloat[$\eps=0.03, p=4$]{\includegraphics[width=0.3\textwidth]{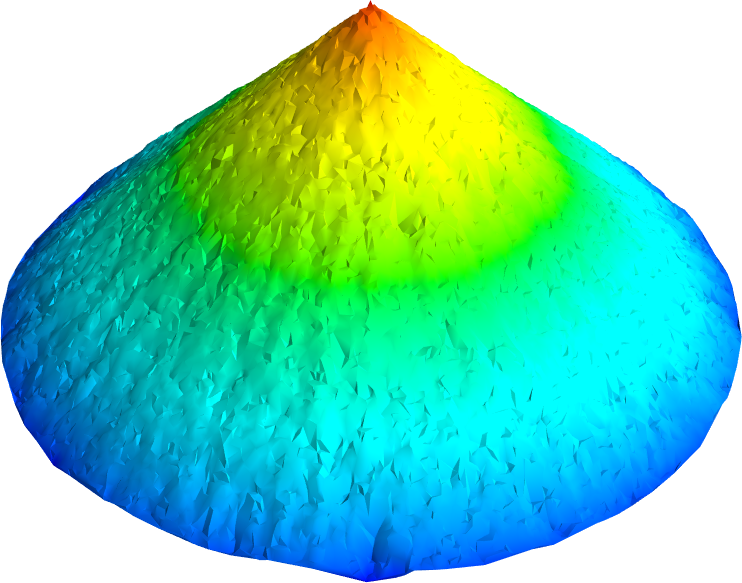}}\hspace{3mm}
\subfloat[$\eps=0.06, p=4$]{\includegraphics[width=0.3\textwidth]{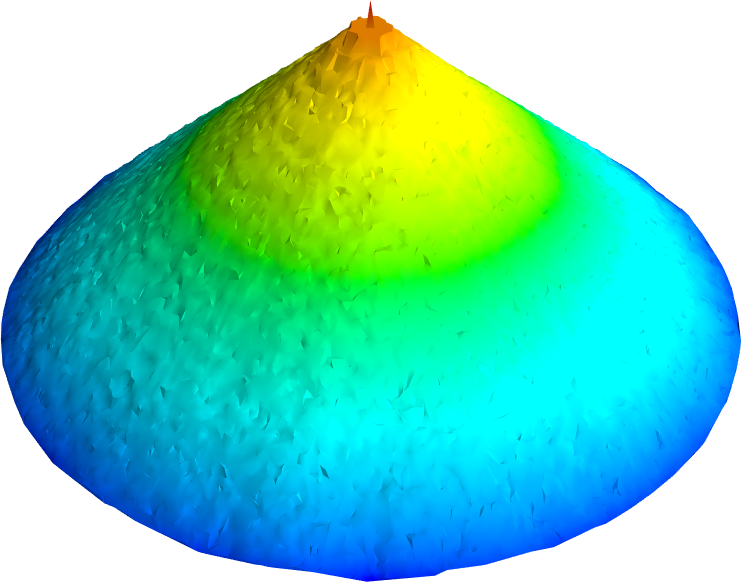}}\hspace{3mm}
\subfloat[$\eps=0.09, p=4$]{\includegraphics[width=0.3\textwidth]{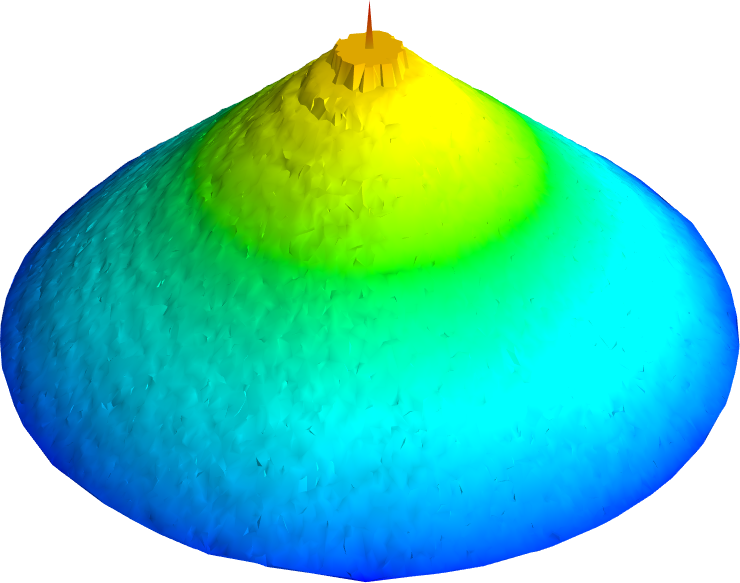}}\hspace{3mm}
\caption{For $\epsilon=0.03$, each node in the graph has on average approximately $20$ neighbors, while for $\epsilon=0.06$ and $\epsilon=0.09$ each node has on average $70$ and $160$ neighbors, respectively. The cones are inverted for a better viewing angle.}
\label{fig:epsrange}
\end{figure}

\begin{remark}\label{rem:rate}
If we choose $\lambda= \sqrt{\epsilon}$ in Theorems \ref{thm:main1} and \ref{thm:main2}, then we obtain that the convergence rate
\[\max_{x\in \X}|u_{n,\epsilon}(x) - d_g(x,\Gamma)| \leq C\left( \sqrt{\epsilon} + \left( n\epsilon^{p+d}\right)^{\frac{1}{p}}\right)\]
holds with probability at least $1 - 5n^2\exp\left( -cn\epsilon^{d+1}\right)$. If we additionally choose $\epsilon$ so that $\left(n\epsilon^{p+d}\right)^{\frac{1}{p}} \leq  \sqrt{\epsilon}$, that is, we make the restriction
\[\epsilon \leq \left( \frac{1}{n}\right)^{\frac{2}{p+2d}},\]
then the rate
\[\max_{x\in \X}|u_{n,\epsilon}(x) - d_g(x,\Gamma)| \leq C \sqrt{\epsilon}\]
holds with the same probability. Without any further assumptions on the boundary set $\Gamma$, we expect the $\cO(\sqrt{\epsilon})$ rate is optimal, in this general setting.
\end{remark}

\subsection{Pointwise consistency}

The first ingredient for a discrete to continuum limit result is pointwise consistency for the operator $\cA_{n,\epsilon}$. As usual, pointwise consistency passes through a nonlocal operator, which in this case has the form 
\begin{equation}
	\label{eq:nonlocalOperator}
	\cA_\eps u(x) := \frac{1}{\sigma_p \eps^{p}} \int_{\Omega} \eta_\eps(|x-y|) \big(u(x)-u(y)\big)_+^p \rho(y) dy,
\end{equation}
for $u\in L^\infty(\Omega)$. Pointwise consistency is obtained in two steps, the first step (Lemma \ref{discreteNonlocal}) passes from the discrete operator $\cA_{n,\epsilon}$ to the nonlocal counterpart $\cA_\epsilon$ via concentration of measure, while the second (Lemma \ref{nonLocalLocal}) uses Taylor expansion to relate the nonlocal operator to the eikonal equation.
\begin{lemma}(Discrete to nonlocal)
	\label{discreteNonlocal}
	Let $u : \bar{\Omega} \rightarrow \R$ be Lipschitz continuous and $n\geq 2$. Then for any $\lambda >0$ we have that
	\begin{equation}
		\label{eq:discreteToNonlocal}
		\max_{x \in \cX} |\cA_{n,\eps}u(x) - \cA_\eps u(x)| \leq \eta(0)\rho_{max}\Lip(u)^p \lambda 
	\end{equation}
holds with probability at least 
\begin{equation}\label{eq:prob}
1 - 2n\exp\left( \frac{-\eta(0)\sigma_p^2\rho_{max}n\epsilon^d \lambda^2}{4\left( 1 + \frac{1}{3}\sigma_p\eta(0) \lambda\right)}\right).
\end{equation}
\end{lemma}
\begin{remark}\label{rem:parametersPC}
We note that to ensure the probability in \eqref{eq:prob} is close to $1$, when $\lambda>0$ can be arbitrarily small,  we require that
\[n\epsilon^d \gg \log(n).\]
This is the same restriction required for graph connectivity in random geometric graphs \cite{penrose2003random} (more correctly, the restriction for graph connectivity is $n\epsilon^d \geq C\log(n)$ for a large enough constant $C$). In contrast, pointwise consistency for graph Laplacians requires a more restrictive length scale restriction of the form $n\epsilon^{d+2}\gg \log(n)$ (see, e.g., \cite{calder2020rates}), which does not cover smaller bandwidths $\epsilon$ where the graph is still connected. The reason for this difference is that graph Laplacians are second order differential operators, and are normalized by an additional factor of $\epsilon$ to obtain meaningful continuum limits. 
\end{remark}

\begin{proof}[Proof of Lemma \ref{discreteNonlocal}]
	Fix $x \in \Omega$ and let $Y_i := \eta_\eps(|x-x_i|)(u(x)-u(x_i))_+^p$
	so that
	\begin{equation}
		\cA_{n,\eps} u(x) = \frac{1}{\sigma_p \eps^{p}} \frac{1}{n} \sum_{i=1}^{n} Y_i
	\end{equation}
	Then we compute
	\begin{equation}
		\E(Y_i) = \int_{\Omega} \eta_\eps(|x-y|)\big(u(x)-u(y)\big)_+^p \rho(y) dy
	\end{equation}
	and
	\begin{align}\label{eq:sigmaBernIneq}
			\sigma^2 \leq \E(Y_i^2) &= \int_{\Omega\cap B(x,\eps)} \eta_\eps^2(|x-y|)\big(u(x)-u(y)\big)_+^{2p} \rho(y) dy \\
			&\leq \rho_{\max}\Lip(u)^{2p} \eps^{2p} \int_{B(x,\eps)} \eta^2_\eps(|x-y|) dy \notag \\
			&\leq \eta(0)\rho_{max}\Lip(u)^{2p} \eps^{2p-d} \int_{B(x,\epsilon)} \eta_\epsilon(|x-y|) dy \notag \\
			&=\eta(0) \rho_{max}\Lip(u)^{2p} \eps^{2p-d}. \notag
	\end{align}
	We also compute
	\begin{equation}
		\begin{split}
			|Y_i| =  \eta_\eps(|x-x_i|)|(u(x)&-u(x_i))_+^p|
			\leq \eta_\eps(|x-x_i|) |u(x)-u(x_i)|^p \leq \eta(0)
			\Lip(u)^p \eps^{p-d} 
		\end{split}
	\end{equation}
We now invoke Bernstein's inequality (see Appendix \ref{sec:conc})  to obtain
	\begin{equation}
		\Big|\frac{1}{n} \sum_{i=1}^n Y_i -\int_{\Omega} \eta_\eps(|x-y|) \big(u(x)-u(y)\big)_+^p \rho(y) dy  \Big| \leq t
	\end{equation}
	holds with probability at least 
\[1 - 2\exp\left( \frac{-nt^2}{2\eta(0)\Lip(u)^p \eps^{p-d}\left(\rho_{max} \Lip(u)^p \eps^p + \frac{t}{3}\right)}\right).\]
Setting $t =\eta(0)\rho_{max}\sigma_p \Lip(u)^p \eps^p \lambda$ for a new parameter $\lambda>0$ we obtain 
\[|\cA_{n,\eps}u(x) - \cA_\eps u(x)| \leq \eta(0)\rho_{max}\Lip(u)^p \lambda \]
with probability at least 
\[1 - 2n\exp\left( \frac{-\eta(0)\sigma_p^2\rho_{max}n\epsilon^d \lambda^2}{2\left( 1 + \frac{1}{3}\sigma_p\eta(0) \lambda\right)}\right).\]

The rest of proof is completed by conditioning on $x_i$ and then applying a union bound.  Indeed, conditioning on $x_i=x$, the other $n-1$ points form an \emph{i.i.d.}~sequence, and we note that
\[\cA_{n,\eps} u(x_i) := \frac{1}{n \sigma_p \eps^{p}} \sum_{j\neq i} \eta_\eps\big(|x-x_j|\big) \big(u(x_i)-u(x_j)\big)_+^p\]
is exactly in the form considered above, except the sum is over $n-1$ \emph{i.i.d.}~random variables, instead of $n$.
Thus, we can apply the argument above, replacing $n$ with $n-1$, and then union bounding over $i=1,\dots,n$. To simplify the probability we use the bound $n-1\geq n/2$ for $n\geq 2$.
\end{proof}
\begin{remark}\label{rem:conditioning}
The conditioning argument used at the end of the proof is standard in probability, and we will use it implicitly in subsequent proofs.
\end{remark}

We now turn to comparing the nonlocal operator $\cA_\eps$ to its continuum counterpart $\rho|\nabla u|^p$.  
\begin{lemma}(Nonlocal to local)
	\label{nonLocalLocal}
	There exists $C > 0$ such that for every $\epsilon>0$, $p\geq 1$ and $\phi \in C^2(\R^d)$, the following hold.
\begin{enumerate}[{\rm (i)}]
\item If $\dist(x,\partial\Omega) \geq \epsilon$ then 
\[\Big| \cA_\eps \phi(x) - \rho(x)|\nabla \phi(x)|^p \Big| \leq C M\epsilon.\]
where
\[M := \sigma_p^{-1}\|\rho\|_{C^{0,1}}\left(p(\Lip(\phi) + \|\phi\|_{C^2}\epsilon)^{p-1}  \|\phi\|_{C^2} + 1\right).\]
\item If $\dist(x,\partial\Omega) < \epsilon$ then 
\[ \cA_\eps \phi(x) - \rho(x)|\nabla \phi(x)|^p  \leq C M\epsilon.\]
\end{enumerate}

\end{lemma}
\begin{proof}
We first prove (i). Since $B(x,\epsilon)\subset \Omega$, we make the change of variables $z := (y-x)/\eps$ in the nonlocal operator \eqref{eq:nonlocalOperator} and obtain
	\begin{equation}
		\label{eq:fullIntgerationBall}
		\cA_\eps \phi(x) = \frac{1}{\sigma_p \eps^{p}} \int_{B(0,1)} \eta(|z|) \big(\phi(x)-\phi(x + \eps z)\big)_+^p \rho(x + \eps z) dz
	\end{equation}
	Using the Taylor expansion
	\begin{equation}
		\rho(x + z \eps) = \rho(x) + \cO(\Lip(\rho)\epsilon)
	\end{equation}
	for $|z| \leq 1$ we have
		\[\cA_\eps \phi(x) = \frac{1}{\sigma_p \eps^{p}} \int_{B(0,1)} \eta(|z|) \big(\phi(x)-\phi(x + \eps z)\big)_+^p \rho(x) dz + \cO(\sigma_p^{-1}\Lip(\rho)^p\epsilon)\]
	We now use the Taylor expansion
	\begin{equation}
		\phi(x)-\phi(x + z \eps) = \eps z \cdot \nabla \phi(x) + \cO(\|\phi\|_{C^2}\eps^2)
	\end{equation}
	to obtain
	\begin{equation}
		\cA_\eps \phi(x) = \frac{1}{\sigma_p} \int_{B(0,1)} \eta(|z|) \big(z \cdot \nabla \phi(x) + \cO(\|\phi\|_{C^2}\eps)
\big)_+^p \rho(x)dz   + \cO(\sigma_p^{-1}\Lip(\rho)\epsilon).
	\end{equation}
	We make the change of variables $y = Az$ for an orthogonal matrix $A$ such that $A \nabla \phi(x) = |\nabla \phi(x)| e_d$. Then we have that $z \cdot \nabla \phi(x) = Az \cdot A \nabla \phi(x) = |\nabla \phi(x)| y_d$ and thereby
	\begin{equation}
		\cA_\eps \phi(x) = \frac{1}{\sigma_p} \int_{B(0,1)} \eta(|z|) \big(|\nabla \phi(x)|y_d + \cO(\|\phi\|_{C^2} \eps)\big)_+^p \rho(x)dz  + \cO(\sigma_p^{-1}\Lip(\rho)\epsilon).
	\end{equation}
	We now use the bound
	\[|(a+t)_+^p - a_+^p| \leq p(|a|+|t|)^{p-1}|t|,\]
	for $a,t\in \R$ and $p\geq 1$, which follows from Taylor expansion, to obtain 
	\[\big(|\nabla \phi(x)|y_d + \cO(\|\phi\|_{C^2} \eps)\big)_+^p = |\nabla \phi(x)|^p (y_d)_+^p + \cO\left(p(\Lip(\phi) + \|\phi\|_{C^2}\epsilon)^{p-1}  \|\phi\|_{C^2} \eps\right).\]
	Substituting this above, we have
	\begin{equation}
		\label{eq:nonlocalExpansion}
		\cA_\eps \phi(x) = \rho(x) |\nabla \phi(x)|^p \frac{1}{\sigma_p} 
		\int_{B(0,1)} \eta(|y|) (y_d)_+^p dy + R,
	\end{equation}
	where 
	\[|R| \leq C\sigma_p^{-1}\left(p(\Lip(\phi) + \|\phi\|_{C^2}\epsilon)^{p-1}  \|\phi\|_{C^2} \rho_{max} + \Lip(\rho)\right)\epsilon. \]
	Applying the definition of $\sigma_p$ and using the bound $\rho_{max},\Lip(\rho)\leq \|\rho\|_{C^{0,1}}$ completes the proof of (i).

The proof of (ii) proceeds in a similar way, except that on the first step, since $B(x,\epsilon) \cap \partial\Omega \neq \varnothing$, the change of variables $z=(y-x)/\epsilon$ yields
\[\cA_\eps \phi(x) \leq \frac{1}{\sigma_p \eps^{p}} \int_{B(0,1)} \eta(|z|) \big(\phi(x)-\phi(x + \eps z)\big)_+^p \bar{\rho}(x + \eps z) dz,\]
where $\bar{\rho}$ is any extension of $\rho$ to $\R^d$ that preserves its Lipschitz constant. The proof then proceeds in the same way as (i).
\end{proof}

\subsection{Lipschitz regularity}
\label{sec:Alip}

Since we allow for general closed Dirichlet boundary sets $\Gamma \subset \Omega$ in our discrete to continuum framework, our results require an \emph{a priori} Lipschitz bound for the discrete solutions of \eqref{eq:eikonalGraphPDELP}. In this section we prove a Lipschitz estimate with the barrier method. The first ingredient is a lower bounds on the volume of set $B_\Omega(x,r)\cap \Omega$. 
\begin{proposition}\label{prop:ball_measure}
For $r>0$ sufficiently small, depending only on $\partial\Omega$, we have 
\[|B_\Omega(x,r)\cap \Omega| \geq c_d r^d \ \ \text{for all } x\in \bar{\Omega},\]
where
\[c_d =\frac{\omega_{d-1}}{2^{\frac{3d+1}{2}}(d+1)},\] 
and $\omega_d=|B(0,1)|$ denotes the volume of the unit ball in $\R^d$.
\end{proposition}
We postpone the proof of Proposition \ref{prop:ball_measure} to Appendix \ref{sec:proofs}, and proceed to define our barrier function for the Lipschitz estimate. For $y\in \R^d$ we define  
\[\delta_y(x) = 
\begin{cases}
1,& \text{if } y=x\\
0,& \text{otherwise.} 
\end{cases}\] 
Our barrier function will be a \emph{geodesic} cone with a jump (or spike) at the origin. In particular, we define
\[v_{\beta,y}(x) := \beta(1-\delta_y(x)) + d_\Omega(x,y)\]
for $\beta>0$ to be determined. We refer the reader to Figure \ref{fig:epsrange} for an illustration of the barrier, for different size spikes (though the cones are inverted in the figure). The following lemma establishes the basic supersolution properties of our barrier function.  
\begin{lemma}\label{localSuperSol_1}
Let $y\in \bar{\Omega}$ and $\beta>0$.	Let $r \in (0,1]$ and $\mu>0$, such that $\eta(|t|) \geq \mu > 0$ for all $|t| \leq r$, and let $c_d$ be the constant from Proposition \ref{prop:ball_measure}. Then for $\epsilon$  sufficiently small, depending only on $\partial \Omega$, the following results hold:
\begin{enumerate}[{\rm (i)}]
\item For $x \in \bar{\Omega} \setminus B(y, r \epsilon)$ it holds that
\begin{equation}
\bP\left(\cA_{n,\eps} v_{\beta,y}(x) \geq \frac{c_d\mu r^{d+p}}{\sigma_p 2^{2d+p+1}}\right) \geq 1-\exp\left( -\tfrac{c_dr^d}{2^{2d+3}}\rho_{min}n\epsilon^d\right).
\end{equation}
\item For $x \in \bar{\Omega} \cap B(y, r \epsilon) \setminus \{y\}$ we have
\begin{equation}
\cA_{n,\eps} v_{\beta,y}(x) \geq \frac{\mu \beta^p}{ \sigma_p n\eps^{p+d}}.
\end{equation}
\end{enumerate}
\end{lemma}
\begin{proof}
	We will prove the two cases above separately. 

(i) Assume $x\in \bar{\Omega}\setminus B(y,r\epsilon)$ and let us define 
		\[D := \left\{z\in B(x,r \eps) \, : \, d_\Omega(x,y) - d_\Omega(y,z) \geq \frac{r\epsilon}{2}\right\}.\]
		Since $x\neq y$ we compute
		\begin{equation}
			\begin{split}
				\cA_{n,\eps} v_{\beta,y}(x) = &\frac{1}{n \sigma_p \eps^p} \sum_{z \in \cX} \eta_\eps(|x-z|) \big(\beta + d_\Omega(x,y)- \beta(1-\delta_y(z)) - d_\Omega(y,z)\big)_+^p \\
				&\geq \frac{\mu }{n \sigma_p \eps^{p+d}} \sum_{z \in \cX \cap B(x,r \eps)} \big(d_\Omega(x,y)-d_\Omega(y,z)\big)_+^p \\
				&\geq \frac{\mu }{n \sigma_p \eps^{p+d}} \sum_{z \in \cX \cap D} \left(\frac{r \eps}{2}\right)_+^p\\
				&= \frac{\mu r^p}{2^pn\sigma_p \eps^{d}} \# ( \cX \cap D).
			\end{split}
		\end{equation}
To bound the number of points in $D\cap \cX$, we use the Chernoff bound (see Appendix \ref{sec:conc}), which produces the lower bound
		\begin{equation}\label{eq:ibound}
			\cA_{n,\eps} v_{\beta,y}(x) \geq \frac{\mu r^p}{2^{p+1}\sigma_p \eps^{d}} |D\cap \Omega| 
		\end{equation}
		with probability at least $1-\exp\left( -\tfrac{1}{8}\rho_{min}|D\cap \Omega|n\right)$. 

		We need to lower bound $|D\cap \Omega|$ to complete the proof. There exists $z_*\in \partial B(x,\frac{3r\epsilon}{4})$ so that
		\[d_\Omega(x,y) = d_\Omega(x,z_*) + d_\Omega(z_*,y).\]
		Since $d_\Omega(x,z_*) \geq |x-z_*| = \frac{3r\epsilon}{4}$ this becomes
		\[d_\Omega(x,y) - d_\Omega(y,z_*) \geq \frac{3r\epsilon}{4}.\]
		It follows that $B_\Omega(z_*,\frac{r\epsilon}{4})\subset D$. Indeed, if $d_\Omega(z,z_*) \leq \frac{r\epsilon}{4}$ then by the triangle inequality we have
		\[d_\Omega(x,y) - d_\Omega(y,z) \geq d_\Omega(x,y) - d_\Omega(y,z_*) -   d_\Omega(z,z_*)\geq \frac{3r\epsilon}{4} - \frac{r\epsilon}{4} = \frac{r\epsilon}{2}.\]
		Invoking Proposition \ref{prop:ball_measure} we have
		\[|D\cap \Omega| \geq |B_\Omega(z_*,\tfrac{r\epsilon}{4}) \cap \Omega|  \geq c_d \left( \frac{r\epsilon}{4}\right)^d,\]
		for $\epsilon$ sufficiently small. Combining this with \eqref{eq:ibound} completes the proof of (i).
		
(ii) Let $x \in \bar{\Omega} \cap B(y,r \eps) \setminus \{y\}$, and compute
		\begin{equation}
			\begin{split}
				\cA_{n,\eps} v_{\beta,y}(x) &\geq \frac{\eta_\eps(|x-y|)}{ \sigma_p n\eps^p}   \big(v_{\beta,y}(x) -  v_{\beta,y}(y)\big)_+^p  
				\\
				&\geq \frac{\mu}{ \sigma_p n\eps^{p+d}}(\beta + d_\Omega(x,y) - d_\Omega(y,y))_+^p\\
				&\geq \frac{\mu \beta^p}{ \sigma_p n\eps^{p+d}},
			\end{split}
		\end{equation} 
which completes the proof.
\end{proof}

We are now equipped to prove global Lipschitzness for the $p$-eikonal equation. The proof is based on the barrier method, using the barrier studied in Lemma \eqref{localSuperSol_1}. 
\begin{theorem}
\label{thm:discreteSolLip}
Let $u$ be the solution of \eqref{eq:eikonalGraphPDELP}. Let $c_d$, $r$, and $\mu$ be as defined in Lemma \ref{localSuperSol_1}. Define 
\[\gamma_p = \left( \frac{c_d r^{d+p}}{2^{2d+p+1}}\right)^{\frac{1}{p}} \ \ \text{and} \ \ c_p =\left(\frac{\sigma_p}{\mu}\right)^{\frac{1}{p}}.\]
Then it holds with probability at least $1-n^2\exp\left( -\tfrac{c_dr^d}{2^{2d+3}}\rho_{min}n\epsilon^d\right)$ that
\begin{equation}
|u(x) - u(y)| \leq c_p \gamma_p^{-1}\max_{\X}f^{\frac{1}{p}}\,d_\Omega(x,y) + \gamma_p \left( n\epsilon^{p+d}\right)^{\frac{1}{p}}, \ \ \text{for all } x,y\in \X.
\end{equation}
\end{theorem}
\begin{proof}
We choose $\beta$ in Lemma \ref{localSuperSol_1} to satisfy
\begin{equation}\label{eq:betachoice}
\beta^p = \frac{c_d r^{d+p}}{2^{2d+p+1}}n\epsilon^{p+d} = \gamma_p^p n\epsilon^{p+d},
\end{equation}
and we set $v_y=v_{\beta,y}$. Then by Lemma \ref{localSuperSol_1} and a union bound, we have that
\begin{equation}\label{eq:supersol_barrier}
\cA_{n,\eps} v_{y}(x) \geq \frac{\mu \gamma_p^p}{\sigma_p} \ \ \text{for all }x,y\in \X, x\neq y, 
\end{equation}
holds with probability at least $1-n^2\exp\left( -\tfrac{c_dr^d}{2^{2d+3}}\rho_{min}n\epsilon^d\right)$.   For the rest of the proof we assume this event holds.

Let us define
\[C = \left( \frac{\sigma_p}{\mu \gamma_p^p}\right)^{\frac{1}{p}}\max_{\X}f^{\frac{1}{p}}.\]
Then since $\cA_{n,\epsilon}$ is $p$-homogeneous we have
\[\cA_{n,\eps} (Cv_{y})(x) = C^p\cA_{n,\epsilon} v_y(x) \geq \max_{\X}f \geq \cA_{n,\epsilon} u(x),\]
for all $x,y\in \X$ with $x\neq y$. Therefore, $Cv_y$ is a supersolution, relative to the function $w(x):=u(x) - u(y)$ on the set $\X\setminus (\Gamma \cup \{y\})$. Furthermore, $w(y) = u(y)-u(y) = 0\leq Cv_y(y)$ and for $x\in \Gamma$ we have $w(x) = u(x)-u(y) \leq 0-u(y) \leq 0 \leq v_y(x)$. Thus, by the comparison principle (Lemma \ref{lem:comparison}) we have that $u(x) -u(y) \leq Cv_y(x)$ for all $x,y\in \X$ with $x\neq y$, which becomes
\[u(x) - u(y) \leq \left( \frac{\sigma_p}{\mu \gamma_p^p}\right)^{\frac{1}{p}}\max_{\X}f^{\frac{1}{p}}\,d_\Omega(x,y) + \beta.\]
Substituting the definition of $\beta$, and reversing the role of $x$ and $y$ to get an absolute value bound, completes the proof.
\end{proof}
\begin{remark}\label{rem:lip}
Similar to Lemma \ref{lem:lip}, we can use the bound $d_\Omega(x,y) \leq C|x-y|$ to obtain that the solution $u$ of \eqref{eq:eikonalGraphPDELP} satisfies
\[|u(x)-u(y)| \leq C\left(|x-y|  + \left( n\epsilon^{p+d}\right)^{\frac{1}{p}}\right),\]
with probability at least $1-n^2\exp\left( -cn\epsilon^d\right)$, where $C$ and $c$ are constants whose precise values are given in Theorem \ref{thm:discreteSolLip}. 
\end{remark}

\subsection{Discrete to continuum convergence}

We now proceed to prove our main discrete to continuum convergence results. The results are split into two theorems. Throughout the proof of Theorems \ref{thm:main1} and \ref{thm:main2}, we use the convention that $0 \leq c \leq 1$ and $C\geq 1$ denote arbitrary constants, whose value can change from line to line, to reduce the notational burden. 
\begin{proof}[Proof of Theorem \ref{thm:main1}]
For $0 < \delta \leq c$, where $c>0$ is given in Theorem \ref{thm:domain_perturbation},  let $u_\delta$ denote the viscosity solution of \eqref{eq:state_constrained_eikonal_delta} over the perturbed domain $\Omega_{-\delta} \setminus \Gamma$, defined in Theorem \ref{thm:domain_perturbation}, except with $g$ in place of $f$ on the right hand side.  For $0 < \theta < 1$ and $1 \leq \alpha \leq \epsilon^{-1}$ we define the auxiliary function 
\begin{equation*}
\Phi(x,y) := (1-\theta)u_\delta(x) - u_{n,\eps}(y) - \frac{\alpha}{2}|x-y|^2, \quad \quad  (x,y)\in \bar{\Omega_{-\delta}} \times \cX.
\end{equation*}
Let $(x_\alpha,y_\alpha) \in \bar{\Omega_{-\delta}} \times \cX$ be a point at which $\Phi$ is maximized over $\bar{\Omega_{-\delta}} \times \cX$. To see why the auxiliary function is useful, we first note that the inequality
\[(1-\theta)u_\delta(x) - u_{n,\epsilon}(x) \leq \Phi(x,x)\]
implies that
\[\max_{\X}\left((1-\theta)u_\delta - u_{n,\epsilon} \right) \leq \max_{x\in \X}\Phi(x,x) \leq \Phi(x_\alpha,y_\alpha).\]
We also have 
\[\max_{\X}(u_\delta - u_{n,\epsilon}) \leq \max_{\X}\left((1-\theta) u_\delta - u_{n,\epsilon}\right) + C\theta,\]
and by Theorem \ref{thm:domain_perturbation} (ii) we have $|u-u_\delta| \leq C\delta$, with $C$ depending on $f$  and $g$, where $u(x)=u_0(x)=d_g(x,\Gamma)$. Therefore, we obtain the bound
\begin{equation}\label{eq:aux_bound}
\max_{\X}(u - u_{n,\epsilon}) \leq \Phi(x_\alpha,y_\alpha) + C(\theta + \delta).
\end{equation}
Thus, we will obtain an error estimate on $u-u_{n,\epsilon}$ by estimating $\Phi(x_\alpha,y_\alpha)$, while choosing the parameters $\theta$ and $\delta$ as small as possible, and optimizing over $\alpha$.

Since $\Phi(x_\alpha,y_\alpha) \geq \Phi(y_\alpha,y_\alpha)$, we have 
\begin{equation}
(1-\theta)u_\delta(x_\alpha) - u_{n,\eps}(y_\alpha) - \frac{\alpha}{2}|x_\alpha-y_\alpha|^2 \geq (1-\theta)u_\delta(y_\alpha) - u_{n,\eps}(y_\alpha).
\end{equation}
By Theorem \ref{thm:domain_perturbation}  (i), $u_\delta$ is Lipschitz continuous, and so
\begin{equation}
\frac{\alpha}{2}|x_\alpha-y_\alpha|^2 \leq (1-\theta)(u_\delta(x_\alpha) - u_\delta(y_\alpha)) \leq C |x_\alpha - y_\alpha|.
\end{equation}
Hence we have the bound 
\begin{equation}
\label{eq:xnynBound}
|x_\alpha - y_\alpha| \leq C\alpha^{-1}.
\end{equation}
Thus, for $\alpha> C\delta^{-1}$, we have $|x_\alpha-y_\alpha|<\delta$ and so $x_\alpha\in \Omega_{-\delta}$, since $y_\alpha\in \Omega$. We assume $\alpha>C\delta^{-1}$ throughout the rest of the proof.

We now have several cases to consider.

(i) If $y_\alpha \in \Gamma$, then $u_{n,\epsilon}(y_\alpha)=0 = u_\delta(y_\alpha)$ and so 
\begin{equation}
\begin{split}
u_{\delta}(x_\alpha) - u_{n,\eps}(y_\alpha) = u_\delta(x_\alpha) - u_\delta(y_\alpha) \leq C|x_\alpha-y_\alpha| \leq C\alpha^{-1}.
\end{split}
\end{equation}
Therefore
\[\Phi(x_\alpha,y_\alpha) \leq u_\delta(x_\alpha) - u_{n,\epsilon}(y_\alpha) \leq C\alpha^{-1}.\]

(ii) If $x_\alpha \in \Gamma$, then $u_{\delta}(x_\alpha) = 0=u_{n,\eps}(x_\alpha)$ and thus
\begin{equation}
u_{\delta}(x_\alpha) - u_{n,\eps}(y_\alpha) = u_{n,\eps}(x_\alpha) - u_{n,\eps}(y_\alpha) \leq 0,
\end{equation}
since $u_{n,\epsilon}\geq 0$. In this case we have $\Phi(x_\alpha,y_\alpha)\leq 0$.

(iii) We now consider the case of interior maxima; in particular, that $x_\alpha \in \Omega_{-\delta} \setminus \Gamma$ and $y_\alpha \in \cX\setminus \Gamma$. Our plan is to choose the parameter $\theta$ so that interior maxima are impossible, and so this case need not be considered when estimating $\Phi(x_\alpha,y_\alpha)$. We first note that the map 
\[x\mapsto u_\delta(x) - \frac{\alpha}{2}(1-\theta)^{-1}|x-y_\alpha|^2\]
attains its maximum at $x_\alpha$ over the open set $\Omega_{-\delta}$. Using $\phi(x)=\frac{\alpha}{2}(1-\theta)^{-1}|x-y_\alpha|^2$ as a test function for the definition of viscosity subsolution for $u_\delta$, we have
\begin{equation}
\label{subSol}
|p_\alpha|  \leq (1-\theta)g(x_\alpha),
\end{equation}
where $p_\alpha = \alpha(x_\alpha-y_\alpha)$. Likewise, the map $y \mapsto u_{n,\eps}(y) + \frac{\alpha}{2 }|x_\alpha-y|^2$ attains its minimum at $y_\alpha \in \cX$ over the point cloud $\cX$. Setting $\psi(y) := -\frac{\alpha}{2}|x_\alpha - y|^2$, we see that the inequality 
\[u_{n,\eps}(y_\alpha) - u_{n,\eps}(y) \leq \psi(y_\alpha) - \psi(y)\]
holds for all $y \in \cX$. It follows that
\[f(y_\alpha) = \cA_{n,\epsilon}u_{n,\epsilon}(y_\alpha) \leq \cA_{n,\epsilon}\psi(y_\alpha).\]
Using pointwise consistency (Lemmas \ref{discreteNonlocal} and \ref{nonLocalLocal}), and noting that $\Lip(\psi)\leq C$, $\|\psi\|_{C^2}\leq C\alpha$, and $\nabla \psi(y_\alpha)=\alpha(x_\alpha-y_\alpha)$, we obtain that
\[f(y_\alpha) \leq \rho(y_\alpha)|p_\alpha|^p + C(\alpha\epsilon + \lambda),\]
holds for any $0 < \lambda \leq 1$ with probability at least $1-2n\exp\left( -cn\epsilon^d \lambda^2\right)$, where $C$ depends on $p$, $\sigma_p$, $\|\rho\|_{C^{0,1}}$, $\eta(0)$, and $\rho_{max}$, and $c$ depends on $\eta(0)$, $\sigma_p$, and $\rho_{max}$. Dividing by $\rho$ on both sides, and combining with \eqref{subSol} yields
\[g(y_\alpha)^p \leq (1-\theta)^p g(x_\alpha)^p + C(\alpha\epsilon + \lambda).\]
Since $(1-\theta)^p \leq 1-\theta$, and $g^p$ is Lipschitz, we can rearrange this and use \eqref{eq:xnynBound} to obtain
\[\theta g(x_\alpha) - C(\alpha\epsilon + \lambda) \leq g(x_\alpha)^p - g(y_\alpha)^p \leq C|x_\alpha-y_\alpha| \leq C\alpha^{-1}.\]
Since $g$ is bounded below by a positive constant, this yields
\[\theta \leq C(\alpha^{-1}+ \alpha\epsilon + \lambda).\]
Hence, we set
\[\theta = (C+1)(\alpha^{-1}+ \alpha\epsilon + \lambda),\]
so that case (iii) cannot hold. 

The proof is completed by noting that cases (i) and (ii) yield $\Phi(x_\alpha,y_\alpha)\leq C\alpha^{-1}$, and so \eqref{eq:aux_bound} yields
\[\max_{\X}(u-u_{n,\epsilon}) \leq C(\alpha^{-1}+ \alpha\epsilon + \delta + \lambda).\]
Optimizing over $\alpha$ yields $\alpha= \frac{1}{ \sqrt{\epsilon}}$. We also made the restriction $\alpha \geq C\delta^{-1}$ earlier, so we choose $\delta \geq C \sqrt{\epsilon}$. Recalling that $u(x)=d_g(x,\Gamma)$ (see Theorem \ref{thm:existence_state_const}), the proof is complete.
\end{proof}

\begin{proof}[Proof of Theorem \ref{thm:main2}]
The proof is similar to Theorem \ref{thm:main1}, so we sketch the main differences here. For $0 < \delta \leq c$, where $c>0$ is given in Theorem \ref{thm:domain_perturbation},  let $u_\delta$ denote the viscosity solution of \eqref{eq:state_constrained_eikonal_delta} over the perturbed domain $\Omega_{\delta} \setminus \Gamma$, defined in Theorem \ref{thm:domain_perturbation}, except with $g$ in place of $f$ on the right hand side.  For $0 < \theta \leq 1$ and $1 \leq \alpha \leq \epsilon^{-1}$ we define the auxiliary function 
\begin{equation*}
\Phi(x,y) := u_{n,\epsilon}(x) - (1+\theta)u_\delta(x) - \frac{\alpha}{2}|x-y|^2, \quad \quad  (x,y)\in \X\times \bar{\Omega_{\delta}}.
\end{equation*}
Let $(x_\alpha,y_\alpha) \in \X \times \bar{\Omega_{\delta}}$ be a point at which $\Phi$ is maximized over $\X \times \bar{\Omega_{\delta}}$. As in the proof of Theorem \ref{thm:main1} we have
\[\max_{\X\cap \bar{\Omega_\delta}}(u_{n,\epsilon}-u_\delta) \leq \Phi(x_\alpha,y_\alpha) + C\theta,\]
where $u(x)=u_0(x)=d_g(x,\Gamma)$. By the Lipschitzness of $u_{n,\epsilon}$ (see Theorem \ref{thm:discreteSolLip} and Remark \ref{rem:lip}) and that of $u_\delta$ (see Theorem \ref{thm:domain_perturbation} (i)), this yields
\begin{equation}\label{eq:aux_bound2}
\max_{\X}(u_{n,\epsilon} - u) \leq \Phi(x_\alpha,y_\alpha) + C\left(\theta + \delta + \left( n\epsilon^{p+d}\right)^{\frac{1}{p}}\right).
\end{equation}
with probability at least $1-n^2\exp(-cn\epsilon^d)$. As in the proof of Theorem \ref{thm:main1}, the proof proceeds by estimating $\Phi(x_\alpha,y_\alpha)$, while choosing the parameters $\theta, \delta$ and $\alpha$ appropriately. 

Since $\Phi(x_\alpha,y_\alpha) \geq \Phi(x_\alpha,x_\alpha)$, we have 
\begin{equation}
u_{n,\epsilon}(x_\alpha) - (1+\theta)u_\delta(y_\alpha) - \frac{\alpha}{2}|x_\alpha-y_\alpha|^2 \geq u_{n,\epsilon}(x_\alpha) - (1+\theta)u_\delta(x_\alpha).
\end{equation}
Since $u_\delta$ is Lipschitz continuous we have
\begin{equation}
\frac{\alpha}{2}|x_\alpha-y_\alpha|^2 \leq (1+\theta)(u_\delta(x_\alpha) - u_\delta(y_\alpha)) \leq C |x_\alpha - y_\alpha|.
\end{equation}
Hence we obtain the same bound $|x_\alpha-y_\alpha|\leq C\alpha^{-1}$ as in  \eqref{eq:xnynBound} from Theorem \ref{thm:main1}. We now make the restriction $\delta \geq 2\epsilon$, and $C\alpha^{-1}\leq \delta$ so that $|x_\alpha-y_\alpha| \leq \epsilon$. Since $y_\alpha\in \bar{\Omega_{\delta}}$, this ensures that 
\[\dist(x_\alpha,\partial\Omega) \geq \dist(y_\alpha,\partial\Omega) - |x_\alpha-y_\alpha| \geq \delta - \epsilon \geq \epsilon.\]
Therefore $B(x_\alpha,\epsilon)\subset \Omega$, which will allow us to utilize the pointwise consistency results (Lemmas \ref{discreteNonlocal} and \ref{nonLocalLocal}) later on in the proof.

We again have several cases to consider.

(i) If $y_\alpha \in \Gamma$, then $u_{\delta}(y_\alpha)=0 = u_{n,\epsilon}(y_\alpha)$ and so by the Lipschitz continuity of $u_{n,\epsilon}$ (see Remark \ref{rem:lip}) we have
\begin{align*}
u_{n,\epsilon}(x_\alpha) - u_{\delta}(y_\alpha) &= u_{n,\epsilon}(x_\alpha) - u_{n,\epsilon}(y_\alpha) \\
&\leq C\left(|x_\alpha-y_\alpha| + \left( n\epsilon^{p+d}\right)^{\frac{1}{p}}\right) \\
&\leq C\left(\alpha^{-1} + \left( n\epsilon^{p+d}\right)^{\frac{1}{p}}\right).
\end{align*}
Therefore
\[\Phi(x_\alpha,y_\alpha) \leq u_{n,\epsilon}(x_\alpha) - u_{\delta}(y_\alpha) \leq C\left(\alpha^{-1} + \left( n\epsilon^{p+d}\right)^{\frac{1}{p}}\right).\]

(ii) If $x_\alpha \in \Gamma$, then $u_{n,\eps}(x_\alpha) = 0 = u_{\delta}(x_\alpha)$ and thus
\begin{equation}
u_{n,\epsilon}(x_\alpha) - u_{\delta}(y_\alpha) = u_{\delta}(x_\alpha) - u_{\delta}(y_\alpha) \leq 0,
\end{equation}
since $u_{\delta}\geq 0$. 

(iii) We now consider the case of $x_\alpha \in \X \setminus \Gamma$ and $y_\alpha \in \bar{\Omega_\delta}\setminus \Gamma$, and we again show that $\theta$ can be chosen to rule out this case. We first note that the map 
\[y\mapsto u_\delta(y) + \frac{\alpha}{2}(1+\theta)^{-1}|x_\alpha-y|^2\]
attains its minimum at $y_\alpha\in \bar{\Omega_\delta}$ relative to the closed set $\bar{\Omega_{\delta}}$. Using $\phi(x)=-\frac{\alpha}{2}(1+\theta)^{-1}|x_\alpha-y|^2$ as a test function for the definition of viscosity supersolution for $u_\delta$, and recalling from Definition \ref{def:viscosity_solution} that the supersolution condition holds even on the boundary $\partial \Omega_\delta$, we have
\begin{equation}
\label{subSol2}
|p_\alpha|  \geq (1+\theta)g(y_\alpha),
\end{equation}
where $p_\alpha = \alpha(x_\alpha-y_\alpha)$.
Likewise, the map $x \mapsto u_{n,\eps}(x) - \frac{\alpha}{2 }|x-y_\alpha|^2$ attains its maximum at $x_\alpha \in \cX$ over the point cloud $\cX$. Setting $\psi(x) := \frac{\alpha}{2}|x - y_\alpha|^2$, we see that the inequality 
\[u_{n,\eps}(x_\alpha) - u_{n,\eps}(x) \geq \psi(x_\alpha) - \psi(x)\]
holds for all $x \in \cX$. It follows that
\[f(x_\alpha) = \cA_{n,\epsilon}u_{n,\epsilon}(x_\alpha) \geq \cA_{n,\epsilon}\psi(x_\alpha).\]
Since $B(x_\alpha,\epsilon)\subset \Omega$, we can use pointwise consistency (Lemmas \ref{discreteNonlocal} and \ref{nonLocalLocal}) to obtain
\[f(x_\alpha) \geq \rho(x_\alpha)|p_\alpha|^p - C(\alpha\epsilon + \lambda),\]
holds for any $0 < \lambda \leq 1$ with probability at least $1-2n\exp\left( -cn\epsilon^d \lambda^2\right)$, where $C$ depends on $p$, $\sigma_p$, $\|\rho\|_{C^{0,1}}$, $\eta(0)$, and $\rho_{max}$, and $c$ depends on $\eta(0)$, $\sigma_p$, and $\rho_{max}$. Dividing by $\rho$ on both sides, and combining with \eqref{subSol2} yields
\[g(x_\alpha)^p  + C(\alpha\epsilon + \lambda) \geq (1+\theta)^p g(y_\alpha)^p.\]
Since $(1+\theta)^p \geq 1+\theta$, and $g^p$ is Lipschitz, we can rearrange this and use \eqref{eq:xnynBound} to obtain
\[\theta g(y_\alpha)^p - C(\alpha\epsilon + \lambda) \leq g(x_\alpha)^p - g(y_\alpha)^p \leq C|x_\alpha-y_\alpha| \leq C\alpha^{-1}.\]
Since $g$ is bounded below by a positive constant, this yields
\[\theta \leq C(\alpha^{-1}+ \alpha\epsilon + \lambda).\]
Hence, we set
\[\theta = (C+1)(\alpha^{-1}+ \alpha\epsilon + \lambda),\]
so that case (iii) cannot hold. 

The proof is completed by combining cases (i) and (ii) with \eqref{eq:aux_bound2} to obtain 
\[\max_{\X}(u_{n,\epsilon}-u) \leq C\left(\alpha^{-1}+ \alpha\epsilon + \delta + \left( n\epsilon^{p+d}\right)^{\frac{1}{p}} +\lambda\right).\]
Optimizing over $\alpha$ yields $\alpha= \frac{1}{ \sqrt{\epsilon}}$.  We also made the restrictions $\delta\geq 2\epsilon$ and $\delta \geq C\alpha^{-1} = C \sqrt{\epsilon}$. Thus, we can again choose $\delta = C \sqrt{\epsilon}$ to satisfy these conditions, which completes the proof.
\end{proof}

%% file: sections/analysis.tex
\section{Continuum analysis}
\label{sec:analysis}

Given the discrete to continuum convergence results from Section \ref{sec:convergence}, which show that the solution of the $p$-eikonal equation converges to a density weighted geodesic distance, we now proceed to study the asymptotic consistency of the $p$-eikonal equation for both data depth and semi-supervised learning. Throughout this section we let $\Omega$ be an open, connected domain, and denote by $\rho$ the density function on $\Omega$. 

\subsection{Data depth}

We first begin with a continuum analysis of the $p$-eikonal data depth. The continuum limit of the discrete $p$-eikonal median \eqref{eq:peikonal_median} is the geodesic geometric median
\begin{equation}\label{eq:continuum_median}
x_* \in \argmin_{x\in \Omega}\int_{\Omega}d_{\rho^{-\alpha}}(y,\{x\})\, dy.
\end{equation}
The associated depth is based on the distance to $x_*$, and is given by
\begin{equation}\label{eq:continuum_depth}
\text{depth}_\alpha(x) = \max_{\Omega}d_{\rho^{-\alpha}}(\cdot,x_*) - d_{\rho^{-\alpha}}(x,x_*).
\end{equation}
We study here the case of a radial density $\rho(x)=\rho(|x|)$ that is radially decreasing on the unit ball $\Omega=B(0,1)$. In this case we expect the median to be the origin $x_*=0$ for $\alpha>0$. We are able to obtain a partial result for uniform densities.
\begin{lemma}\label{lem:origin}
If $\rho\equiv 1$ on $\Omega=B(0,1)$, then $x_*=0$. 
\end{lemma}
\begin{proof}
Since the ball is convex and $\rho\equiv 1$, we have that $d_{\rho^{-\alpha}}(x,y)=|x-y|$ for all $x,y\in B(0,1)$. Therefore
\[x_* \in \argmin_{x\in B(0,1)}\int_{B(0,1)}|x-y|\, dy.\]
Let $x\neq 0$.  We first note that
\begin{equation}\label{eq:depth1}
\int_{B(0,1)}|y|\, dy = \int_{B(x,1)}|x-y| \, dy=\int_{B(0,1)\cap B(x,1)}|x-y| \, dy + \int_{B(x,1)\setminus B(0,1)}|x-y|\, dy.
\end{equation}
Since $x\neq 0$ and $|B(x,1)\setminus B(0,1)| = |B(0,1)\setminus B(x,1)|$ and $|x-y| < 1$ for $y\in B(x,1)$ we have
\[\int_{B(x,1)\setminus B(0,1)}|x-y|\, dy < \int_{B(x,1)\setminus B(0,1)}\, dy = \int_{B(0,1)\setminus B(x,1)}\, dy.\]
Since $1 < |x-y|$ for $y\in B(0,1)\setminus B(x,1)$ we obtain
\[\int_{B(x,1)\setminus B(0,1)}|x-y|\, dy < \int_{B(0,1)\setminus B(x,1)}|x-y|\, dy.\]
Substituting this into \eqref{eq:depth1} yields
\[\int_{B(0,1)}|y|\, dy < \int_{B(0,1)\cap B(x,1)}|x-y| \, dy + \int_{B(0,1)\setminus B(x,1)}|x-y|\, dy = \int_{B(0,1)}|x-y|\, dy.\]
It follows that 
\[0 = \argmin_{x\in B(0,1)}\int_{B(0,1)}|x-y|\, dy,\]
which completes the proof.
\end{proof}
\begin{remark}\label{rem:conjecture}
We expect that Lemma \ref{lem:origin} holds for any radially decreasing density $\rho$ on the unit ball  $B(0,1)$ provided $\alpha\geq 0$, but it appears the proof would be substantially different than Lemma \ref{lem:origin}. 
\end{remark}

If the median is at the origin, we can easily compute the depth function.
\begin{lemma}\label{lem:depth}
Let $\alpha\geq 0$. Assume $\rho(x)=\rho(|x|)$ is radially decreasing and $\Omega=B(0,1)$. If $x_*=0$ then
\[\text{depth}_\alpha(x) = \int_{1-|x|}^1 \rho(t)^{-\alpha}\, dt.\]
\end{lemma}
\begin{proof}
Since $\rho$ is radial and decreasing, the shortest paths to the origin are straight lines and by definition we have
\[d_{\rho^{-\alpha}}(x,0) = \int_0^{|x|} \rho(t)^{-\alpha}\, dt.\]
Hence $\max_{\Omega}d_{\rho^{-\alpha}}(\cdot,0) = \int_{0}^1 \rho(t)^{-\alpha}\, dt$,  which completes the proof.
\end{proof}
\begin{remark}\label{rem:uniform}
Note in Lemma \ref{lem:depth} that if we take $\rho\equiv 1$ then $\text{depth}_\alpha(x)=1-|x|$.
\end{remark}

\subsection{Semi-supervised learning}
\label{sec:asymptotic_consistency}

In order to study the consistency of semi-supervised learning, we make a clusterability assumption on the density $\rho$. We assume there are $k$ classes, represented by the open and connected sets $\Omega_1,\dots,\Omega_k \subset \Omega$, all of which are mutually disjoint. For each $j=1,\dots,k$ we let
\[\rho_j = \min_{\bar{\Omega_j}} \rho,\]
and we set $\tilde{\Omega} = \Omega \setminus \bigcup_{j=1}^k \Omega_j$ and
\[\delta = \max_{\tilde{\Omega} }\rho.\]
We assume there are closed sets $\Gamma_j \subset \Omega_j$ for each $j=1,\dots,k$ that correspond to the labeled data for each class. Then the continuum limit of the $p$-eikonal semi-supervised learning algorithm from Section \ref{sec:depth_discrete} produces the predicted labels $\ell:\Omega\to \{1,\dots,k\}$ given by
\begin{equation}\label{eq:continuum_label_dec}
\ell(x) = \argmin_{1\leq j\leq k}d_{\rho^{-\alpha}}(x,\Gamma_j).
\end{equation}
\begin{definition}\label{def:consistent}
We say that the classification is \emph{asymptotically consistent} if for all $j=1,\dots,k$ we have  $\ell(x)=j$  for all $x\in \Omega_j$.
\end{definition}
Note that the definition of asymptotic consistency does not place any conditions on the label function in the space between classes $\tilde{\Omega}$. 

We define the Hausdorff distance
\[\cH (\Gamma_j,\Omega_j) = \max_{x\in \bar{\Omega_j}} d_{\Omega_j}(x,\Gamma_j),\]
which measures how well the labeled set $\Gamma_j$ covers the class $\Omega_j$ via geodesic distance on $\Omega_j$. We note that for any $A\subset \Omega$, we take the definition of $d_A$ to be $d_A(x,\Gamma)=d_{f}(x,\Gamma)$ where $f=\one_A$ is the indicator function of $A$, and $d_f$ is defined in Section \ref{sec:continuum_eikonal}. Thus, the feasible paths for $d_A(x,\Gamma)$ can travel outside of $A$, as long as they remain inside $\Omega$, but we only measure the length of the segments of the path that lie in $A$. We also define the separation of classes $i$ and $j$ by
\[\cS (\Omega_i,\Omega_j) = \min \{ d_{\tilde{\Omega}}(x,y) \, : \, x\in \Omega_i  \  \ \text{and} \ \ y\in \Omega_j\}.\]
The separation $\cS(\Omega_i,\Omega_j)$ is the length of the shortest path from a point in $\Omega_i$ to a point in $\Omega_j$, where only the distance traveled in $\tilde{\Omega}$ is counted. 

We now define 
\begin{equation}\label{eq:beta_cluster}
\beta_{ij} = \frac{\delta^{\alpha}\cH (\Gamma_j,\Omega_j)}{\rho_j^{\alpha}\cS(\Omega_i,\Omega_j)}.
\end{equation}
We assume $\beta_{ij}>0$ for all $i\neq j$. As we shall see in the results below, our clusterability assumption relates to the smallness of $\beta_{ij}$. This includes measures of how well $\Gamma_j$ covers $\Omega_j$, the ratio of the background density $\delta$ to the class density $\rho_j$, and the separation between classes $i$ and $j$.

\begin{theorem}\label{thm:ssl1}
Let $\alpha \geq 0$. If $\beta_{ij} < 1$ for all $i\neq j$, then the classification \eqref{eq:continuum_label_dec} is asymptotically consistent.
\end{theorem}
\begin{proof}
To show that the classification is asymptotically consistent, we need to show that for all $i\neq j$ we have
\begin{equation}\label{eq:needtoshow}
d_{\rho^{-\alpha}}(x,\Gamma_j) < d_{\rho^{-\alpha}}(x,\Gamma_i) \ \ \text{for all } x\in \Omega_j.
\end{equation}
Let $x\in \Omega_j$. Since $\rho \geq \rho_j$ on $\Omega_j$ we have
\[d_{\rho^{-\alpha}}(x,\Gamma_j) \leq \rho_j^{-\alpha} d_\Omega(x,\Omega_j) \leq \rho_j^{-\alpha}\cH(\Gamma_j,\Omega_j).\]
Similarly, since $\rho\leq \delta$ in $\tilde{\Omega}$ we have
\[d_{\rho^{-\alpha}}(x,\Gamma_i) \geq \delta^{-\alpha}\cS(\Omega_i,\Omega_j).\]
Combining these two inequalities, we have that \eqref{eq:needtoshow} holds provided
\[\delta^{-\alpha}\cS(\Omega_i,\Omega_j) >\rho_j^{-\alpha}\cH(\Gamma_j,\Omega_j)\]
for all $i\neq j$. Rearranging we obtain $\beta_{ij}<1$, which completes the proof.
\end{proof}

We now consider the inclusion of class priors. Given positive weights $s_1,\dots,s_k$, the continuum limit of the class priors label decision \eqref{eq:label_dec_priors} is given by
\begin{equation}\label{eq:continuum_label_dec_priors}
\ell(x) = \argmin_{1\leq j\leq k}\{s_jd_{\rho^{-\alpha}}(x,\Gamma_j)\}.
\end{equation}
\begin{theorem}\label{thm:ssl2}
Let $\alpha \geq 0$ and define
\[[\beta]_* = \max_C\left(\prod_{(i,j) \in C} \beta_{ij} \right)^{\frac{1}{|C|}},\]
where the maximum is over all cycles of $\{1,2,\dots,k\}$.  If $[\beta]_* < 1$, then there exists $s\in \R^k_+$ such that the classification \eqref{eq:continuum_label_dec_priors} is asymptotically consistent.
\end{theorem}
\begin{remark}\label{rem:betastar}
We note that $[\beta]_* \leq \max_{i\neq j}\beta_{ij}$, so Theorem \ref{thm:ssl2} shows that the utilization of class priors leads to a weaker condition for asymptotic consistency. In the case of binary classification, $k=2$, there is only one cycle $C=\{(1,2),(2,1)\}$ and we have
\[[\beta]_* =  \sqrt{\beta_{12}\beta_{21}}.\]
Thus, Theorem \ref{thm:ssl2} shows that the class priors label decision \eqref{eq:continuum_label_dec_priors} with the optimal choice of $s$ is asymptotically consistent for binary classification provided $\beta_{12}\beta_{21}<1$, which allows, for example $\beta_{12}>1$ and $\beta_{21}<1$ (or vice versa). This is a much more relaxed condition compared to the consistency of the label decision \eqref{eq:continuum_label_dec} without class priors, which requires both $\beta_{12}<1$ and $\beta_{21}< 1$. Thus, Theorem \ref{thm:ssl2} shows how class priors are able to correct for poor separation between classes, poor choices of labeled training data, or low density clusters, provided there is another class with good clusterability properties to tradeoff with. 
\end{remark}

The proof of Theorem \ref{thm:ssl2} is based on an alternative characterization of $[\beta]_*$.
\begin{proposition}\label{prop:betastar}
We have
\begin{equation}\label{eq:betastar}
[\beta]_* = \min_{s\in \R^k_+}\max_{i\neq j}  \{s_i^{-1}s_j \beta_{ij}\}.
\end{equation}
\end{proposition}
\begin{proof}
Let us define $F:\R^k_+ \to \R$ by 
\[F(s) = \max_{i\neq j}  \{s_i^{-1}s_j \beta_{ij}\}.\]
We first show that the minimum of $F$ exists.  Since only ratios of $s$ appear, we may restrict to $s$ with $s_1=1$. Set $\beta_{min}=\min_{i\neq j}\beta_{ij}$ and $\beta_{max}=\max_{i\neq j}\beta_{ij}$, and note that $\beta_{min}>0$ by assumption. Then $F(s) \geq \beta_{min}s_j$ for all $j$. Since $\inf F \leq \beta_{max}$, we may also restrict to $s$ such that $\beta_{min}s_j \leq \beta_{max}$, that is $s_j \leq \beta_{max}/\beta_{min}$. Likewise, we have $F(s) \geq \beta_{min}s_i^{-1}$ for all all $i$, so we may restrict to $s$ with $\beta_{min}s_i^{-1} \leq \beta_{max}$, or $s_i \geq \beta_{min}/\beta_{max}$. Thus, we have reduced the problem to minimizing the continuous function $F$ over a compact set, and so the minimum exists.

Let us write $F_* = \min_{s\in \R^k_+} F(s)$. Let $s\in \R^k_+$ be a minimizer of $F$.  Let $C$ be any cycle in the complete graph on $\{1,2,\dots,k\}$. Then since $s_{i}^{-1}s_j \beta_{ij} \leq F_*$ for all $i\neq j$ we have
\[\prod_{(i,j) \in C} s_i^{-1}s_j \beta_{ij} \leq F_*^{|C|}.\]
In the product on the left side, the weights $s_i$ all cancel out, since $C$ is a cycle, and so we have
\begin{equation}\label{eq:lowerbound}
F_* \geq \left(\prod_{(i,j) \in C} \beta_{ij} \right)^{\frac{1}{|C|}}.
\end{equation}
Maximizing over $C$ on the right hand side yields one direction of the proposition, that $F_* \geq [\beta]_*$. 

To prove the other direction, for $s\in \R^k_+$ let us define
\[M(s) = \{ (i,j) \, : \, i\neq j \ \ \text{and} \ \ s_i^{-1}s_j\beta_{ij}=F_*\}.\]
For any minimizer $s$ of $F$, we have $\# M(s)\geq 2$. Indeed, if $M(s)$ contained only one edge $(i,j)$, then we could decrease $s_j$ slightly to decrease $F(s)$, which contradicts the minimality of $s$.  We now select a minimizer $s$ for which $M(s)$ contains the fewest number of edges (this minimizer need not be unique). We claim that $M(s)$ must contain a cycle. To see this, note that if $M(s)$ did not contain a cycle, then there would exist an edge $(i,j)\in M(s)$ such that $(j,k)\not\in M(s)$ for all $k\neq j$.  We can therefore decrease $s_j$ slightly to produce another minimizer $\tilde{s}$ with $2 \leq \# M(\tilde{s}) < \# M(s)$, which contradicts our selection of $s$. Therefore $M(s)$ must contain a cycle.

Let $C$ be a cycle contained in $M(s)$. Then for each $(i,j)\in C$ we have $s_i^{-1}s_j \beta_{ij} = F_*$ and so 
\[\prod_{(i,j)\in C}\beta_{ij} = \prod_{(i,j)\in C} s_i^{-1}s_j \beta_{ij} = F_*^{|C|},\]
which shows that $F_* \leq [\beta]_*$, and completes the proof.
\end{proof}

We now give the proof of Theorem \ref{thm:ssl2}.
\begin{proof}[Proof of Theorem \ref{thm:ssl2}]
To show that the classification is asymptotically consistent, we need to show that there exist weights $s_j$ such that for all $i\neq j$ we have
\begin{equation}\label{eq:needtoshow_priors}
s_jd_{\rho^{-\alpha}}(x,\Gamma_j) < s_id_{\rho^{-\alpha}}(x,\Gamma_i) \ \ \text{for all } x\in \Omega_j.
\end{equation}
Applying the same arguments as in the proof of Theorem \ref{thm:ssl1}, we find that \eqref{eq:needtoshow_priors} is equivalent to $s_{i}^{-1}s_j \beta_{ij} < 1$ for all $i \neq j$. If $[\beta]_*<1$, then such weights exist, by Proposition \ref{prop:betastar}, and the proof is complete.
\end{proof}

%% file: sections/numerics.tex
\section{Numerical experiments}
\label{sec:numerics}

We present here some numerical experiments with real datasets. All code for the experiments is available online\footnote{\url{https://github.com/jwcalder/peikonal}} and uses the GraphLearning Python package \cite{calder2022graphlearning}. In all experiments we solved the graph $p$-eikonal equation \eqref{eq:graph_peikonal} with the fast marching solver described in Section \ref{sec:comp}, implemented in the C programming language. The rest of this section is broken up into data depth experiments in Section \ref{sec:depth_exp} and semi-supervised learning experiments in Section \ref{sec:ssl_exp}.

\subsection{Data depth}
\label{sec:depth_exp}

We consider the MNIST dataset of handwritten digits \cite{lecun1998gradient} and the FashionMNIST dataset \cite{xiao2017fashion}, which is a drop-in replacement for MNIST consisting of 10 classes of clothing items. Each dataset has 70,000 grayscale images of size $28\times28$ pixels. For both datasets we restricted the computations of data depth to each individual class, which consists of about 7000 datapoints per class. We constructed the graph by connecting each image to its $K$-nearest neighbors with Gaussian weights given by
\begin{equation}\label{eq:weights}
w_{ij} =\exp\left( -\frac{4|x_i-x_j|^2}{d_K(x_i)^2} \right),
\end{equation}
where $x_i$ represents the pixel values for image $i$, and $d_K(x_i)$ is the distance between $x_i$ and its $K^{\rm th}$ nearest neighbor. We used $K=20$ in all experiments. The weight matrix was then symmetrized by replacing $W$ with $W+W^T$. 

We computed the $p$-eikonal median via the definition \eqref{eq:peikonal_median} with $p=1$ and $\alpha=2$. For the density estimator $\hat{\rho}$ we used a $k$-nearest neighbor density estimator with $k=30$. To speed up the computation of \eqref{eq:peikonal_median}, we computed the minimum in \eqref{eq:peikonal_median} over 5\% of the nodes in each class, chosen at random. This takes about 5 minutes to compute for each dataset (30 seconds per class), which includes the time for the $k$-nearest neighbor search.

\begin{figure}[!t]
\centering
\subfloat[Deepest images (median)]{\includegraphics[clip=true,trim=55 40 45 40, width=0.48\textwidth]{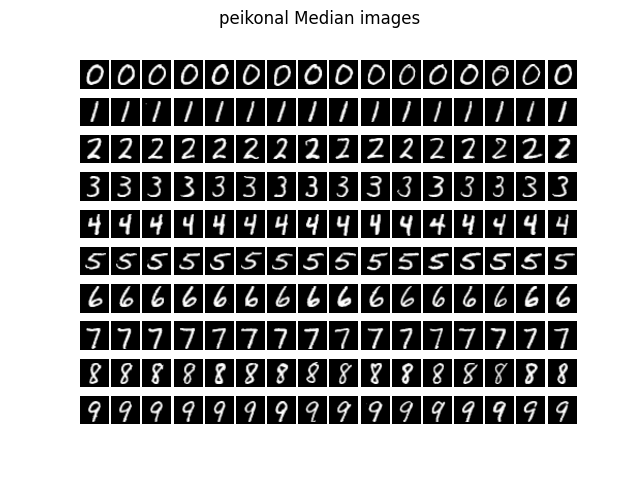}}
\hfill
\subfloat[Shallowest images (outliers)]{\includegraphics[clip=true,trim=55 40 45 40, width=0.48\textwidth]{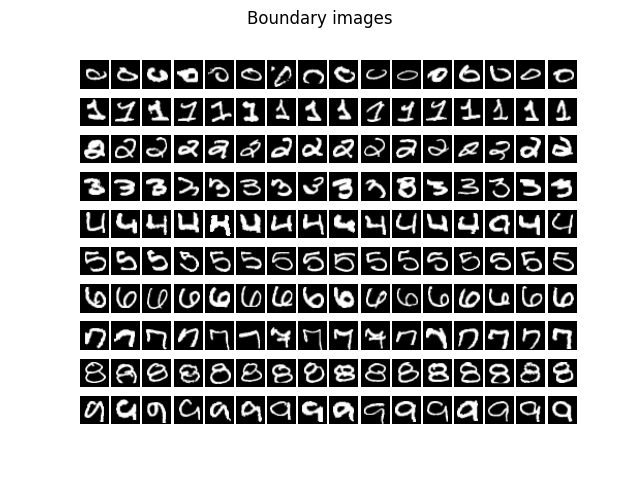}}
\caption{Comparison of deepest (median) images to shallowest (outlier) images from each MNIST digit.}
\label{fig:mnist_depth}
\end{figure}
\begin{figure}[!t]
\centering
\subfloat[Deepest images (median)]{\includegraphics[clip=true,trim=55 40 45 40, width=0.48\textwidth]{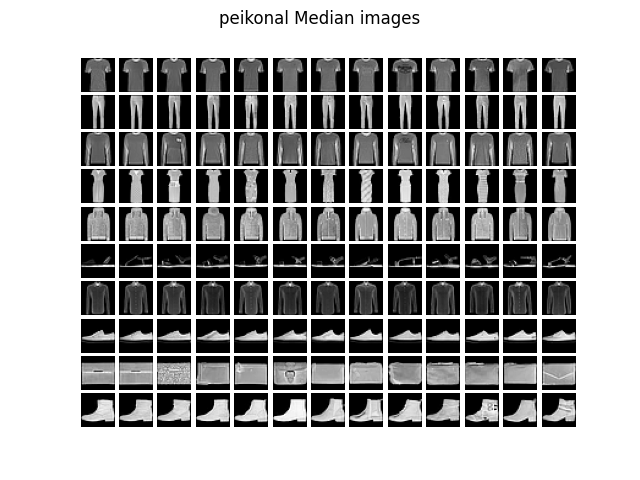}}
\hfill
\subfloat[Shallowest images (outliers)]{\includegraphics[clip=true,trim=55 40 45 40, width=0.48\textwidth]{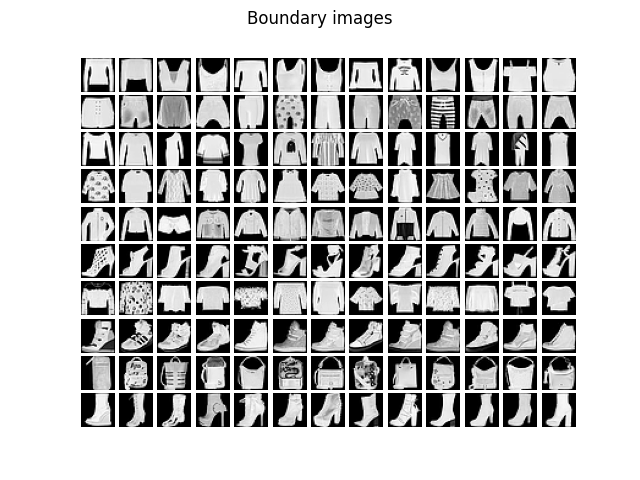}}
\caption{Comparison of deepest (median) images to shallowest (outlier) images from each FashionMNIST class.}
\label{fig:fashionmnist_depth}
\end{figure}
\begin{figure}[!t]
\centering
\subfloat[MNIST]{\includegraphics[clip=true,trim=55 40 45 40, width=0.48\textwidth]{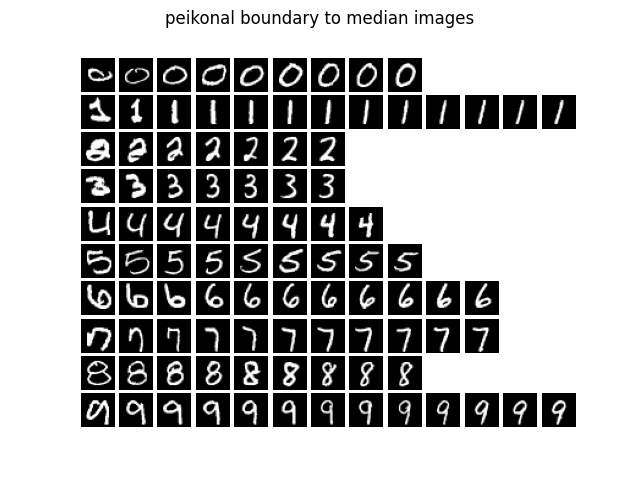}}
\hfill
\subfloat[FashionMNIST]{\includegraphics[clip=true,trim=55 40 45 40, width=0.48\textwidth]{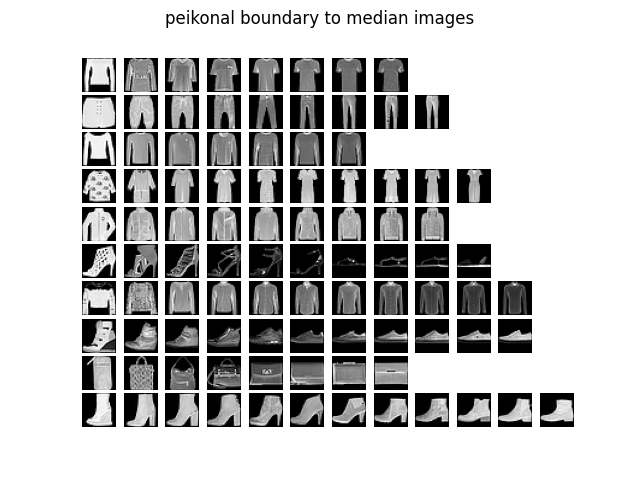}}
\caption{Paths from shallowest point to median for each class computed with the gradient descent method from Section \ref{sec:short}.}
\label{fig:paths}
\end{figure}

In Figures \ref{fig:mnist_depth} and \ref{fig:fashionmnist_depth} we show the deepest images (i.e., the medians) and the shallowest images (i.e., outliers) from each class for the MNIST and FashionMNIST datasets. We can see that the deepest handwritten digits are very clean and self-consistent, while the shallowest do appear visually to be outliers. For FashionMNIST the deepest images are again self-similar and very plain, while the shallowest images tend to be more varied and have patterns on the clothing items.  Finally, in Figure \ref{fig:paths} we show paths through each class from the shallowest point to the deepest point, following the gradient descent path construction from Section \ref{sec:short}.

\subsection{Semi-supervised learning}
\label{sec:ssl_exp}

\begin{figure}[!t]
\centering
\subfloat[MNIST]{\includegraphics[clip=true,trim=10 15 10 10,width=0.48\textwidth]{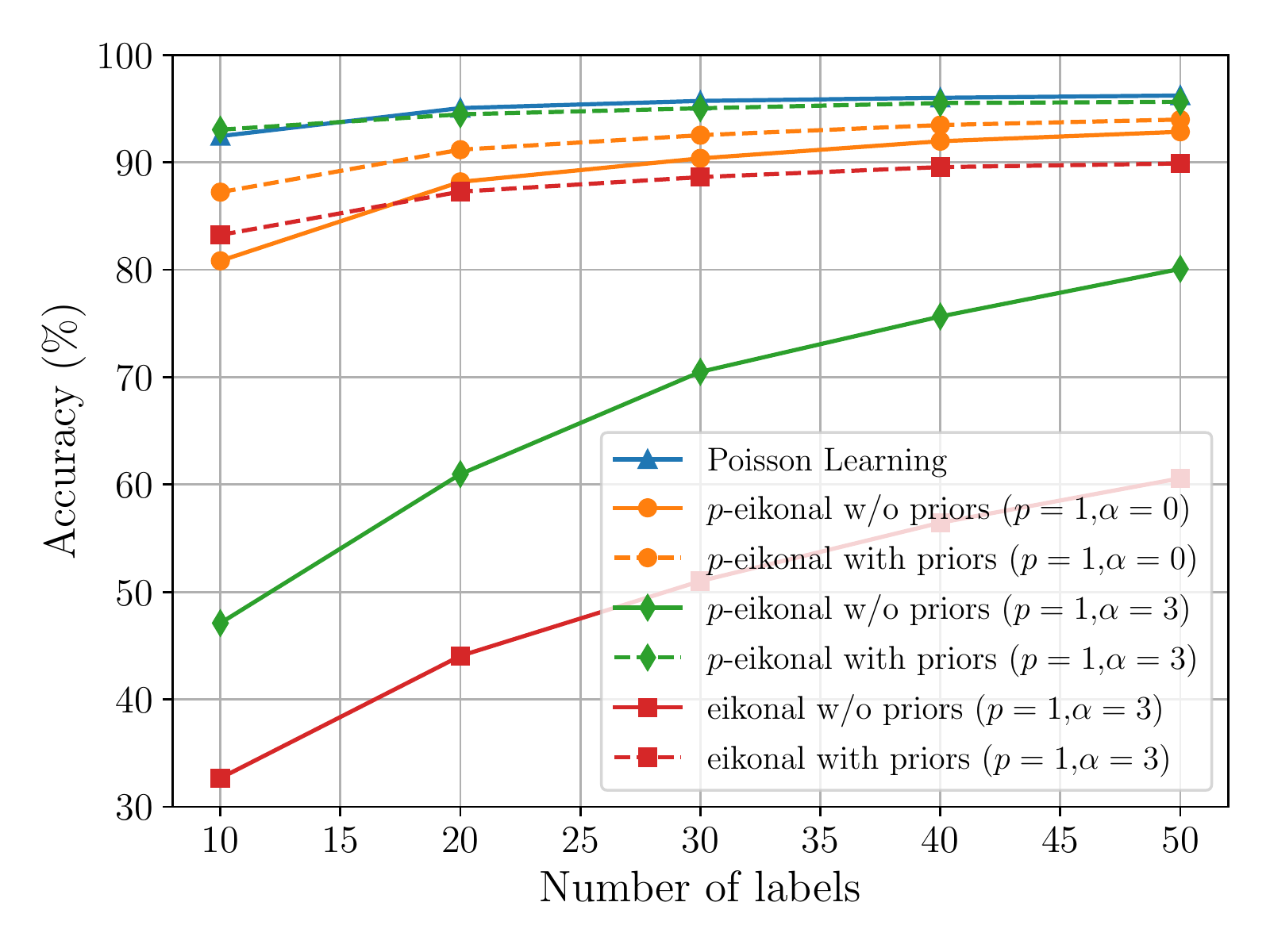}}
\hfill
\subfloat[FashionMNIST]{\includegraphics[clip=true,trim=7 15 7 10,width=0.48\textwidth]{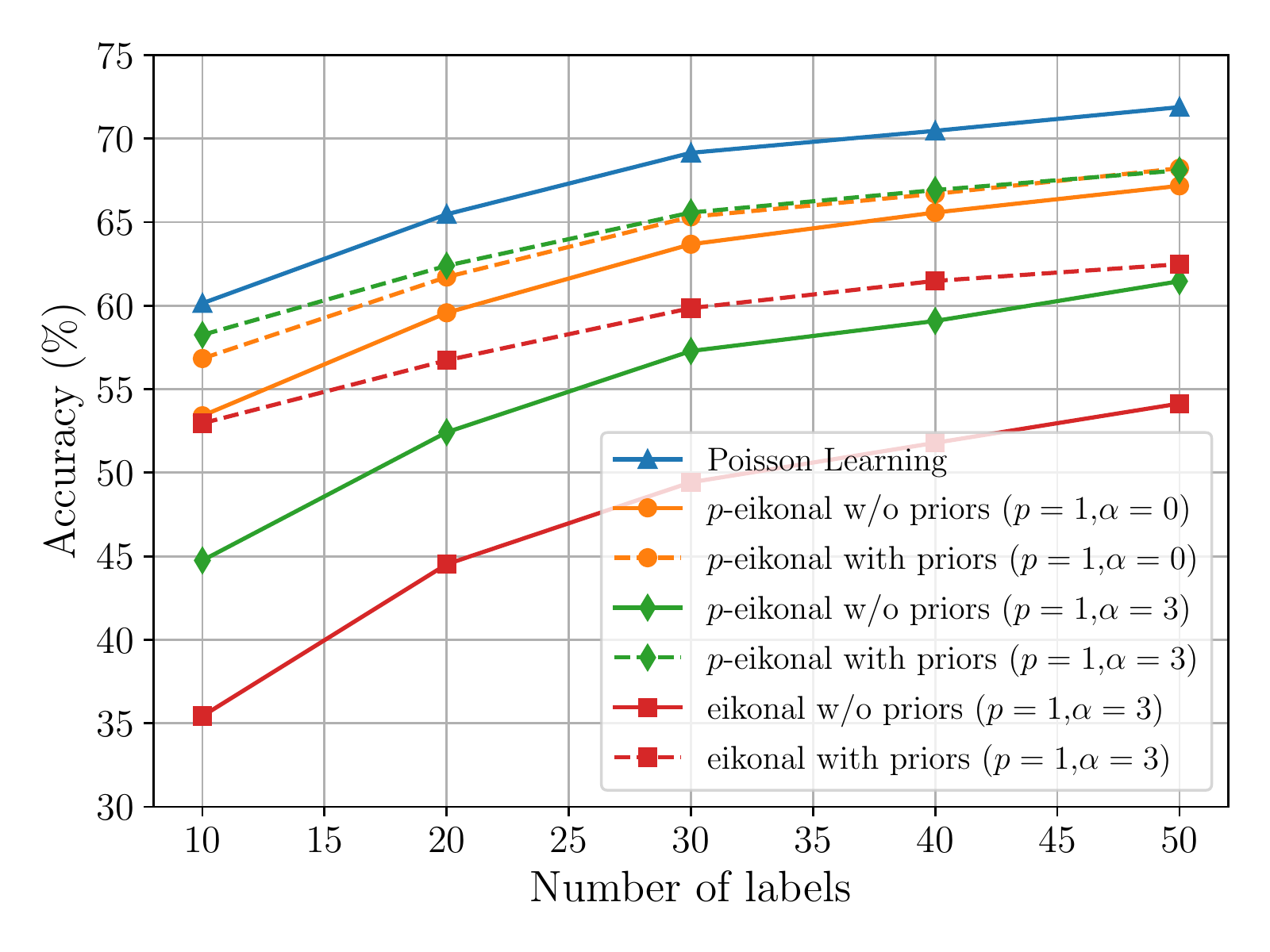}}
\caption{Comparison of the $p$-eikonal equation with $p=1$ for semi-supervised image classification to Poisson learning \cite{calder2020poisson} and the eikonal equation \eqref{eq:density_eikonal}.}
\label{fig:ssl_mnist}
\end{figure}

\begin{figure}[!t]
\centering
\subfloat[CIFAR-10]{\includegraphics[clip=true,trim=10 15 10 10,width=0.48\textwidth]{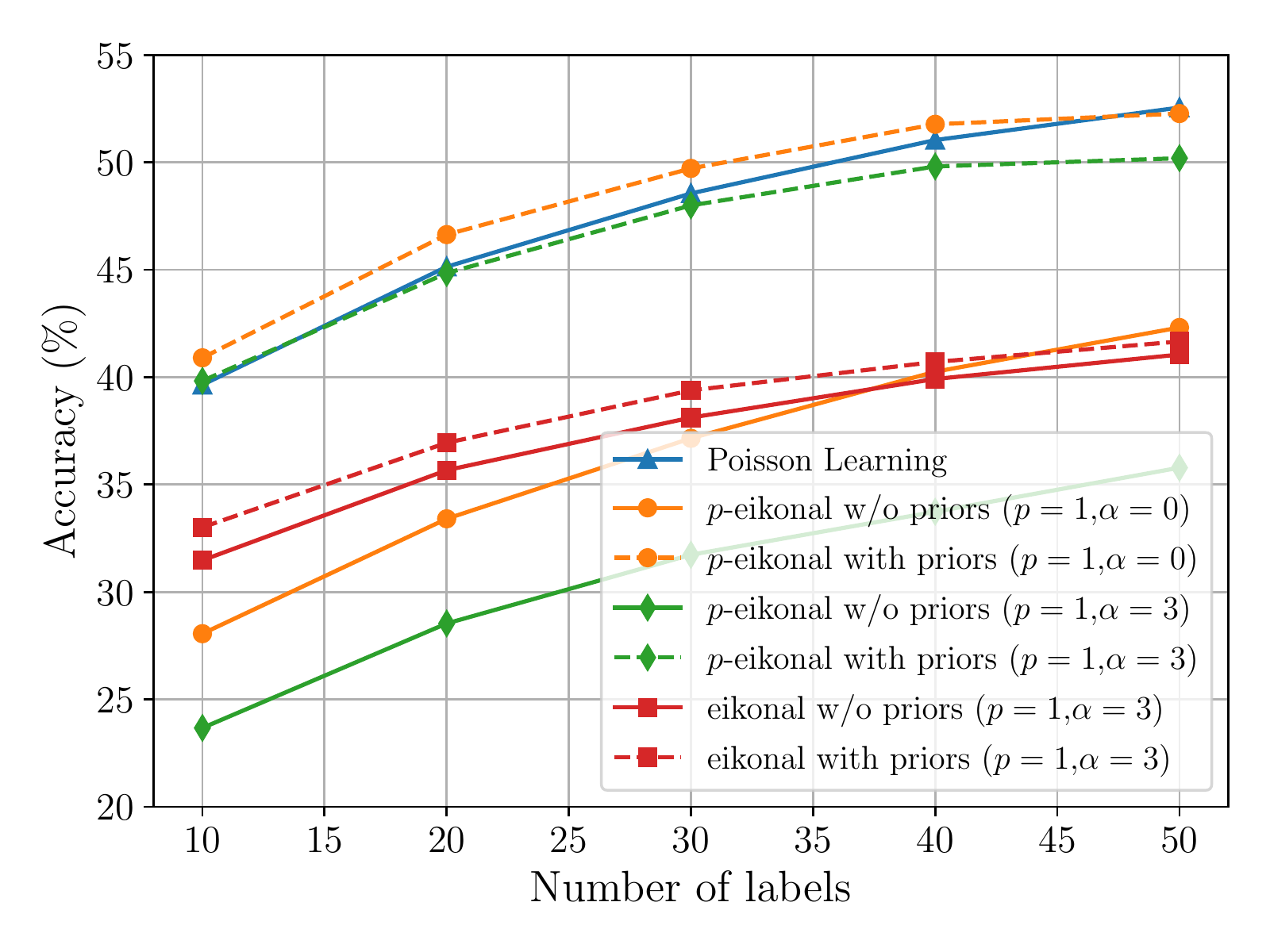}\label{fig:ssl_cifar}}
\hfill
\subfloat[Accuracy vs $\alpha$]{\includegraphics[clip=true,trim=7 15 7 10,width=0.48\textwidth]{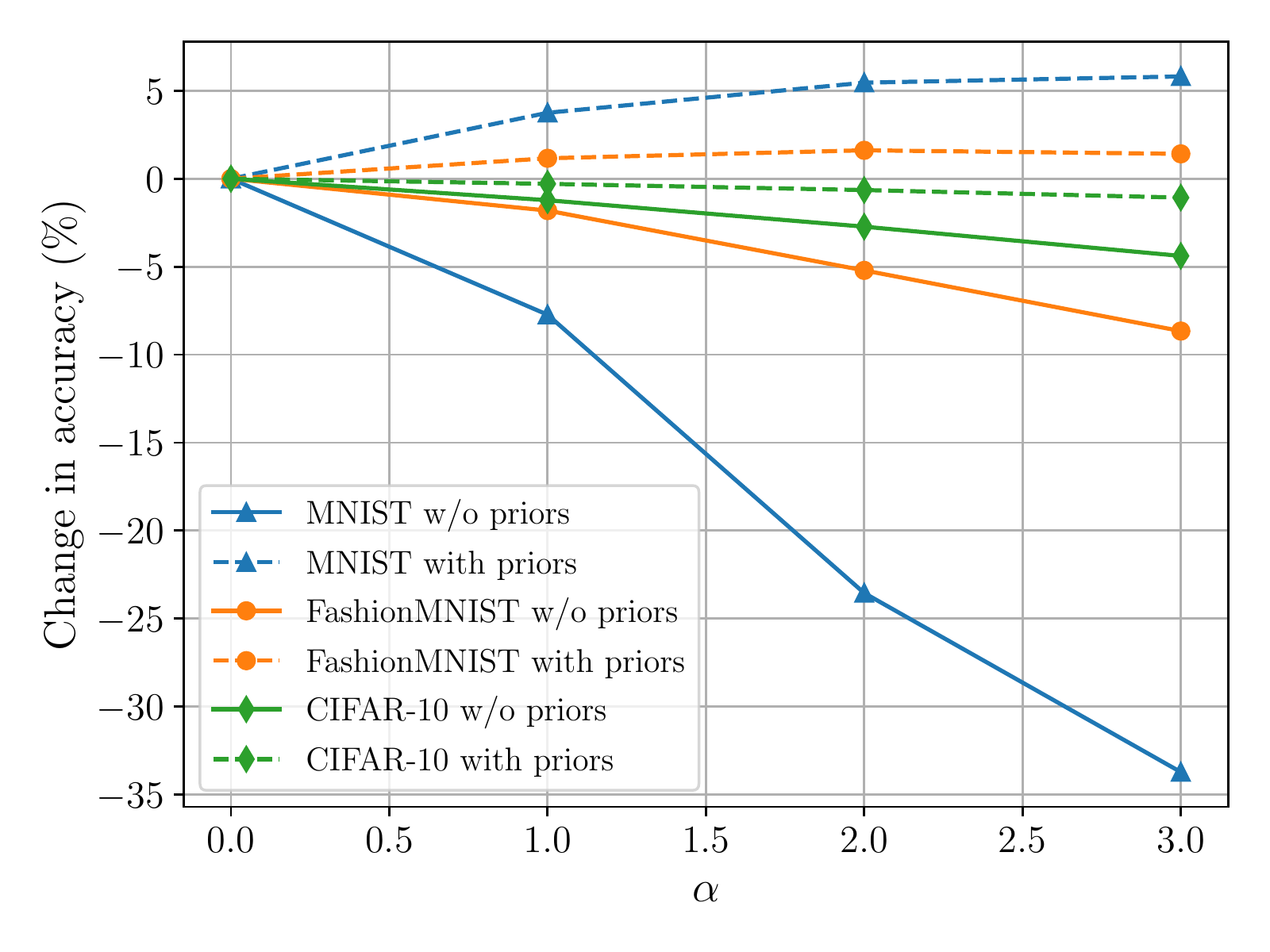}\label{fig:ssl_alphacomp}}
\caption{(a) Accuracy results for the $p$-eikonal equation with $p=1$ for semi-supervised image classification on CIFAR-10, and (b) change in accuracy as the density reweighting exponent $\alpha$ is adjusted.}
\end{figure}

\begin{figure}[!t]
\centering
\subfloat[Accuracy vs $p$]{\includegraphics[clip=true,trim=10 15 10 10,width=0.48\textwidth]{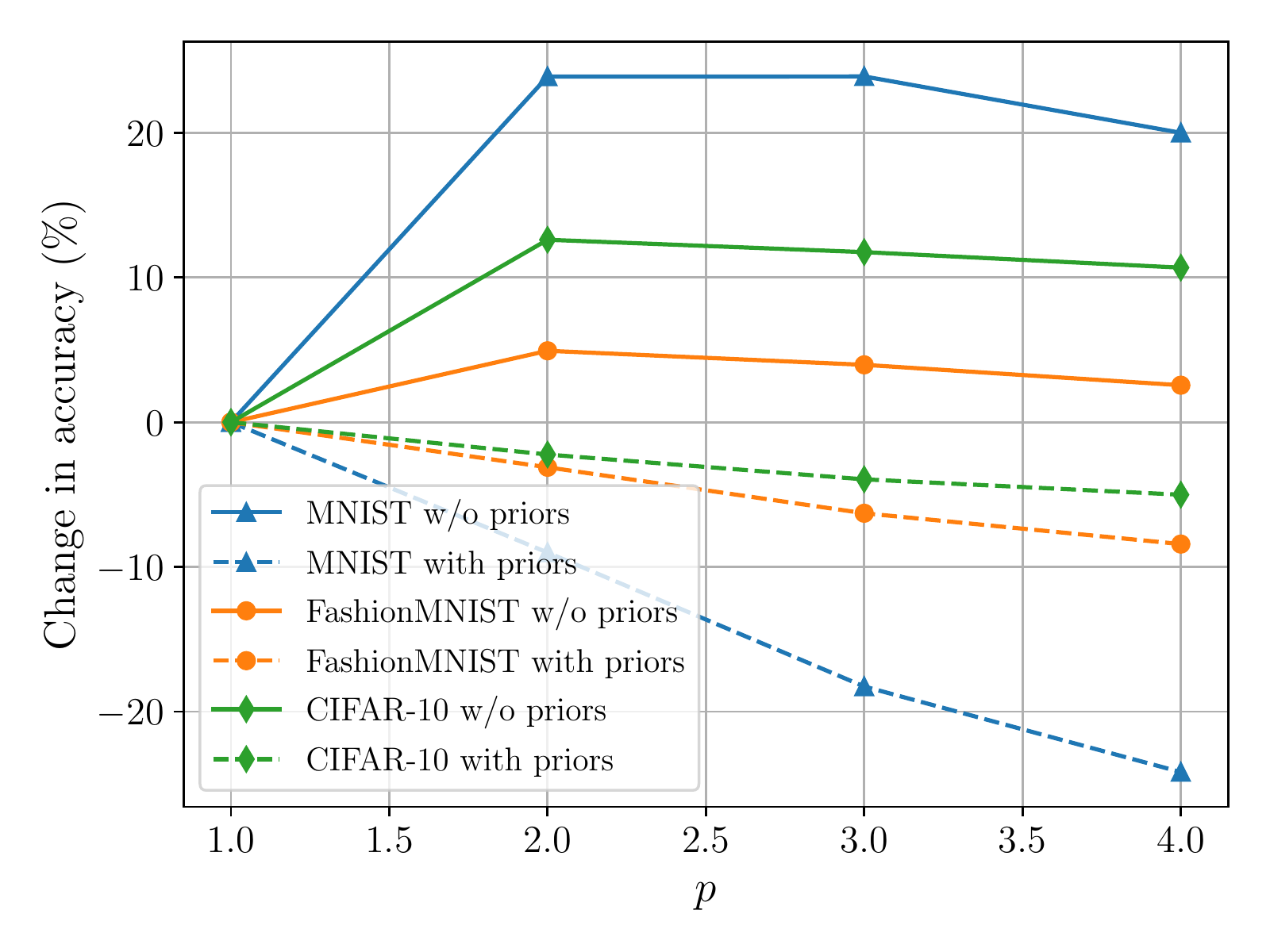}\label{fig:ssl_pcomp}}
\hfill
\subfloat[$p=2,\alpha=3$]{\includegraphics[clip=true,trim=7 15 7 10,width=0.48\textwidth]{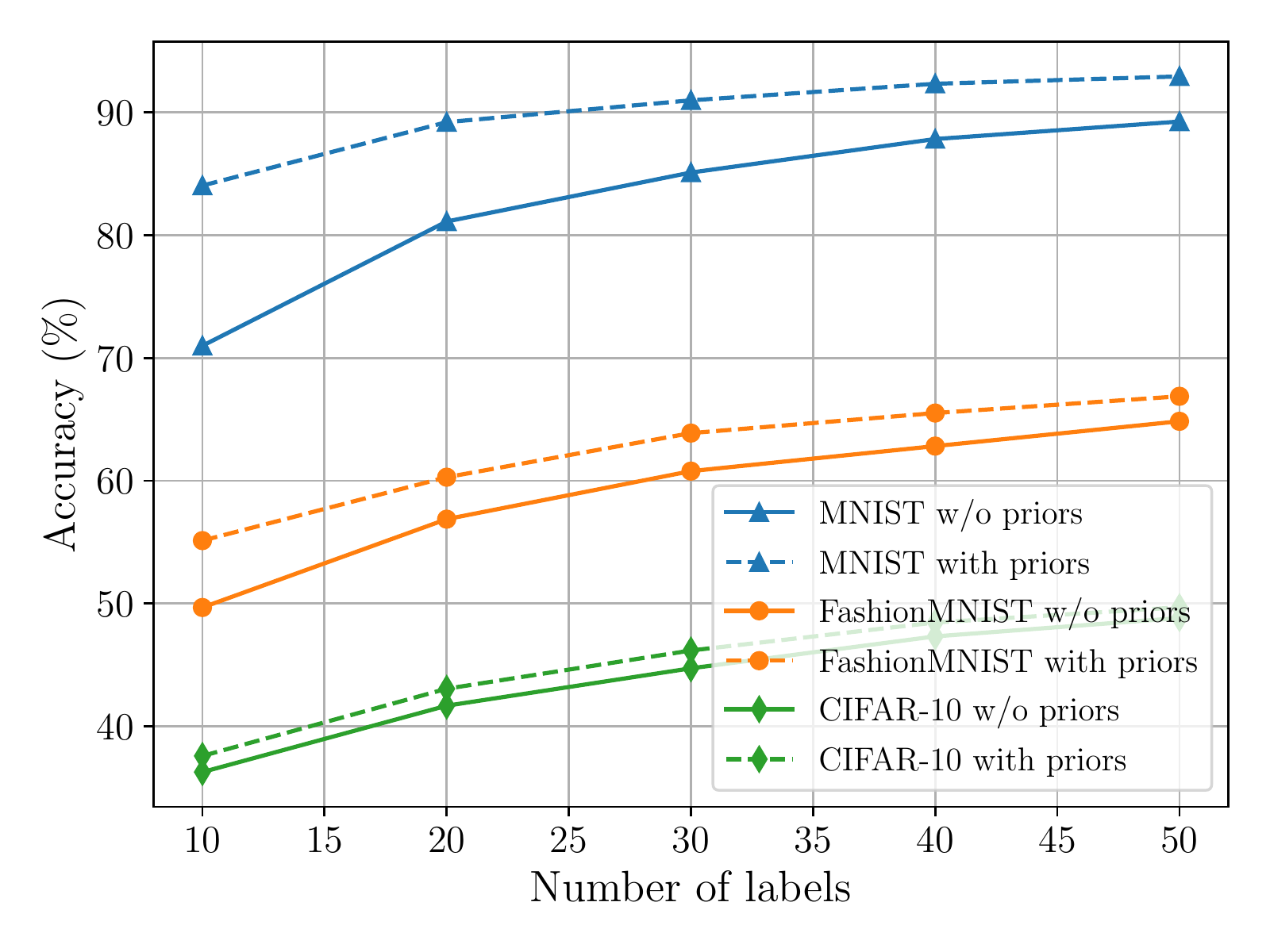}\label{fig:ssl_p2}}
\caption{Comparison of how the classification accuracy depends on the exponent $p$ in the $p$-eikonal equation. In both experiments we used density weighting with $\alpha=3$. }
\end{figure}

We tested the $p$-eikonal equation for semi-supervised learning at very low label rates with $p=1$. In addition to MNIST and FashionMNIST, we also tested on CIFAR-10 \cite{krizhevsky2009learning}. To build good quality graphs for classification, we cannot use the pixel-wise differences that we did for data depth in Section \ref{sec:depth_exp}. Instead we follow the methods from \cite{calder2020poisson} and trained autoencoders to extract important features from the data.  For MNIST and FashionMNIST, we used variational autoencoders, similar to \cite{kingma2013bayes}, while for CIFAR-10 we used the AutoEncodingTransformations architecture from \cite{zhang2019aet}. After training the autoencoders we built $K$-nearest neighbor graphs with weights given by \eqref{eq:weights} over the latent variables using the angular similarity with $K=20$ neighbors. We again used a $k$-nearest neighbor density estimator with $k=30$ neighbors for the reweighting. We refer to \cite{calder2020poisson} for more details about the autoencoder graph construction, which was also used successfully in another recent work \cite{miller2022graphbased}. After the graphs have been constructed, the classification results on any of the 3 datasets, which requires solving $10$ $p$-eikonal equations, takes a few seconds to run the classification for each trial.

We ran 100 trials at $1$ label per class up to $5$ labels per class, randomly choosing different labeled data for each trial. We compared against Poisson learning \cite{calder2020poisson} and the graph distance eikonal equation \eqref{eq:density_eikonal} with the same density reweighting schemes. We tested the $p$-eikonal and eikonal equations with and without class priors, as described in Section \ref{sec:ssl_discrete}. Figure \ref{fig:ssl_mnist} shows the results for MNIST and FashionMNIST, while Figure \ref{fig:ssl_cifar} shows the results on CIFAR-10. We see that with class priors, $p$-eikonal learning is comparable to Poisson learning on MNIST, slightly worse on FashionMNIST and slightly better on CIFAR-10. We also see that $p$-eikonal offers a significant improvement over the shortest path based eikonal classifier, even though we applied the same density reweighting to both. Since the asymptotic consistency results from Section \ref{sec:asymptotic_consistency} would hold equally well for the density reweighted eikonal equation, we attribute the improved results to the robustness properties of the $p$-eikonal equation (see Theorem \ref{thm:robust}) to perturbations in graphs, which are common in real data.

In Figure \ref{fig:ssl_alphacomp} we show how the accuracy changes for each dataset as the density exponent is increased. We find a quite surprising result here; without class priors the accuracy actually \emph{decreases} when density reweighting is used. It is only with the addition of class priors that the density reweighting can increase the accuracy of the classifier. This is true across all datasets and validates our theoretical findings in Theorem \ref{thm:ssl2} that class priors can effectively make use of density reweighting to improve classification results.

Finally, in Figure \ref{fig:ssl_pcomp} we show how the accuracy changes as the exponent $p$ in the $p$-eikonal equation is changed (here, $\alpha=3$). All our previous experiments were with $p=1$, and we find another surprising result here; the classification accuracy improves up to $p=2$ without class priors, but is monotonically decreasing when utilizing class priors. This may simply be due to the fact that the classification accuracy is already very high with class priors, and very low without. Indeed, in Figure \ref{fig:ssl_p2} we show the accuracy for $p=2$ and $\alpha=3$ for each dataset, and both with and without class priors. We see that even though $p=2$ is better for the classifiers without class priors, the incorporation of class priors still improves the accuracy significantly.

%% file: sections/conclusion.tex
\section{Conclusion}

We introduced and studied a family of graph-based distance-type equations called the $p$-eikonal equation. We showed that the $p$-eikonal equation for $p=1$ is a robust estimator of the geodesic density weighted path distance on the underlying Euclidean space, compared to the standard shortest-path graph distance. We proved that, while the $p$-eikonal equation is not a distance function on a graph, it has similar properties and its continuum limit recovers the geodesic density weighted distance on the underlying Euclidean space, with quantitative convergence rates. We used the continuum limit theory to prove asymptotic consistency of data depth and semi-supervised learning with the $p$-eikonal equation and then gave some experiments with real data on the MNIST, FashionMNIST, and CIFAR-10 datasets.

%% file: sections/appendix.tex
\section{Concentration of measure}
\label{sec:conc}

We recall here some useful concentration of measure results.
\begin{theorem}(Bernstein Inequality)
	Let $x_1, x_2, \ldots, x_n$ be a sequence of $i.i.d$ real-valued random variables with finite expectation $\mu = \E(x_i)$ and variance $\sigma^2 = \text{Var}(x_i)$, and write $S_n := \frac{1}{n} \sum_{i=1}^{n} x_i$. Assume there exists $b > 0$ such that $|x-\mu|\leq b$ almost surely. Then for any $t > 0$ we have
	\begin{equation}
		\P(S_n - \mu \geq t) \leq \exp\bigg(-\frac{nt^2}{2(\sigma^2+\frac{bt}{3})}\bigg)
	\end{equation}
\end{theorem}

\begin{theorem}(Chernoff bounds)
	Let $x_1, x_2, \ldots, x_n$ be a sequence of $i.i.d$ Bernoulli random variables with parameter $p \in [0,1]$. Then for any $\delta > 0$ we have
	\begin{equation}
		\P\bigg(\sum_{i=1}^{n} x_i \geq (1+\delta)np\bigg) \leq \exp\bigg(-\frac{np\delta^2}{2(1+\frac{1}{3} \delta)}\bigg)
	\end{equation}
	and for any $0 \leq \delta < 1$ we have
	\begin{equation}
		\P\bigg(\sum_{i=1}^{n} x_i \leq (1-\delta)np\bigg) \leq \exp\bigg(-\frac{1}{2}np\delta^2\bigg)
	\end{equation}
\end{theorem}

\section{Technical proofs}
\label{sec:proofs}

We include here some technical, but elementary, proofs from the paper.
\begin{proof}[Proof of Theorem \ref{thm:existence}]
Since $H$ admits comparison, there is at most one solution of \eqref{eq:genPDE}, so we only have to establish existence. We use the Perron method. Let $\F$ be the set of all $v\in \lx$ such that
\[H(\nabla_\X v(x_i),v(x_i),x_i)\leq 0 \ \ \text{ for all }x_i\in \X\setminus \Gamma\] 
and $v=g$ on $\Gamma$. The set $\F$ is nonempty, since $\phi\in \F$. Define the Perron function
\[u(x_i) = \sup \left\{ v(x_i) \, : \, v\in \F \right\}.\]
Since $H$ admits comparison, we have $v\leq \psi$ for all $v\in \F$, and so $\phi \leq u\leq \psi$.

We now claim that 
\[H(\nabla_\X u(x_i),u(x_i),x_i)\leq 0 \ \ \text{ for all }x_i \in \X\setminus \Gamma.\]
To see this, let $x_i\in \X\setminus \Gamma$ and let $\eps>0$. There exist $v\in \F$ such that $u(x_i) \leq v(x_i) - \eps$. By definition we have $u(x_j)\geq v(x_j)$ for all $j$, and so $\nabla_\X u(x_i) + \eps \one \leq \nabla_\X v(x_i)$. Therefore
\[0 \geq H(\nabla_\X v(x_i),v(x_i),x_i)\geq H(\nabla_\X u(x_i) + \eps \one,u(x_i) + \eps,x_i).\]
Sending $\eps\to 0$ and using continuity of $H$ establishes the claim.

We now claim that
\[H(\nabla_\X u(x_i),u(x_i),x_i)\geq 0 \ \ \text{ for all }x_i \in \X\setminus \Gamma,\]
which will complete the proof.  Assume, by way of contradiction, that 
\[H(\nabla_\X u(x_i),u(x_i),x_i)< 0 \ \ \text{ for some }x_i \in \X\setminus \Gamma.\]
Let $\eps>0$ and define $u_\eps\in \lx$ by $u_\eps(x_j)=u(x_j)$ for $j\neq i$, and $u_\eps(x_i)=u(x_i) + \eps$. By continuity of $H$, there is a sufficiently small $\eps>0$ so that
\[H(\nabla_\X u_\eps(x_i),u_\eps(x_i),x_i)\leq 0.\]
Furthermore, for any $j\neq i$, we have $u_\eps(x_j)=u(x_j)$ and $\nabla_\X u_\eps(x_j) \leq \nabla_\X u(x_j)$. Since $H$ is monotone we find that
\[H(\nabla_\X u_\eps(x_j),u_\eps(x_j),x_j)\leq H(\nabla_\X u(x_j),u(x_j),x_j)\leq 0\] 
for $j\neq i$ with $x_j \in \X\setminus \Gamma$. Therefore $u_\eps\in \F$, which is a contradiction (since $u_\eps(x_i)>u(x_i)$), establishing the claim and completing the proof.
\end{proof}

\begin{proof}[Proof of Theorem \ref{thm:existence_state_const}]
The proof is uses the following dynamic programming principle
\[u(x) = \min_{y\in \partial B(x,r)\cap \bar{\Omega}} \{ u(y) + d_f(x,y)\},\]
which holds provided $B(x,r)\subset \Omega\setminus \Gamma$ and is immediate to verify. Rearranging the dynamic programming principle we obtain
\begin{equation}\label{eq:dpp_Lf}
\max_{y\in \partial B(x,r)\cap \bar{\Omega}} \{ u(x) - u(y) - d_f(x,y)\} = 0.
\end{equation}
Since $f$ is Lipschitz continuous, we have
\[d_f(x,y) = f(x)|x-y| + \cO(|x-y|^2),\]
which, when substituted above, yields
\begin{equation}\label{eq:dpp_Lf_conseq}
\max_{y\in \partial B(x,r)\cap \bar{\Omega}} \left\{ \frac{u(x) - u(y)}{r}\right\} = f(x) + \cO(r).
\end{equation}

We now prove the subsolution property. Let $x\in \Omega\setminus \Gamma$ and let $\phi\in C^\infty(\R^d)$ such that $u-\phi$ has a local maximum at $x$. Then for $r>0$ sufficiently small we have $B(x,r)\subset \Omega\setminus \Gamma$ and
\[u(x) - \phi(x) \geq u(y) - \phi(y) \ \ \text{for all } y\in B(x,r).\]
Rearranging we have
\[u(x) - u(y) \geq \phi(x) - \phi(y) \ \ \text{for all } y\in B(x,r).\]
Plugging this into \eqref{eq:dpp_Lf_conseq} yields
\[\max_{y\in \partial B(x,r)} \left\{ \frac{\phi(x) - \phi(y)}{r}\right\} \leq f(x) + \cO(r).\]
Notice the maximum is over only $\partial B(x,r)$, since $B(x,r)\subset \Omega\setminus\Gamma$. Sending $r\to 0$ yields $|\nabla \phi(x)| \leq f(x)$, which is exactly the subsolution property.

To prove the supersolution property, let $x\in \bar{\Omega}\setminus \Gamma$ and let $\phi\in C^\infty(\R^d)$ such that $u-\phi$ has a local minimum at $x$. As above, this means that
\[u(x) - u(y) \leq \phi(x) - \phi(y) \ \ \text{for all } y\in B(x,r)\cap \bar{\Omega}.\]
For $r>0$ small enough $B(x,r)\subset \bar{\Omega} \setminus \Gamma$, and so we can substitute this into \eqref{eq:dpp_Lf_conseq} to obtain
\[\max_{y\in \partial B(x,r)\cap \bar{\Omega}} \left\{ \frac{\phi(x) - \phi(y)}{r}\right\} \geq f(x) + \cO(r).\]
By enlarging the domain in the maximum above, we obtain
\[\max_{y\in \partial B(x,r)} \left\{ \frac{\phi(x) - \phi(y)}{r}\right\} \geq f(x) + \cO(r).\]
We now send $r\to 0$ to obtain $|\nabla \phi(x)| \geq f(x)$, which completes the proof.
\end{proof}
\begin{proof}[Proof of Proposition \ref{prop:ball_measure}]
We note that the inequality \eqref{eq:geodesic_euclidean} can be restated as
\begin{equation}\label{eq:geodesic_inclusion}
B(x,r) \subset B_\Omega(x,r + Cr^2) \ \ \text{and} \ \ B_\Omega(x,r) \subset B(x,r).
\end{equation}
For $r>0$ sufficiently small, so that $Cr\leq \frac{1}{2}$, the first inclusion above implies that
\begin{equation}\label{eq:geodesic_inclusion2}
B_\Omega(x,r) \supset B(x,r-Cr^2) \supset B(x,\tfrac{r}{2}).
\end{equation}
Since the boundary $\partial\Omega$ is $C^{1,1}$, there exists $v\in \R^d$ with $|v|=1$ and $c>0$ such that
\[B(x,\tfrac{r}{2})\cap \Omega \supset \{y\in  B(x,\tfrac{r}{2}) \, : \, (y-x)\cdot v \geq cr^2\}.\]
For $r$ smaller, so that $cr\leq \frac{1}{4}$ as well, we have
\begin{align*}
|B_\Omega(x,r)\cap \Omega| &\geq |B(x,\tfrac{r}{2})\cap \Omega | \\
&\geq |\{y\in B(x,\tfrac{r}{2}) \, : \, (y-x)\cdot v \geq \tfrac{r}{4}\}|\\
&=\left( \frac{r}{2}\right)^d| \{z\in B(0,1) \, : \, z\cdot v \geq \tfrac{1}{2}\}|\\
&=c_dr^d
\end{align*}
where
\[c_d := \frac{1}{2^d}\int_{B(0,1)\cap \{z_1\geq \tfrac{1}{2}\}} \, dx.\]
We finally compute
\begin{align*}
c_d &= \frac{1}{2^d}\int_{\frac{1}{2}}^1 \omega_{d-1}(1-z_1^2)^{\frac{d-1}{2}}\, dx\\
&\geq \frac{\omega_{d-1}}{2^d}\int_{\frac{1}{2}}^1 z_1(1-z_1^2)^{\frac{d-1}{2}}\, dx\\
&= -\frac{\omega_{d-1}}{2^d(d+1)}(1-z_1^2)^{\frac{d+1}{2}}\Big\vert_{\frac{1}{2}}^1\\
&=\frac{\omega_{d-1}}{2^d(d+1)}\left(\frac{3}{4}\right)^{\frac{d+1}{2}}.
\end{align*}
Applying the lower bound $\tfrac34\geq \tfrac12$ above to simplify the constant completes the proof.
\end{proof}

%\begin{remark}\label{rem:supersolution}
%In the proof of Theorem \ref{thm:existence_state_const}, the supersolution holds for $x\in \partial\Omega$ in a trivial way, since the restriction of the maximum to $\bar{\Omega}$ can be relaxed when the inequalities go in the right direction. The same argument does not hold for the subsolution condition, since the inequality goes in the opposite direction here. 
%
%As an example, we can consider the function $u(x)=|x|$, which solves the problem $|u'(x)|=1$ for $x\in [-1,1]\setminus \{0\}$. At the boundary (i.e., $x\in \{\pm 1\})$, the subsolution property does not hold, since we can touch with a function like $\phi(x)=2x$ at $x=-1$, but $|\phi'|=2 > 1$. 
%\end{remark}

%% file: main.bbl
\begin{thebibliography}{10}

\bibitem{alamgir2012shortest}
M.~Alamgir and U.~Von~Luxburg.
\newblock Shortest path distance in random k-nearest neighbor graphs.
\newblock {\em arXiv preprint arXiv:1206.6381}, 2012.

\bibitem{bardi1997optimal}
M.~Bardi, I.~C. Dolcetta, et~al.
\newblock {\em Optimal control and viscosity solutions of
  Hamilton-Jacobi-Bellman equations}, volume~12.
\newblock Springer, 1997.

\bibitem{barnett1976ordering}
V.~Barnett.
\newblock The ordering of multivariate data.
\newblock {\em Journal of the Royal Statistical Society: Series A (General)},
  139(3):318--344, 1976.

\bibitem{belkin2003laplacian}
M.~Belkin and P.~Niyogi.
\newblock Laplacian eigenmaps for dimensionality reduction and data
  representation.
\newblock {\em Neural computation}, 15(6):1373--1396, 2003.

\bibitem{bijral2012semi}
A.~S. Bijral, N.~Ratliff, and N.~Srebro.
\newblock Semi-supervised learning with density based distances.
\newblock {\em arXiv preprint arXiv:1202.3702}, 2012.

\bibitem{borgwardt2005shortest}
K.~M. Borgwardt and H.-P. Kriegel.
\newblock Shortest-path kernels on graphs.
\newblock In {\em Fifth IEEE international conference on data mining
  (ICDM'05)}, pages 8--pp. IEEE, 2005.

\bibitem{bou2021hamilton}
A.~Bou-Rabee and P.~S. Morfe.
\newblock Hamilton-{J}acobi scaling limits of pareto peeling in 2d.
\newblock {\em arXiv preprint arXiv:2110.06016}, 2021.

\bibitem{bungert2021uniform}
L.~Bungert, J.~Calder, and T.~Roith.
\newblock Uniform convergence rates for {L}ipschitz learning on graphs.
\newblock {\em arXiv:2111.12370}, 2021.

\bibitem{calder2015a}
J.~Calder.
\newblock A direct verification argument for the {H}amilton-{J}acobi equation
  continuum limit of nondominated sorting.
\newblock {\em Nonlinear Analysis Series A: Theory, Methods, \& Applications},
  141:88--108, 2016.

\bibitem{calder2015numerical}
J.~Calder.
\newblock Numerical schemes and rates of convergence for the
  {H}amilton-{J}acobi equation continuum limit of nondominated sorting.
\newblock {\em Numerische Mathematik}, 137(4):819--856, 2017.

\bibitem{calder2018game}
J.~Calder.
\newblock The game theoretic p-{L}aplacian and semi-supervised learning with
  few labels.
\newblock {\em Nonlinearity}, 32(1), 2018.

\bibitem{calder2019lip}
J.~Calder.
\newblock Consistency of {L}ipschitz learning with infinite unlabeled data and
  finite labeled data.
\newblock {\em SIAM Journal on Mathematics of Data Science}, 1:780--812, 2019.

\bibitem{calder2022graphlearning}
J.~Calder.
\newblock {GraphLearning Python Package}.
\newblock {\em \emph{\texttt{doi:10.5281/zenodo.5850940}}}, 2022.
\newblock \url{https://github.com/jwcalder/GraphLearning}.

\bibitem{calder2020poisson}
J.~Calder, B.~Cook, M.~Thorpe, and D.~Slep\v{c}ev.
\newblock {Poisson Learning: Graph based semi-supervised learning at very low
  label rates}.
\newblock {\em Proceedings of the 37th International Conference on Machine
  Learning, PMLR}, 119:1306--1316, 2020.

\bibitem{calder2014}
J.~Calder, S.~Esedo\=glu, and A.~O. Hero.
\newblock A {H}amilton-{J}acobi equation for the continuum limit of
  non-dominated sorting.
\newblock {\em SIAM Journal on Mathematical Analysis}, 46(1):603--638, 2014.

\bibitem{calder2015pde}
J.~Calder, S.~Esedo\=glu, and A.~O. Hero.
\newblock A {PDE}-based approach to nondominated sorting.
\newblock {\em SIAM Journal on Numerical Analysis}, 53(1):82--104, 2015.

\bibitem{calder2019improved}
J.~Calder and N.~Garc\'ia~Trillos.
\newblock Improved spectral convergence rates for graph {Laplacians} on
  $\varepsilon$-graphs and k-{NN} graphs.
\newblock {\em arXiv:1910.13476}, 2019.

\bibitem{calder2020Lip}
J.~Calder, N.~Garc\'ia~Trillos, and M.~Lewicka.
\newblock {Lipschitz regularity of graph Laplacians on random data clouds}.
\newblock {\em To appear in SIAM Journal on Mathematical Analysis}, 2021.

\bibitem{calder2021boundary}
J.~Calder, S.~Park, and D.~Slep\v{c}ev.
\newblock Boundary estimation from point clouds: {A}lgorithms, guarantees and
  applications.
\newblock {\em arXiv:2111.03217}, 2021.

\bibitem{calder2020rates}
J.~Calder, D.~Slep{\v{c}}ev, and M.~Thorpe.
\newblock Rates of convergence for {L}aplacian semi-supervised learning with
  low labeling rates.
\newblock {\em arXiv preprint arXiv:2006.02765}, 2020.

\bibitem{calder2019properly}
J.~Calder and D.~Slep\v{c}ev.
\newblock {Properly-weighted graph {L}aplacian for semi-supervised learning}.
\newblock {\em Applied Mathematics and Optimization: Special Issue on
  Optimization in Data Science}, 82:1111--1159, 2019.

\bibitem{calder2020convex}
J.~Calder and C.~K. Smart.
\newblock {The limit shape of convex hull peeling}.
\newblock {\em Duke Mathematical Journal}, 169(11):2079--2124, 2020.

\bibitem{capuzzo1990hamilton}
I.~Capuzzo-Dolcetta and P.-L. Lions.
\newblock Hamilton-{J}acobi equations with state constraints.
\newblock {\em Transactions of the American Mathematical Society},
  318(2):643--683, 1990.

\bibitem{carrizosa1996characterization}
E.~Carrizosa.
\newblock A characterization of halfspace depth.
\newblock {\em Journal of multivariate analysis}, 58(1):21--26, 1996.

\bibitem{chapelle2005semi}
O.~Chapelle and A.~Zien.
\newblock Semi-supervised classification by low density separation.
\newblock In {\em International workshop on artificial intelligence and
  statistics}, pages 57--64. PMLR, 2005.

\bibitem{chernozhukov2017monge}
V.~Chernozhukov, A.~Galichon, M.~Hallin, and M.~Henry.
\newblock Monge--kantorovich depth, quantiles, ranks and signs.
\newblock {\em The Annals of Statistics}, 45(1):223--256, 2017.

\bibitem{coifman2006diffusion}
R.~R. Coifman and S.~Lafon.
\newblock Diffusion maps.
\newblock {\em Applied and computational harmonic analysis}, 21(1):5--30, 2006.

\bibitem{cook2020nondom}
B.~Cook and J.~Calder.
\newblock {Rates of convergence for the continuum limit of nondominated
  sorting}.
\newblock {\em To appear in SIAM Journal on Mathematical Analysis}, 2021.

\bibitem{de2020depth}
P.~L. de~Micheaux, P.~Mozharovskyi, and M.~Vimond.
\newblock Depth for curve data and applications.
\newblock {\em Journal of the American Statistical Association}, pages 1--17,
  2020.

\bibitem{desquesnes2017nonmonotonic}
X.~Desquesnes and A.~Elmoataz.
\newblock Nonmonotonic front propagation on weighted graphs with applications
  in image processing and high-dimensional data classification.
\newblock {\em IEEE Journal of Selected Topics in Signal Processing},
  11(6):897--907, 2017.

\bibitem{desquesnes2013eikonal}
X.~Desquesnes, A.~Elmoataz, and O.~L{\'e}zoray.
\newblock Eikonal equation adaptation on weighted graphs: fast geometric
  diffusion process for local and non-local image and data processing.
\newblock {\em Journal of Mathematical Imaging and Vision}, 46(2):238--257,
  2013.

\bibitem{el2016asymptotic}
A.~El~Alaoui, X.~Cheng, A.~Ramdas, M.~J. Wainwright, and M.~I. Jordan.
\newblock Asymptotic behavior of $\ell_p$-based {L}aplacian regularization in
  semi-supervised learning.
\newblock In {\em Conference on Learning Theory}, pages 879--906, 2016.

\bibitem{fletcher2009geometric}
P.~T. Fletcher, S.~Venkatasubramanian, and S.~Joshi.
\newblock The geometric median on riemannian manifolds with application to
  robust atlas estimation.
\newblock {\em NeuroImage}, 45(1):S143--S152, 2009.

\bibitem{flores2022algorithms}
M.~Flores, J.~Calder, and G.~Lerman.
\newblock {Analysis and algorithms for Lp-based semi-supervised learning on
  graphs}.
\newblock {\em To appear in Applied and Computational Harmonic Analysis}, 2022.

\bibitem{garcia2020error}
N.~Garc{\'\i}a~Trillos, M.~Gerlach, M.~Hein, and D.~Slep{\v{c}}ev.
\newblock Error estimates for spectral convergence of the graph {L}aplacian on
  random geometric graphs toward the {L}aplace--{B}eltrami operator.
\newblock {\em Foundations of Computational Mathematics}, 20(4):827--887, 2020.

\bibitem{hoffmann2022spectral}
F.~Hoffmann, B.~Hosseini, A.~A. Oberai, and A.~M. Stuart.
\newblock Spectral analysis of weighted laplacians arising in data clustering.
\newblock {\em Applied and Computational Harmonic Analysis}, 56:189--249, 2022.

\bibitem{hwang2016shortest}
S.~J. Hwang, S.~B. Damelin, and A.~O. Hero~III.
\newblock Shortest path through random points.
\newblock {\em The Annals of Applied Probability}, 26(5):2791--2823, 2016.

\bibitem{jacobs2018auction}
M.~Jacobs, E.~Merkurjev, and S.~Esedoḡlu.
\newblock Auction dynamics: A volume constrained {MBO} scheme.
\newblock {\em Journal of Computational Physics}, 354:288--310, 2018.

\bibitem{kingma2013bayes}
D.~P. Kingma and M.~Welling.
\newblock Auto-encoding variational {B}ayes.
\newblock In {\em Proceedings of the 2nd International Conference on Learning
  Representations (ICLR)}, 2014.

\bibitem{krizhevsky2009learning}
A.~Krizhevsky, G.~Hinton, et~al.
\newblock Learning multiple layers of features from tiny images.
\newblock {\em Citeseer}, 2009.

\bibitem{kyng2015algorithms}
R.~Kyng, A.~Rao, S.~Sachdeva, and D.~A. Spielman.
\newblock Algorithms for {L}ipschitz learning on graphs.
\newblock In {\em Conference on Learning Theory}, pages 1190--1223, 2015.

\bibitem{lecun1998gradient}
Y.~LeCun, L.~Bottou, Y.~Bengio, and P.~Haffner.
\newblock Gradient-based learning applied to document recognition.
\newblock {\em Proceedings of the IEEE}, 86(11):2278--2324, 1998.

\bibitem{lewicka2020domains}
M.~Lewicka and Y.~Peres.
\newblock Which domains have two-sided supporting unit spheres at every
  boundary point?
\newblock {\em Expositiones Mathematicae}, 38(4):548--558, 2020.

\bibitem{liu1999multivariate}
R.~Y. Liu, J.~M. Parelius, and K.~Singh.
\newblock Multivariate analysis by data depth: descriptive statistics, graphics
  and inference,(with discussion and a rejoinder by liu and singh).
\newblock {\em The annals of statistics}, 27(3):783--858, 1999.

\bibitem{mai2018randomJMLR}
X.~Mai and R.~Couillet.
\newblock A random matrix analysis and improvement of semi-supervised learning
  for large dimensional data.
\newblock {\em The Journal of Machine Learning Research}, 19(1):3074--3100,
  2018.

\bibitem{mai2018random}
X.~Mai and R.~Couillet.
\newblock Random matrix-inspired improved semi-supervised learning on graphs.
\newblock In {\em International Conference on Machine Learning}, 2018.

\bibitem{manfredi2015nonlinear}
J.~J. Manfredi, A.~M. Oberman, and A.~P. Sviridov.
\newblock Nonlinear elliptic partial differential equations and p-harmonic
  functions on graphs.
\newblock {\em Differential Integral Equations}, 28(1-2):79--102, 2015.

\bibitem{miller2022graphbased}
K.~Miller, X.~Baca, J.~Mauro, J.~Setiadi, Z.~Shi, J.~Calder, and A.~Bertozzi.
\newblock Graph-based active learning for semi-supervised classification of
  {SAR} data.
\newblock {\em To appear in SPIE Defense and Commercial Sensing: Algorithms for
  Synthetic Aperture Radar Imagery XXIX}, 2022.

\bibitem{molina2021tukey}
M.~Molina-Fructuoso and R.~Murray.
\newblock Tukey depths and {H}amilton-{J}acobi differential equations.
\newblock {\em arXiv:2104.01648}, 2021.

\bibitem{molina2022eikonal}
M.~Molina-Fructuoso and R.~Murray.
\newblock Eikonal depth: an optimal control approach to statistical depths.
\newblock {\em arXiv:2201.05274}, 2022.

\bibitem{moscovichfast}
A.~Moscovich, A.~Jaffe, and B.~Nadler.
\newblock Fast semi-supervised regression: a geodesic nearest neighbor
  approach.
\newblock {\em Online}, 2016.
\newblock \url{https://mosco.github.io/geodesicknn/geodesic_knn.pdf}.

\bibitem{nadler2009semi}
B.~Nadler, N.~Srebro, and X.~Zhou.
\newblock Semi-supervised learning with the graph laplacian: The limit of
  infinite unlabelled data.
\newblock {\em Advances in neural information processing systems},
  22:1330--1338, 2009.

\bibitem{ng2002spectral}
A.~Y. Ng, M.~I. Jordan, and Y.~Weiss.
\newblock On spectral clustering: Analysis and an algorithm.
\newblock In {\em Advances in neural information processing systems}, pages
  849--856, 2002.

\bibitem{oberman2006convergent}
A.~M. Oberman.
\newblock Convergent difference schemes for degenerate elliptic and parabolic
  equations: Hamilton--jacobi equations and free boundary problems.
\newblock {\em SIAM Journal on Numerical Analysis}, 44(2):879--895, 2006.

\bibitem{penrose2003random}
M.~Penrose.
\newblock {\em Random geometric graphs}, volume~5.
\newblock OUP Oxford, 2003.

\bibitem{rozza2014novel}
A.~Rozza, M.~Manzo, and A.~Petrosino.
\newblock A novel graph-based fisher kernel method for semi-supervised
  learning.
\newblock In {\em 2014 22nd International Conference on Pattern Recognition},
  pages 3786--3791. IEEE, 2014.

\bibitem{sethian1996fast}
J.~A. Sethian.
\newblock A fast marching level set method for monotonically advancing fronts.
\newblock {\em Proceedings of the National Academy of Sciences},
  93(4):1591--1595, 1996.

\bibitem{shi2017weighted}
Z.~Shi, S.~Osher, and W.~Zhu.
\newblock Weighted nonlocal {L}aplacian on interpolation from sparse data.
\newblock {\em Journal of Scientific Computing}, 73(2-3):1164--1177, 2017.

\bibitem{slepcev2019analysis}
D.~Slep\v{c}ev and M.~Thorpe.
\newblock Analysis of p-{L}aplacian regularization in semisupervised learning.
\newblock {\em SIAM Journal on Mathematical Analysis}, 51(3):2085--2120, 2019.

\bibitem{small1997multidimensional}
C.~G. Small.
\newblock Multidimensional medians arising from geodesics on graphs.
\newblock {\em The Annals of Statistics}, pages 478--494, 1997.

\bibitem{tenenbaum2000global}
J.~B. Tenenbaum, V.~De~Silva, and J.~C. Langford.
\newblock A global geometric framework for nonlinear dimensionality reduction.
\newblock {\em science}, 290(5500):2319--2323, 2000.

\bibitem{tukey1975mathematics}
J.~W. Tukey.
\newblock Mathematics and the picturing of data.
\newblock In {\em Proceedings of the International Congress of Mathematicians,
  Vancouver, 1975}, volume~2, pages 523--531, 1975.

\bibitem{xiao2017fashion}
H.~Xiao, K.~Rasul, and R.~Vollgraf.
\newblock Fashion-{MNIST}: A novel image dataset for benchmarking machine
  learning algorithms.
\newblock {\em arXiv preprint arXiv:1708.07747}, 2017.

\bibitem{yang2021spagan}
Y.~Yang, X.~Wang, M.~Song, J.~Yuan, and D.~Tao.
\newblock Spagan: Shortest path graph attention network.
\newblock {\em arXiv preprint arXiv:2101.03464}, 2021.

\bibitem{yuan2020continuum}
A.~Yuan, J.~Calder, and B.~Osting.
\newblock {A continuum limit for the {PageRank} algorithm}.
\newblock {\em European Journal of Applied Mathematics}, 2021.

\bibitem{zhang2019aet}
L.~Zhang, G.-J. Qi, L.~Wang, and J.~Luo.
\newblock Aet vs. aed: Unsupervised representation learning by auto-encoding
  transformations rather than data.
\newblock In {\em Proceedings of the IEEE Conference on Computer Vision and
  Pattern Recognition}, pages 2547--2555, 2019.

\bibitem{zhou2011semi}
X.~Zhou and M.~Belkin.
\newblock Semi-supervised learning by higher order regularization.
\newblock In {\em Proceedings of the Fourteenth International Conference on
  Artificial Intelligence and Statistics}, pages 892--900, 2011.

\bibitem{zhu2003semi}
X.~Zhu, Z.~Ghahramani, and J.~D. Lafferty.
\newblock Semi-supervised learning using {G}aussian fields and harmonic
  functions.
\newblock In {\em Proceedings of the 20th International Conference on Machine
  learning (ICML-03)}, pages 912--919, 2003.

\end{thebibliography}
